\documentclass{amsart}

\usepackage[T1]{fontenc}
\usepackage{amsfonts}
\usepackage{amsfonts,amsmath,amssymb}
\usepackage{mathrsfs,mathtools,stmaryrd,wasysym}
\usepackage{enumerate,enumitem}
\usepackage{esint}
\usepackage{graphicx,psfrag}
\usepackage{hyperref}
\usepackage{stackengine}
\usepackage{yfonts}
\usepackage{color}
\usepackage[linesnumbered,algoruled,boxed]{algorithm2e}

\newtheorem{theorem}{Theorem}[section]
\newtheorem{lemma}[theorem]{Lemma}
\newtheorem{corollary}[theorem]{Corollary}
\newtheorem{proposition}[theorem]{Proposition}

\theoremstyle{definition}
\newtheorem{definition}[theorem]{Definition}

\theoremstyle{remark}
\newtheorem{remark}[theorem]{Remark}

\theoremstyle{assumption}
\newtheorem{assu}[theorem]{Assumption}

\numberwithin{equation}{section}

\begin{document}

\title[Error estimates for an unregularized optimal control problem]{Error estimates for an unregularized optimal control problem for the stationary Navier-Stokes equations}
\author{Francisco Fuica$^{\dagger}$}
\address{$^{\dagger}$Departamento de Matem\'atica y Ciencia de la Computaci\'on, Universidad de Santiago de Chile, Santiago, Chile.}
\email{francisco.fuica@usach.cl}
\thanks{The first author has been partially supported by ANID through FONDECYT grant 11260142.}

\author{Nicolai Jork$^{\ddagger}$}
\address{$^{\ddagger}$Department of Mathematics, Eberhard-Karls-Universit\"at T\"ubingen, D-72076 T\"ubingen, Germany.}
\email{nicolai.jork@uni-tuebingen.de}
\thanks{This work was funded by the Deutsche Forschungsgemeinschaft (German Research Foundation) – Project number 577735529.}



\date{}

\dedicatory{}

\begin{abstract}
We consider an unregularized optimal control problem subject to the steady-state Navier-Stokes equations. 
We derive the existence of optimal solutions and prove first- and second-order optimality conditions. 
To approximate solutions to the optimal control problem, we consider the variational discretization scheme.
We analyze convergence properties of the discretization and prove a priori error estimates for locally optimal controls that are nonsingular and which satisfy a growth condition which implies a bang-bang structure.  
We also propose a residual-type a posteriori error estimator that accounts for the discretization of the state and adjoint equations, and prove suitable reliability properties for such an error estimator.
\end{abstract}

\maketitle

\section{Introduction}

In this work, we are interested in the analysis and discretization of an affine optimal control problem governed by the Navier--Stokes equations; bilateral control constraints are also considered.
To make matters precise, let $\Omega\subset\mathbb{R}^n$, with $n\in \{2, 3\}$, be an open, bounded, and convex polygonal/polyhedral domain with boundary $\partial\Omega$. 
Given $\mathbf{u}\in \mathbf{L}^{2}(\Omega)$, we consider the stationary Navier--Stokes equations
\begin{align}\label{eq:state_equations}
-\nu\Delta \mathbf{y}+(\mathbf{y}\cdot\nabla)\mathbf{y}+\nabla p = \mathbf{u} \text{ in } \Omega, \quad \text{div }\mathbf{y}=0 \text{ in } \Omega,\quad \mathbf{y}=\boldsymbol 0  \text{ on } \partial\Omega.
\end{align}
Here, $\nu> 0$ denotes the kinematic viscosity.
We notice that problem \eqref{eq:state_equations} may have several solutions. 
Let us denote by $S(\mathbf{u})$ the set of (weak) solutions of \eqref{eq:state_equations} associated with $\mathbf{u}\in \mathbf{L}^{2}(\Omega)$.
Given a desired state $\mathbf{y}_\Omega \in \mathbf{L}^2(\Omega)$, we define the cost function
\begin{equation}\label{def:cost_functional}
\mathcal{J}(\mathbf{u}):=\inf\left\{\frac{1}{2}\|\mathbf{y}_{\mathbf{u}}-\mathbf{y}_\Omega\|_{\mathbf{L}^{2}(\Omega)}^2 ~ : ~ (\mathbf{y}_{\mathbf{u}}, p_{\mathbf{u}})\in S(\mathbf{u})\right\}.
\end{equation}
Note that the function $\mathcal{J}(\mathbf{u})$ reduces to $\frac{1}{2}\|\mathbf{y}_{\mathbf{u}}-\mathbf{y}_\Omega\|_{\mathbf{L}^{2}(\Omega)}^2$ when the problem \eqref{eq:state_equations} has a unique solution.
We shall be concerned with the following optimal control problem: 
\begin{align}\label{eq:minimize_cost_func}
\min_{\mathbf{u}\in \mathbf{U}_{ad}}{\mathcal{J}(\mathbf{u})} \quad \text{subject to} \quad \eqref{eq:state_equations} \quad \text{and}\quad \mathbf{u}\in \mathbf{U}_{ad},
\end{align}
where $
\mathbf{U}_{ad}:=\{ \mathbf{v}\in \mathbf{L}^2(\Omega): \mathbf{a} \leq \mathbf{v} \leq \mathbf{b} \text{ a.e. } \text{in } \Omega \}
$,
with $\mathbf{a}, \mathbf{b} \in  \mathbb{R}^n$ satisfying $\mathbf{a}_i < \mathbf{b}_i$ for every $i\in\{1,\ldots, n\}$.

In recent years, the study of affine optimal control problems has received considerable attention. Although its origin lies in optimal control problems
subject to ordinary differential equations (ODEs), this has rapidly evolved to problems subject to partial differential equations (PDEs). In this regard, the study of bang-bang optimal controls in the context of elliptic equations can be found in \cite{zbMATH06111099, CDJ2023, zbMATH06797168, zbMATH06982426, zbMATH07643482, zbMATH07865499, 2026arXiv260214632W} and the references therein.
For results on numerical approximation and error analysis of these problems, particularly without the use of regularization strategies, we refer the
reader to \cite{CM2021,MR2891922, Fuica2024,2025arXiv250504439F,MR4791221,zbMATH08071521}. 
We also mention the recent work \cite{2025arXiv250924829W}, in which the authors study the continuous differentiability of the signum function and its application to Newton's method to solve bang-bang optimal control problems numerically, proving that the signum function is Fréchet differentiable between suitable function spaces.

For problems governed by the Navier--Stokes equations, considerably fewer results are available. 
To our knowledge, the first work that studied a bang-bang optimal control problem subject to the instationary Navier-Stokes equations is \cite{zbMATH06811877}, where first- and second-order optimality conditions are derived together with a priori error estimates. 
The authors propose a fully discrete scheme combining discontinuous Galerkin methods in time with conforming finite elements in space, and establish convergence rates for the state variable in the $L^{2}$-norm under suitable discretization assumptions; error estimates for the control variable were not derived.

To the best of our knowledge, this is the first work to prove control error estimates for a variational discretization of an unregularized affine optimal control problem governed by the \emph{stationary} Navier–Stokes equations with pointwise control constraints.
Such unregularized formulations are known to induce bang–bang optimal controls and lead to significant analytical challenges, particularly due to the lack of additional smoothing typically provided by Tikhonov regularization.
Although the structure of optimal controls in this setting is fairly understood, the development of a rigorous numerical analysis for the underlying nonlinear PDE-constrained problem remains comparatively limited, especially in the case of nonlinear PDEs.
The main contribution of this work is a detailed error analysis for a variational discretization scheme applied to this class of problems. 
We establish a priori and a posteriori error estimates for locally optimal solutions under relatively mild regularity assumptions, including error bounds in the $\mathbf{L}^{\infty}(\Omega)$-norm for an auxiliary adjoint velocity field and in the $\mathbf{L}^{1}(\Omega)$-norm for the control variable. 
Our approach combines refined regularity results for the Navier–Stokes equations and their linearizations with stability properties of the optimality system, allowing us to handle the nonlinear coupling without imposing a global smallness condition on the data; instead, the analysis is carried out locally around a nonsingular solution, i.e., a solution for which the linearized Navier–Stokes operator is an isomorphism  (cf. \cite[Chapter IV, \S~3]{MR851383}).
In particular, we derive error estimates that reflect the intrinsic nonsmooth structure of bang–bang controls, without introducing any regularization parameter.
The contributions of this work are summarized as follows:

$\bullet$ \emph{Optimality conditions.} We derive the existence of solutions and, assuming that $\bar{\mathbf{u}}$ is a local nonsingular solution, we obtain first-order necessary optimality conditions and sufficient second-order optimality conditions. 

$\bullet$ \emph{Error estimates for the adjoint equation.} Assuming that $\Omega$ is a convex polytope and $\mathbf{y}_{\Omega}\in \mathbf{L}^{\mathsf{p}}(\Omega)$ with $\mathsf{p} > n$, in Proposition \ref{prop:max_norm_estimate_zz}, we prove a priori error estimates for an auxiliary adjoint velocity field in the maximum norm. 
In particular, our result is also nearly optimal with respect to regularity.
We also design a residual-type a posteriori error estimator for an auxiliary adjoint equation and prove, under suitable assumption on discrete solutions, its global reliability; see Lemma \ref{lemma:a_post_adj_hat_discrete}. 

$\bullet$ \emph{A priori error estimates.} To approximate optimal variables, we consider the variational discretization approach introduced in \cite{MR2122182}: we discretize the optimal state and adjoint state variables using the lowest order Taylor--Hood finite element spaces, while the control variable is not discretized.  Assuming that $\Omega$ is a convex polytope, $\mathbf{y}_\Omega \in \mathbf{L}^\mathsf{p}(\Omega)$ with $\mathsf{p} > n$, and the growth condition \eqref{sufficientcon}, we derive the error bound (see Theorem \ref{thm:estimate_control} and Corollary \ref{coro:error_state})
    \[
    \|\bar{\mathbf{u}} - \bar{\boldsymbol{u}}_{h}\|_{\mathbf{L}^{1}(\Omega)} + \|\bar{\mathbf{y}} - \bar{\mathbf{y}}_{h}\|_{\mathbf{L}^{2}(\Omega)} 
    \lesssim
    (h |\log h|)^{\gamma},
    \]
with $\gamma \in (n/(n + 2), 1]$. 
We note that the growth condition \eqref{sufficientcon} is a weaker form of optimality designed for optimal control problems where the controls are expected to exhibit a bang-bang structure, which is typical of unregularized affine problems. 
We recall that our approach does not involve a regularization parameter, thus avoiding the simultaneous treatment of discretization and regularization errors.

$\bullet$ \emph{A posteriori error estimates.} We propose a residual-type a posteriori error estimator that accounts for all the discrete optimal variables.
Assuming that $\Omega$ is a convex polytope, $\mathbf{y}_\Omega \in \mathbf{L}^\mathsf{p}(\Omega)$, and considering assumptions on discrete solutions similar to those utilized in \cite{MR4301394}, we prove a global reliability property for such an error estimator (Theorem \ref{thm:rel_ocp}):
    \[
    \|\bar{\mathbf{u}} - \bar{\boldsymbol{u}}_{h}\|_{\mathbf{L}^{1}(\Omega)}
    \lesssim
    (\eta_{\textrm{st},2} + \eta_{\textrm{st},\mathsf{p}} + |\log h_{\min}|^{4/n}\eta_{\textrm{adj},\infty} +  \|\textnormal{div }\bar{\mathbf{y}}_{\ell}\|_{\mathbf{L}^{\mu}(\Omega)})^{\gamma},
    \]
with $\gamma \in (n/(n + 2), 1]$. 

We organize the remainder of the manuscript as follows. 
Section~\ref{sec:not_and_prel} introduces the notation and collects preliminary results related to the PDEs involved in the analysis. 
In Section~\ref{sec:ocp}, we establish existence of solutions and derive first- and second-order optimality conditions. 
The core of our work is contained in Sections~\ref{sec:fem} and~\ref{sec:apost_OCP}. 
In Section~\ref{sec:fem}, we introduce a semidiscrete approximation scheme and derive a priori error estimates for strict locally optimal solutions satisfying a growth condition. 
Finally, in Section~\ref{sec:apost_OCP}, we develop a residual-type a posteriori error analysis for the associated optimality system and prove suitable reliability properties of the resulting error estimator.


\section{Notation and preliminary remarks}\label{sec:not_and_prel}
Let us establish the notation and describe the framework in which we will work.


\subsection{Notation}\label{sec:notation}

Throughout the manuscript, we adopt standard notation for Lebesgue and Sobolev spaces, as well as their associated norms.
Given an open and bounded domain $\omega$, we denote by $(\cdot,\cdot)_{\omega}$ and $\| \cdot \|_{\omega}$ the inner product and the norm of $L^{2}(\omega)$, respectively.
We denote by $L^2_0(\omega)$ the space of functions in $L^2(\omega)$ that have zero average on $\omega$. 
We use uppercase bold letters to denote the vector-valued counterparts of the aforementioned spaces, while lowercase bold letters are used to denote vector-valued functions. 
For $n\in \{2,3\}$, we abbreviate $ \mathbf{H}^1(\Omega)=H^1(\Omega;\mathbb R^n)$, $\mathbf{H}_0^1(\Omega)=H^1_0(\Omega;\mathbb R^n)$, $\mathbf{H}^{-1}(\Omega)=\mathbf{H}^1_0(\Omega)'$, and $\mathbf{W}^{s,p}(\Omega)=W^{\mathsf{s},\mathsf{p}}(\Omega;\mathbb R^n)$ for $\mathsf{s}>0$ and $1\leq \mathsf{p} \leq \infty$.
We also define the set $\mathbf{V}(\Omega):= \{ \mathbf{v} \in \mathbf{H}_{0}^{1}(\Omega): \text{div }\mathbf{v}  = 0\}$.

If $\mathfrak{X}$ and $\mathfrak{Y}$ are normed vector spaces, we write $\mathfrak{X}\hookrightarrow\mathfrak{Y}$ to denote that $\mathfrak{X}$ is continuously embedded in $\mathfrak{Y}$. 
We denote by $\mathfrak{X}'$ the dual of $\mathfrak{X}$ and by $\langle \cdot,\cdot \rangle_{\mathfrak{X}',\mathfrak{X}}$ the duality pairing between $\mathfrak{X}'$ and $\mathfrak{X}$.
When the spaces $\mathfrak{X}'$ and $\mathfrak{X}$ are clear from the context, we simply denote $\langle \cdot,\cdot \rangle_{\mathfrak{X}',\mathfrak{X}}$ by $\langle \cdot,\cdot \rangle$.
The relation $\mathfrak{a} \lesssim \mathfrak{b}$ indicates that $\mathfrak{a} \leq C \mathfrak{b}$, with a positive constant that depends neither on $\mathfrak{a}$, $\mathfrak{b}$ nor on the discretization parameters. 
The value of the constant $C$ can vary from one occurrence to another.


\subsection{Results on the involved PDEs}

This section collects properties of solutions to the Navier--Stokes and suitable \emph{linearized} Navier--Stokes equations.

\subsubsection{Navier--Stokes equations}
We begin this section by defining the trilinear form
$
b(\mathbf{v}_1;\mathbf{v}_2,\mathbf{v}_3):=((\mathbf{v}_1\cdot \nabla)\mathbf{v}_2,\mathbf{v}_3)_{\mathbf{L}^2(\Omega)}
$, which satisfies the following properties: If $\mathbf{v}_1\in \mathbf{V}(\Omega)$ and $\mathbf{v}_2, \mathbf{v}_3 \in \mathbf{H}_0^1(\Omega)$, then
\begin{equation}\label{eq:properties_trilinear}
b(\mathbf{v}_1;\mathbf{v}_2,\mathbf{v}_3)=-b(\mathbf{v}_1;\mathbf{v}_3,\mathbf{v}_2), 
\qquad  
b(\mathbf{v}_1;\mathbf{v}_2,\mathbf{v}_2)=0;
\end{equation}
see, e.g., \cite[Chapter IV, Lemma 2.2]{MR851383} and \cite[Chapter II, Lemma 1.3]{MR603444}.
Moreover, the form $b$ is well defined and continuous on $[\mathbf{H}_0^1(\Omega)]^3$ and satisfies, for some $\mathcal{C}_b>0$, the following estimate (\cite[Lemma IX.1.1]{MR2808162} and \cite[Chapter II, Lemma 1.1]{MR603444}):
\[
|b(\mathbf{v}_1;\mathbf{v}_2,\mathbf{v}_3)|\leq \mathcal{C}_b\|\nabla \mathbf{v}_1\|_{\mathbf{L}^2(\Omega)}\|\nabla \mathbf{v}_2\|_{\mathbf{L}^2(\Omega)}\|\nabla \mathbf{v}_3\|_{\mathbf{L}^2(\Omega)}.
\]

We now present a weak formulation for the Navier--Stokes system \cite[Chapter IV, eq. (2.8)]{MR851383}: Given $\mathbf{f} \in \mathbf{H}^{-1}(\Omega)$, find $(\mathbf{y},p) \in \mathbf{V}(\Omega)\times L_0^2(\Omega)$ such that 
\begin{equation}\label{eq:weak_ns_divergence_free}
\nu(\nabla \mathbf{y}, \nabla \mathbf{v})_{\mathbf{L}^2(\Omega)}+b(\mathbf{y};\mathbf{y},\mathbf{v})-(p,\text{div } \mathbf{v})_{\Omega}  = \langle \mathbf{f},\mathbf{v} \rangle  ~~~  \forall \mathbf{v} \in \mathbf{H}_0^1(\Omega).
\end{equation}
The following result states the existence of solutions to problem \eqref{eq:weak_ns_divergence_free}. 
For its proof, we refer, e.g., to \cite[Chapter IV, Theorem 2.1]{MR851383} and \cite[Chapter II, Theorem 1.2]{MR603444}.
 
\begin{theorem}[existence of solutions]\label{thm:well_posedness_navier_stokes}
Let $\nu>0$ and $\mathbf{f} \in \mathbf{H}^{-1}(\Omega)$. Then, problem \eqref{eq:weak_ns_divergence_free} admits at least one solution $(\mathbf{y},p) \in \mathbf{V}(\Omega)\times L_0^2(\Omega)$ that satisfies the stability estimate $\|\nabla \mathbf{y}\|_{\mathbf{L}^2(\Omega)} + \|p\|_{\Omega} \lesssim \|\mathbf{f}\|_{\mathbf{H}^{-1}(\Omega)}$.
\end{theorem}

\begin{remark}[equivalent formulation]\label{rmk:equivalent}
We note that problem \eqref{eq:weak_ns_divergence_free} can be reformulated as follows: Find $(\mathbf{y},p)\in\mathbf{H}_0^1(\Omega)\times L_0^2(\Omega)$ such that 
\begin{equation}\label{eq:weak_navier_stokes_eq}
\nu(\nabla \mathbf{y}, \nabla \mathbf{v})_{\mathbf{L}^2(\Omega)}+b(\mathbf{y};\mathbf{y},\mathbf{v})-(p,\text{div } \mathbf{v})_{\Omega} \\
= \langle \mathbf{f},\mathbf{v} \rangle, \qquad (q,\text{div } \mathbf{y})_{\Omega}  =   0,
\end{equation}
for all $(\mathbf{v},q) \in \mathbf{H}_0^1(\Omega)\times L^2_0(\Omega)$, see, for example, \cite[Chapter IV, Section 2.1]{MR851383}.
\end{remark}

Since $\Omega$ is a convex polygon/polyhedron, if we assume that $\mathbf{f}\in \mathbf{L}^2(\Omega)$, then any solution $(\mathbf{y},p)\in \mathbf{H}_0^{1}(\Omega)\times L_0^{2}(\Omega)$ to \eqref{eq:weak_navier_stokes_eq} satisfies $(\mathbf{y},p) \in \mathbf{H}^2(\Omega) \times H^1(\Omega)$. We refer to \cite[Corollary 7.3.3.5]{MR775683} for $n=2$ and to \cite[Theorem 11.3.1]{MR2641539} for case $n=3$.


\subsubsection{Auxiliary equations}

We begin this section by introducing the concept of \emph{regular solution} to the Navier--Stokes equations (see, e.g., \cite[Definition 2.3]{MR2338434}).
\begin{definition}[regular solution]\label{def:reg_sol}
Let $(\mathbf{y},p)$ be a solution to \eqref{eq:state_equations} associated with some $\mathbf{u} \in \mathbf{U}_{ad}$, i.e., $(\mathbf{y},p)\in S(\mathbf{u})$. 
We say that $\mathbf{y}$ is regular if for every $\mathbf{g}\in\mathbf{H}^{-1}(\Omega)$ the problem: Find $(\boldsymbol{\varphi},\zeta)\in \mathbf{H}_0^1(\Omega)\times L_0^2(\Omega)$ such that
\begin{align}\label{eq:first_deriv_S*}
\begin{split}
\nu(\nabla \boldsymbol{\varphi}, \nabla \mathbf{v})_{\mathbf{L}^2(\Omega)} + b(\mathbf{y};\boldsymbol{\varphi},\mathbf{v}) + b(\boldsymbol{\varphi};\mathbf{y},\mathbf{v}) - (\zeta, \textnormal{div }\mathbf{v})_{\Omega} &\, = \langle \mathbf{g},\mathbf{v} \rangle \\
(q,\textnormal{div }\boldsymbol{\varphi})_{\Omega} &\, = 0
\end{split}
\end{align}
for all $(\mathbf{v},q)\in\mathbf{H}_0^1(\Omega)\times L_0^2(\Omega)$, is well posed.
\label{def:regular_solution}
\end{definition}

It can be proved, based on the standard inf-sup theory for saddle point problems, that $\mathbf{y}$ is regular if the kinematic viscosity $\nu$ is large enough; see, e.g., \cite[Remark 3.3]{MR3936891}.

\begin{proposition}[auxiliary estimate: linearized state I]\label{prop:aux_linear_states}
Let $\mathbf{u},\mathbf{u}_{\theta}\in \mathbf{U}_{ad}$ and $(\mathbf{y},p)\in S(\mathbf{u})$ and $(\mathbf{y}_{\theta},p_{\theta})\in S(\mathbf{u}_{\theta})$.
Assume that both $\mathbf{y}$ and $\mathbf{y}_{\theta}$ are regular solutions in the sense of Definition~\ref{def:reg_sol}. Let $\mathbf{g}\in \mathbf{H}^{-1}(\Omega)$ and denote by $(\boldsymbol{\varphi},\zeta)$ and $(\boldsymbol{\varphi}_{\theta},\zeta_{\theta})$ the unique solutions to \eqref{eq:first_deriv_S*} associated with $\mathbf{y}$ and
$\mathbf{y}_{\theta}$, respectively.
Then, we have
\begin{align*}
    \|\boldsymbol{\varphi}_{\theta} - \boldsymbol{\varphi}\|_{\mathbf{L}^{2}(\Omega)} 
    \lesssim
    \|\nabla(\mathbf{y} - \mathbf{y}_{\theta})\|_{\mathbf{L}^{\mathsf{p}}(\Omega)}\|\boldsymbol{\varphi}\|_{\mathbf{L}^{2}(\Omega)},
\end{align*}
with $\mathsf{p} > n$ arbitrarily close to $n$.
\end{proposition}
\begin{proof}
Observe that $(\boldsymbol \varphi_{\theta} - \boldsymbol \varphi, \zeta_{\theta} - \zeta) \in \mathbf{H}_0^{1}(\Omega)\times L_0^{2}(\Omega)$ solves 
\begin{align*}
&\, \nu(\nabla (\boldsymbol \varphi_{\theta} - \boldsymbol \varphi), \nabla \mathbf{v})_{\mathbf{L}^2(\Omega)} + b(\mathbf{y}_{\theta};\boldsymbol \varphi_{\theta} - \boldsymbol \varphi,\mathbf{v}) + b(\boldsymbol \varphi_{\theta} - \boldsymbol \varphi;\mathbf{y}_{\theta},\mathbf{v}) - (\zeta_{\theta} - \zeta, \textnormal{div }\mathbf{v})_{\Omega}\\
&\, = b(\mathbf{y} - \mathbf{y}_{\theta}; \boldsymbol \varphi, \mathbf{v}) + b(\boldsymbol \varphi; \mathbf{y} - \mathbf{y}_{\theta},\mathbf{v}), \qquad
(q,\textnormal{div}(\boldsymbol \varphi_{\theta} - \boldsymbol \varphi))_{\Omega} = 0, 
\end{align*}
for all $(\mathbf{v},q)\in \mathbf{H}_0^{1}(\Omega)\times L_0^{2}(\Omega)$.
Since the form $b$ is well defined on $[\mathbf{H}_0^1(\Omega)]^3$, the right-hand side of this problem belongs to $\mathbf{H}^{-1}(\Omega)$.
Thus, since $\mathbf{y}_{\theta}$ is regular, we find that this problem is well posed and it holds that
\begin{align*}
    \|\nabla (\boldsymbol{\varphi}_{\theta} - \boldsymbol{\varphi})\|_{\mathbf{L}^{2}(\Omega)}
    \lesssim
    \| b(\mathbf{y} - \mathbf{y}_{\theta};\boldsymbol \varphi,\cdot) \|_{\mathbf{H}^{-1}(\Omega)} + \| b(\boldsymbol \varphi;\mathbf{y} - \mathbf{y}_{\theta},\cdot) \|_{\mathbf{H}^{-1}(\Omega)}.
\end{align*}
Now, using \eqref{eq:properties_trilinear}, we have
\begin{align*}
\| b(\mathbf{y} - \mathbf{y}_{\theta};\boldsymbol \varphi,\cdot) \|_{\mathbf{H}^{-1}(\Omega)} 
& \leq 
\sup_{\mathbf{v} \in \mathbf{H}_0^{1}(\Omega) } 
\frac{ \|\mathbf{y} - \mathbf{y}_{\theta}\|_{\mathbf{L}^{\infty}(\Omega)} \|  \nabla \mathbf{v} \|_{\mathbf{L}^{2}(\Omega)}\|\boldsymbol \varphi\|_{\mathbf{L}^{2}(\Omega)}}{\| \nabla \mathbf{v} \|_{\mathbf{L}^{2} (\Omega)}}\\
&\, =
\| \mathbf{y} - \mathbf{y}_{\theta} \|_{\mathbf{L}^{\infty}(\Omega)} \|\boldsymbol{\varphi}\|_{\mathbf{L}^{2}(\Omega)}.
\end{align*}
Moreover, using the embedding $\mathbf{H}_0^{1}(\Omega) \hookrightarrow \mathbf{L}^{\mathsf{s}}(\Omega)$ with $\mathsf{s} < \infty$ when $n=2$ and $\mathsf{s} \leq 6$ when $n=3$, we get $\| b(\boldsymbol \varphi;\mathbf{y} - \mathbf{y}_{\theta},\cdot) \|_{\mathbf{H}^{-1}(\Omega)}
\lesssim \|\boldsymbol{\varphi} \|_{\mathbf{L}^{2}(\Omega)} \| \nabla  (\mathbf{y} - \mathbf{y}_{\theta})\|_{\mathbf{L}^{\mathsf{p}}(\Omega)}$ ($\mathsf{p} > n$).
Therefore, we have that 
\[
    \|\nabla (\boldsymbol{\varphi}_{\theta} - \boldsymbol{\varphi})\|_{\mathbf{L}^{2}(\Omega)}
    \lesssim
    (\| \mathbf{y} - \mathbf{y}_{\theta} \|_{\mathbf{L}^{\infty}(\Omega)} + \| \nabla  (\mathbf{y} - \mathbf{y}_{\theta})\|_{\mathbf{L}^{\mathsf{p}}(\Omega)})\|\boldsymbol{\varphi} \|_{\mathbf{L}^{2}(\Omega)}.
\]
We obtain the desired bound by combining the last estimate with Poincaré’s inequality and with the embedding $\mathbf{W}_0^{1,\mathsf{p}}(\Omega) \hookrightarrow \mathbf{L}^{\infty}(\Omega)$ for $\mathsf{p}>n$ arbitrarily close to $n$.
\end{proof}

Let $(\mathbf{y},p)$ be a (weak) solution to \eqref{eq:state_equations} associated with some $\mathbf{u} \in \mathbf{U}_{ad}$.
Given $\mathbf{h}\in \mathbf{H}^{-1}(\Omega)$, we introduce the pair $(\mathbf{z},\pi)\in \mathbf{H}_0^{1}(\Omega)\times L_0^{2}(\Omega)$ as the solution to
\begin{equation}\label{eq:adj_eq}
-\nu \Delta \mathbf{z} - (\mathbf{y}\cdot \nabla) \mathbf{z} + (\nabla \mathbf{y})^T \mathbf{z} + \nabla\pi  = \mathbf{h} \text{ in } \Omega, \quad
        \text{div }\mathbf{z} = 0 \text{ in } \Omega, \quad \mathbf{z} = \mathbf{0} \text{ on } \partial\Omega.
\end{equation}
We consider the variational formulation: Find $(\mathbf{z},\pi)\in \mathbf{H}_0^{1}(\Omega)\times L_0^{2}(\Omega)$ such that
\begin{align}\label{eq:adj_eq_var_form}
\begin{split}
\nu(\nabla \mathbf{w}, \nabla \mathbf{z})_{\mathbf{L}^2(\Omega)} + b(\mathbf{y}_;\mathbf{w},\mathbf{z}) + b(\mathbf{w};\mathbf{y},\mathbf{z}) - (\pi,\text{div } \mathbf{w})_{\Omega} &\, = \langle \mathbf{h},\mathbf{w}\rangle \\
\qquad
(s,\text{div } \mathbf{z})_{\Omega} & \,=   0
\end{split}
\end{align}
for all $(\mathbf{w},s)\in\mathbf{H}_0^1(\Omega)\times L_0^2(\Omega)$.

If we assume that $\mathbf{y}$ is regular, then problem \eqref{eq:adj_eq} (equivalently \eqref{eq:adj_eq_var_form}) is well posed; see Step 1 of the proof of Theorem 2.9 in \cite{MR3936891} for details. 
Moreover, since $\mathbf{y}\in \mathbf{H}^{2}(\Omega)$, the terms $(\mathbf{y}\cdot \nabla) \mathbf{z}$ and $(\nabla \mathbf{y})^T \mathbf{z}$ are bounded in $\mathbf{L}^{2}(\Omega)$:
\begin{align*}
    \|(\nabla \mathbf{y})^T \mathbf{z}\|_{\mathbf{L}^{2}(\Omega)} + \|(\mathbf{y}\cdot \nabla) \mathbf{z}\|_{\mathbf{L}^{2}(\Omega)}
    \lesssim &\, \|\nabla \mathbf{y}\|_{\mathbf{L}^{3}(\Omega)}\|\mathbf{z}\|_{\mathbf{L}^{6}(\Omega)} + \|\mathbf{y}\|_{\mathbf{L}^{\infty}(\Omega)}\|\nabla \mathbf{z}\|_{\mathbf{L}^{2}(\Omega)} \\
    \lesssim &\,\|\mathbf{y}\|_{\mathbf{H}^{2}(\Omega)}\|\nabla \mathbf{z}\|_{\mathbf{L}^{2}(\Omega)}.
\end{align*}
Hence, when $\mathbf{h}\in \mathbf{L}^{2}(\Omega)$, we can use regularity results for the Stokes equations to conclude that $(\mathbf{z},\pi)\in \mathbf{H}^{2}(\Omega)\times H^{1}(\Omega)$ and $\|\mathbf{z}\|_{\mathbf{H}^{2}(\Omega)} \lesssim \|\mathbf{h}\|_{\mathbf{L}^{2}(\Omega)}$.
Even more, given $\mathsf{p}>n$ arbitrarily close to $n$, we observe that
\begin{align*}
    \|(\nabla \mathbf{y})^T \mathbf{z}\|_{\mathbf{L}^{\mathsf{p}}(\Omega)} + \|(\mathbf{y}\cdot \nabla) \mathbf{z}\|_{\mathbf{L}^{\mathsf{p}}(\Omega)}
    \lesssim &\, \|\nabla \mathbf{y}\|_{\mathbf{L}^{\mathsf{p}}(\Omega)}\|\mathbf{z}\|_{\mathbf{L}^{\infty}(\Omega)} + \|\mathbf{y}\|_{\mathbf{L}^{\infty}(\Omega)}\|\nabla \mathbf{z}\|_{\mathbf{L}^{\mathsf{p}}(\Omega)} \\
    \lesssim &\,\|\nabla \mathbf{y}\|_{\mathbf{L}^{\mathsf{p}}(\Omega)}\|\nabla \mathbf{z}\|_{\mathbf{L}^{\mathsf{p}}(\Omega)} \lesssim \|\nabla \mathbf{z}\|_{\mathbf{L}^{\mathsf{p}}(\Omega)}.
\end{align*}
Consequently, if $\mathbf{h}\in \mathbf{L}^{\mathsf{p}}(\Omega)$, we can apply the result of \cite[Lemma 14]{MR3008832} (see also \cite[Section 1.3]{MR3422453}) to conclude that there exists $\alpha\in (0,1)$ (depending on the maximum edge-opening angle of $\Omega$) such that $(\mathbf{z},\pi)\in \mathbf{C}^{1,\alpha}(\bar{\Omega})\times C^{0,\alpha}(\bar{\Omega})$ and 
\begin{align}\label{eq:estimate_z_max_norm}
    \|\nabla \mathbf{z}\|_{\mathbf{L}^\infty(\Omega)} 
    \lesssim 
    \|\mathbf{h}\|_{\mathbf{L}^{\mathsf{p}}} + \|\nabla \mathbf{z}\|_{\mathbf{L}^{\mathsf{p}}(\Omega)}.
\end{align}
With this estimate at hand, we prove the following result.

\begin{lemma}[auxiliary property: adjoint state]\label{lemma:aux_prop_z1-z2}
Let $\mathbf{u},\mathbf{u}_{\theta}\in \mathbf{U}_{ad}$ and $(\mathbf{y},p)\in S(\mathbf{u})$ and $(\mathbf{y}_{\theta},p_{\theta})\in S(\mathbf{u}_{\theta})$.
Assume that both $\mathbf{y}$ and $\mathbf{y}_{\theta}$ are regular solutions in the sense of Definition~\ref{def:reg_sol}.
Let $\mathbf{h}, \mathbf{h}_{\theta} \in \mathbf{L}^{\mathsf{p}}(\Omega)$ with $\mathsf{p} > n$ arbitrarily close to $n$ and denote by $(\mathbf{z},\pi)$ the unique solution to \eqref{eq:adj_eq} associated with $\mathbf{y}$ and $\mathbf{h}$, and by $(\mathbf{z}_{\theta}, \pi_{\theta})$ the unique solution to \eqref{eq:adj_eq} associated with $\mathbf{y}_{\theta}$ and $\mathbf{h}_{\theta}$. 
Then, we have
\[
    \|\nabla(\mathbf{z}_{\theta} - \mathbf{z})\|_{\mathbf{L}^{\infty}(\Omega)} 
    \lesssim
    \|\mathbf{h}_{\theta} - \mathbf{h}\|_{\mathbf{L}^{\mathsf{p}}(\Omega)} + \|\nabla (\mathbf{y}_{\theta} - \mathbf{y})\|_{\mathbf{L}^{\mathsf{p}}(\Omega)}.
\]
\end{lemma}
\begin{proof}
The pair $(\mathbf{z}_{\theta} - \mathbf{z},\pi_{\theta} - \pi)$ solves
    \begin{align}\label{eq:z_t-z-problem}
    \begin{split}
       &-\nu \Delta (\mathbf{z}_{\theta} - \mathbf{z}) + (\nabla \mathbf{y}_{\theta})^T (\mathbf{z}_{\theta} - \mathbf{z}) - (\mathbf{y}_{\theta}\cdot \nabla) (\mathbf{z}_{\theta} - \mathbf{z}) + \nabla(\pi - \pi_{\theta}) =  (\mathbf{h} - \mathbf{h}_{\theta}) \\
       & + (\nabla (\mathbf{y} - \mathbf{y}_{\theta}))^T \mathbf{z} - ((\mathbf{y} - \mathbf{y}_{\theta})\cdot \nabla) \mathbf{z} ~\text{ in } \Omega, 
       \quad
        \text{div}(\mathbf{z}_{\theta} - \mathbf{z}) = 0 ~\text{ in }\Omega, 
        \end{split}
\end{align}
with Dirichlet boundary condition $\mathbf{z}_{\theta} - \mathbf{z}  = \mathbf{0} \text{ on } \partial\Omega$. 
On the one hand, following the arguments that lead to \eqref{eq:estimate_z_max_norm} we find that 
\begin{align}\label{eq:estimate_max_z-z_t_I}
    \|\nabla(\mathbf{z}_{\theta} - \mathbf{z})\|_{\mathbf{L}^{\infty}(\Omega)} 
    \lesssim
    \|\mathbf{h}_{\theta} - \mathbf{h}\|_{\mathbf{L}^{\mathsf{p}}(\Omega)} + \|\nabla (\mathbf{y}_{\theta} - \mathbf{y})\|_{\mathbf{L}^{\mathsf{p}}(\Omega)} + \|\nabla(\mathbf{z}_{\theta} - \mathbf{z})\|_{\mathbf{L}^{\mathsf{p}}(\Omega)},
\end{align}
upon using that $\|\nabla \mathbf{y}_{\theta}\|_{\mathbf{L}^{\mathsf{p}}(\Omega)},\|\nabla \mathbf{z}\|_{\mathbf{L}^{\mathsf{p}}(\Omega)} \leq C$ ($C>0$).
The use of regularity results for the Stokes equations \cite[Theorem 2.9]{Brown_Shen} (see also \cite[eq. (1.52)]{MR2987056}) yields
\begin{align*}
    &\|\nabla(\mathbf{z}_{\theta} - \mathbf{z})\|_{\mathbf{L}^{\mathsf{p}}(\Omega)} 
    \lesssim 
    \|(\mathbf{y}_{\theta}\cdot \nabla) (\mathbf{z}_{\theta} - \mathbf{z}) -  (\nabla \mathbf{y}_{\theta})^T (\mathbf{z}_{\theta} - \mathbf{z})\|_{\mathbf{W}^{-1,\mathsf{p}}(\Omega)} \\
    &\, \qquad \qquad \qquad \qquad \, + \|(\mathbf{h} - \mathbf{h}_{\theta}) + (\nabla (\mathbf{y} - \mathbf{y}_{\theta}))^T \mathbf{z} - ((\mathbf{y} - \mathbf{y}_{\theta})\cdot \nabla)\mathbf{z}\|_{\mathbf{W}^{-1,\mathsf{p}}(\Omega)}\\
   & \lesssim 
   \|\nabla \mathbf{y}_{\theta}\|_{\mathbf{L}^{\mathsf{p}}(\Omega)} \|\nabla (\mathbf{z}_{\theta} - \mathbf{z})\|_{\mathbf{L}^{2}(\Omega)} + \|\mathbf{h}_{\theta} - \mathbf{h}\|_{\mathbf{L}^{\mathsf{p}}(\Omega)} + \|\nabla (\mathbf{y}_{\theta} - \mathbf{y})\|_{\mathbf{L}^{\mathsf{p}}(\Omega)}.
\end{align*}
This, in combination with the stability estimate $\|\nabla (\mathbf{z}_{\theta} - \mathbf{z})\|_{\mathbf{L}^{2}(\Omega)}\lesssim \|(\mathbf{h} - \mathbf{h}_{\theta}) + (\nabla (\mathbf{y} - \mathbf{y}_{\theta}))^T \mathbf{z} - ((\mathbf{y} - \mathbf{y}_{\theta})\cdot \nabla)\mathbf{z}\|_{\mathbf{H}^{-1}(\Omega)}$, which follows from the well posedness of problem \eqref{eq:z_t-z-problem}, allows us to obtain $\|\nabla(\mathbf{z}_{\theta} - \mathbf{z})\|_{\mathbf{L}^{\mathsf{p}}(\Omega)}  \lesssim \|\mathbf{h}_{\theta} - \mathbf{h}\|_{\mathbf{L}^{\mathsf{p}}(\Omega)} + \|\nabla (\mathbf{y}_{\theta} - \mathbf{y})\|_{\mathbf{L}^{\mathsf{p}}(\Omega)}$.
Using the latter in \eqref{eq:estimate_max_z-z_t_I} concludes the proof.
\end{proof}

\begin{proposition}[auxiliary estimate: linearized state II]\label{prop:aux_linear_states_rhsL1}
Let $\mathbf{u}\in \mathbf{U}_{ad}$ and $(\mathbf{y},p)\in S(\mathbf{u})$.
Assume that $\mathbf{y}$ is a regular solution in the sense of Definition~\ref{def:reg_sol}.
Let $\mathbf{g}\in \mathbf{L}^{1}(\Omega)\cap \mathbf{H}^{-1}(\Omega)$ and denote by $(\boldsymbol{\varphi},\zeta)$ the unique solution of the linearized problem \eqref{eq:first_deriv_S*}.
Then, the following estimate holds:
\[
    \|\boldsymbol{\varphi}\|_{\mathbf{L}^{2}(\Omega)} 
    \lesssim
    \|\mathbf{g}\|_{\mathbf{L}^{1}(\Omega)}.
\]
\end{proposition}
\begin{proof}
    Denote by $(\mathbf{z}_{\boldsymbol{\varphi}},\pi_{\boldsymbol{\varphi}})\in \mathbf{H}_0^{1}(\Omega)\times L_0^{2}(\Omega)$ the unique solution to \eqref{eq:adj_eq_var_form} with $\mathbf{h} = \boldsymbol{\varphi}$. Choose $(\mathbf{w},s) = (\boldsymbol{\varphi},0)$ in the problem that $(\mathbf{z}_{\boldsymbol{\varphi}},\pi_{\boldsymbol{\varphi}})$ solves and take $(\mathbf{v},q) = (\mathbf{z}_{\boldsymbol{\varphi}},0)$ in \eqref{eq:first_deriv_S*}. 
    Thus, we obtain $\|\boldsymbol{\varphi}\|_{\mathbf{L}^{2}(\Omega)}^{2} = (\mathbf{g},\mathbf{z}_{\boldsymbol{\varphi}})_{\mathbf{L}^{2}(\Omega)}$. 
    Using H\"older's inequality, the embedding $\mathbf{H}^{2}(\Omega)\hookrightarrow \mathbf{C}(\bar{\Omega})$, and the estimate $\|\mathbf{z}_{\boldsymbol{\varphi}}\|_{\mathbf{H}^{2}(\Omega)} \lesssim \|\boldsymbol{\varphi}\|_{\mathbf{L}^{2}(\Omega)}$ we get
\[
        \|\boldsymbol{\varphi}\|_{\mathbf{L}^{2}(\Omega)}^{2}
        \leq \|\mathbf{g}\|_{\mathbf{L}^{1}(\Omega)}\|\mathbf{z}_{\boldsymbol{\varphi}}\|_{\mathbf{L}^{\infty}(\Omega)}
        \lesssim
        \|\mathbf{g}\|_{\mathbf{L}^{1}(\Omega)}\|\mathbf{z}_{\boldsymbol{\varphi}}\|_{\mathbf{H}^{2}(\Omega)}
        \lesssim \|\mathbf{g}\|_{\mathbf{L}^{1}(\Omega)}\|\boldsymbol{\varphi}\|_{\mathbf{L}^{2}(\Omega)},
\]
    which concludes the proof.
\end{proof}


\section{The optimal control problem}\label{sec:ocp}
In this section, the existence of a global solution to the optimal control problem \eqref{eq:minimize_cost_func} is presented and first- and second-order necessary and sufficient optimality conditions are shown.
\begin{theorem}[existence of solutions]
\label{thm:exsol}
    There exists at least one solution of \eqref{eq:minimize_cost_func}.
\end{theorem}
\begin{proof}
The proof is standard; we include the details for completeness.

Let $\{\mathbf{u}_{k}\}_{k\in\mathbb{N}}$ be a minimizing sequence, i.e., $\mathcal{J}(\mathbf{u}_{k})\to \mathfrak{j} := \inf_{\mathbf{u}\in \mathbf{U}_{ad}}\mathcal{J}(\mathbf{u})$ as $k\to \infty$. 
Since $\mathbf{U}_{ad}$ is a nonempty, closed, convex and bounded subset of $\mathbf{L}^{2}(\Omega)$, there exists a subsequence (denoted in the same way) such that $\mathbf{u}_{k} \rightharpoonup \bar{\mathbf{u}}$ in $\mathbf{L}^{2}(\Omega)$ with $\bar{\mathbf{u}}\in \mathbf{U}_{ad}$.
Associated with each $\mathbf{u}_{k}$, we chose $(\mathbf{y}_{k},p_{k})\in S(\mathbf{u}_{k})$ such that 
$\frac{1}{2}\|\mathbf{y}_{k} - \mathbf{y}_\Omega\|_{\mathbf{L}^{2}(\Omega)}^{2} \leq \mathcal{J}(\mathbf{u}_{k}) + \frac{1}{k}$.
Theorem \ref{thm:well_posedness_navier_stokes} implies that, for every $k\in \mathbb{N}$, $\|\nabla \mathbf{y}_{k}\|_{\mathbf{L}^{2}(\Omega)} + \|p_{k}\|_{\Omega} \leq C$ ($C > 0$), and thus there exists a subsequence (still denoted in the same way) such that $(\mathbf{y}_{k},p_{k})\rightharpoonup (\bar{\mathbf{y}},\bar{p})$ in $\mathbf{H}_0^{1}(\Omega)\times L_0^{2}(\Omega)$.
The use of the embedding $\mathbf{H}_0^{1}(\Omega) \hookrightarrow \mathbf{L}^{4}(\Omega)$ yields $(\bar{\mathbf{y}},\bar{p})\in S(\bar{\mathbf{u}})$.
Using this and the embedding $\mathbf{H}_0^{1}(\Omega) \hookrightarrow \mathbf{L}^{2}(\Omega)$ we obtain
\[
    \mathcal{J}(\bar{\mathbf{u}}) \leq \frac{1}{2}\|\bar{\mathbf{y}} - \mathbf{y}_\Omega\|_{\mathbf{L}^{2}(\Omega)}^{2}
    = \lim_{k\to \infty} \frac{1}{2}\|\mathbf{y}_{k} - \mathbf{y}_\Omega\|_{\mathbf{L}^{2}(\Omega)}^{2}
    \leq \mathfrak{j},
\]
which concludes the proof.
\end{proof}

Since problem \eqref{eq:minimize_cost_func} is not convex, we also need to consider local solutions in the sense of $\mathbf{L}^1(\Omega)$ \cite[Definition 3.1]{MR2338434}, \cite[\S~4.4.2]{MR2583281}.

\begin{definition}[local solution]\label{def:local_sol}
We say that $\bar{\mathbf{u}}$ is a local minimum of \eqref{eq:minimize_cost_func} in the $\mathbf{L}^{1}(\Omega)$ sense, if there exists $\varepsilon > 0$ such that
\[
\mathcal{J}(\bar{\mathbf{u}}) \leq \mathcal{J}(\mathbf{u}) \quad \forall\,\mathbf{u}\in \mathbf{U}_{ad} \text{ with }\|\mathbf{u} - \bar{\mathbf{u}}\|_{\mathbf{L}^{1}(\Omega)} < \varepsilon.
\]
If the inequality is strict for every $\mathbf{u}\in \mathbf{U}_{ad}\setminus \{\bar{\mathbf{u}}\}$ such that $\|\mathbf{u} - \bar{\mathbf{u}}\|_{\mathbf{L}^{1}(\Omega)} < \varepsilon$, we say that $\bar{\mathbf{u}}$ is a \emph{strict} local minimum.
\end{definition}

Note that, since the admissible set $\mathbf{U}_{ad}$ is bounded in $\mathbf{L}^{\infty}(\Omega)$, if $\bar{\mathbf{u}}$ is a (strict) local minimum of problem \eqref{eq:minimize_cost_func} in the sense of $\mathbf{L}^{1}(\Omega)$, then $\bar{\mathbf{u}}$ is a (strict) local minimum of \eqref{eq:minimize_cost_func} in the sense of $\mathbf{L}^{t}(\Omega)$ with $1 \leq t \leq \infty$.

\begin{theorem}[differentiability]
\label{thm:properties_C_to_S}
Assume that $\bar{\mathbf{u}}$ is a local solution of \eqref{eq:minimize_cost_func} such that $\bar{\mathbf{y}}$, with $(\bar{\mathbf{y}},\bar{p})\in S(\bar{{\mathbf{u}}})$, is regular. Then there is an open neighborhood $\mathcal{O}(\bar{\mathbf{u}}) \subset \mathbf{L}^{2}(\Omega)$ of $\bar{\mathbf{u}}$, open neighborhoods $\mathcal{O}(\bar{\mathbf{y}}) \subset \mathbf{V}(\Omega)$ and $\mathcal{O}(\bar{p}) \subset L_0^2(\Omega)$ of $\bar{\mathbf{y}}$ and $\bar{p}$, respectively, and a $C^2$ mapping $\mathcal{S}: \mathcal{O}(\bar{\mathbf{u}}) \to \mathcal{O}(\bar{\mathbf{y}}) \times \mathcal{O}(\bar{p})$, such that $\mathcal{S}(\bar{\mathbf{u}}) = (\bar{\mathbf{y}},\bar{p})$. In addition, the neighborhood $\mathcal{O}(\bar{\mathbf{u}})$ can be taken convex and such that, for each $\mathbf{u}\in \mathcal{O}(\bar{\mathbf{u}})$, 

$\bullet$ pair $(\mathbf{y},p) = \mathcal{S}(\mathbf{u})$ uniquely solves \eqref{eq:weak_navier_stokes_eq} with $\mathbf{f}=\mathbf{u}$ in $\mathcal{O}(\bar{\mathbf{y}}) \times \mathcal{O}(\bar{p})$,

$\bullet$ $\mathcal{S}'(\mathbf{u}):\mathbf{H}^{-1}(\Omega)\to \mathbf{V}(\Omega)\times L_0^2(\Omega)$ is an isomorphism,

$\bullet$ if $\mathbf{g}\in \mathbf{H}^{-1}(\Omega)$, then $(\boldsymbol{\varphi},\zeta)=\mathcal{S}'(\mathbf{u})\mathbf{g}\in \mathbf{H}_0^1(\Omega)\times L_0^2(\Omega)$ corresponds to the unique solution to \eqref{eq:first_deriv_S*}.
\end{theorem}
\begin{proof}
The proof follows from \cite[Theorem 2.10]{MR3936891}, which is based on the implicit function theorem; see also \cite[Theorem 2.5]{MR2338434}. For brevity, we omit the details.  
\end{proof}

Based on the previous result, we prove the following result.

\begin{proposition}[auxiliary estimate: state variable]\label{prop:states_y-yt}
Let $\mathbf{u},\mathbf{u}_{\theta}\in \mathbf{U}_{ad}\cap \mathcal{O}(\bar{\mathbf{u}})$ and $(\mathbf{y},p) = \mathcal{S}(\mathbf{u})$ and $(\mathbf{y}_{\theta},p_{\theta}) = \mathcal{S}(\mathbf{u}_{\theta})$.
Then, for every $\mathsf{r} \in (2n/(n+2),2]$, we have
\[
    \|\mathbf{y} - \mathbf{y}_{\theta}\|_{\mathbf{W}^{2,\mathsf{r}}(\Omega)} 
    \lesssim
    \|\mathbf{u} - \mathbf{u}_{\theta}\|_{\mathbf{L}^{\mathsf{r}}(\Omega)}.
\]
\end{proposition}
\begin{proof}
The pair $(\mathbf{y} - \mathbf{y}_{\theta}, p - p_{\theta}) \in \mathbf{H}_0^{1}(\Omega)\times L_0^{2}(\Omega)$ is the unique solution to:
\begin{align*}
 - \nu\Delta (\mathbf{y} - \mathbf{y}_{\theta}) + \nabla(p - p_{\theta}) &\, = \mathbf{u} - \mathbf{u}_{\theta} - (\mathbf{y}_{\theta}\cdot \nabla) (\mathbf{y} - \mathbf{y}_{\theta}) - ((\mathbf{y} - \mathbf{y}_{\theta})\cdot \nabla) \mathbf{y} \text{ in }\Omega,\\
\textnormal{div}(\mathbf{y} - \mathbf{y}_{\theta}) &\,  = 0 \text{ in }\Omega, \qquad \mathbf{y} - \mathbf{y}_{\theta} = \mathbf{0} \text{ on }\partial\Omega.
\end{align*}
Let us prove that the right-hand side of the momentum equation belongs to $\mathbf{L}^{\mathsf{r}}(\Omega)$ with $\mathsf{r}\in (2n/(n+2),2]$.
It is clear that $\mathbf{u} - \mathbf{u}_{\theta} \in \mathbf{L}^{\mathsf{r}}(\Omega)$.
Moreover, using that $\mathbf{y},\mathbf{y}_{\theta}\in \mathbf{H}^{2}(\Omega)$ and the Sobolev embedding $\mathbf{H}_0^{1}(\Omega)\hookrightarrow \mathbf{L}^{6}(\Omega)$, we get
\begin{align*}
\|(\mathbf{y}_{\theta}\cdot \nabla) (\mathbf{y} - \mathbf{y}_{\theta})\|_{\mathbf{L}^{\mathsf{r}}(\Omega)} &\, + \|((\mathbf{y} - \mathbf{y}_{\theta})\cdot \nabla) \mathbf{y}\|_{\mathbf{L}^{\mathsf{r}}(\Omega)}
\\
\lesssim &\, 
\|\mathbf{y}_{\theta}\|_{\mathbf{L}^{\frac{2\mathsf{r}}{2-\mathsf{r}}}(\Omega)}\|\nabla(\mathbf{y} - \mathbf{y}_{\theta})\|_{\mathbf{L}^{2}(\Omega)} + \|\mathbf{y} - \mathbf{y}_{\theta}\|_{\mathbf{L}^{6}(\Omega)}\|\nabla\mathbf{y}\|_{\mathbf{L}^{\frac{6\mathsf{r}}{6-\mathsf{r}}}(\Omega)}\\
\lesssim &\, 
\|\nabla(\mathbf{y} - \mathbf{y}_{\theta})\|_{\mathbf{L}^{2}(\Omega)}\left(\|\mathbf{y}_{\theta}\|_{\mathbf{H}^{2}(\Omega)} + \|\mathbf{y}\|_{\mathbf{H}^{2}(\Omega)}\right).  
\end{align*}
Hence, using \cite[Corollary 1.8]{MR2987056}, we have $\|\mathbf{y} - \mathbf{y}_{\theta}\|_{\mathbf{W}^{2,\mathsf{r}}(\Omega)}
    \lesssim
    \|\mathbf{u} - \mathbf{u}_{\theta}\|_{\mathbf{L}^{\mathsf{r}}(\Omega)} + \|\nabla(\mathbf{y} - \mathbf{y}_{\theta})\|_{\mathbf{L}^{2}(\Omega)}$.
To estimate $\|\nabla(\mathbf{y} - \mathbf{y}_{\theta})\|_{\mathbf{L}^{2}(\Omega)}$, we use that $\mathbf{y}$ and $\mathbf{y}_{\theta}$ are regular solutions to apply a mean value theorem for operators \cite[Proposition 5.3.11]{MR2511061}. 
This yields
\[
    \|\nabla(\mathbf{y} - \mathbf{y}_{\theta})\|_{\mathbf{L}^{2}(\Omega)}
    \leq \sup_{\mathfrak{t} \in [0,1]} \| \mathcal{S}'( (1-\mathfrak{t}) \mathbf{u}_{\theta}  + \mathfrak{t} \mathbf{u} )\|  
    \|\mathbf{u} - \mathbf{u}_{\theta}\|_{\mathbf{H}^{-1}(\Omega)},
\]
which, in view of the estimate $\|\mathbf{u} - \mathbf{u}_{\theta}\|_{\mathbf{H}^{-1}(\Omega)}\lesssim \|\mathbf{u} - \mathbf{u}_{\theta}\|_{\mathbf{L}^{\mathsf{r}}(\Omega)}$, finishes the proof. 
\end{proof}

By Theorem \ref{thm:properties_C_to_S}, for every $\mathbf{u}\in \mathcal{O}(\bar{\mathbf{u}})$, the set $S(\mathbf{u})\cap \mathcal{O}(\bar{\mathbf{y}})\times \mathcal{O}(\bar{p})$ consists exclusively of the pair $(\mathbf{y}_{\mathbf{u}},p_{\mathbf{u}})$.
Note that this local uniqueness does not exclude the existence of other weak solutions outside this neighborhood.

From now on, we assume that $\bar{\mathbf{u}}$ is a local minimum of \eqref{eq:minimize_cost_func} such that $\bar{\mathbf{y}}$ is regular. 
Solutions satisfying this condition will be referred to as \emph{local nonsingular solutions}. 
For such a local solution, we can apply Theorem \ref{thm:properties_C_to_S} and define the control problem
\[
\text{inf}\{ J(\mathbf{u}) : ~ \mathbf{u}\in \mathbf{U}_{ad}\cap \mathcal{O}(\bar{\mathbf{u}})\},
\]
where $J: \mathbf{U}_{ad} \cap \mathcal{O}(\bar{\mathbf{u}}) \to \mathbb{R}$ is given by $J(\mathbf{u}) := \mathcal{J}(\mathbf{u})\rvert_{ \mathbf{U}_{ad}\cap \mathcal{O}(\bar{\mathbf{u}})} = \frac{1}{2}\|\mathbf{y}_{\mathbf{u}}-\mathbf{y}_\Omega\|_{\mathbf{L}^{2}(\Omega)}^2$.
Then $\bar{\mathbf{u}}$ is a local solution to the previous problem.

The following result is a direct consequence of Theorem \ref{thm:properties_C_to_S}.
\begin{theorem}
    The function $J: \mathbf{U}_{ad} \cap \mathcal{O}(\bar{\mathbf{u}}) \to \mathbb{R}$ is of class $C^2$ and we have the following identities for the first and second variation:
    \begin{align}\label{eq:charact_J'}
       &J'(\mathbf{u})\mathbf{g}= (\mathbf{z}_{\mathbf{u}},\mathbf{g})_{\mathbf{L}^{2}(\Omega)},\\
    &J''(\mathbf{u})\mathbf{g}^2 =  \|\boldsymbol\varphi_{\mathbf{g}}\|_{\mathbf{L}^{2}(\Omega)}^2 -2 b(\boldsymbol\varphi_{\mathbf{g}};\boldsymbol\varphi_{\mathbf{g}},\mathbf{z}_{\mathbf{u}}). \label{eq:charact_J''}
    \end{align}
    Here, $(\boldsymbol\varphi_{\mathbf{g}},\zeta_{\mathbf{g}})$ and $(\mathbf{z}_{\mathbf{u}},\pi_{\mathbf{u}})$ solve \eqref{eq:first_deriv_S*} and \eqref{eq:adj_eq_var_form}, respectively.
\end{theorem}

The first-order necessary optimality condition reads as follows: Every locally optimal nonsingular solution $\bar{\mathbf{u}} \in \mathbf{U}_{ad}$ satisfies 
\begin{align}\label{eq:1st_opt_cond}
    J'(\bar{\mathbf{u}})(\mathbf{u} - \bar{\mathbf{u}}) = (\bar{\mathbf{z}}, \mathbf{u} - \bar{\mathbf{u}})_{\mathbf{L}^{2}(\Omega)} \geq 0 \qquad \forall \mathbf{u} \in \mathbf{U}_{ad},
\end{align}
where $(\bar{\mathbf{z}},\bar{\pi}) := (\mathbf{z}_{\bar{\mathbf{u}}},\pi_{\bar{\mathbf{u}}})\in \mathbf{H}_0^1(\Omega)\times L_0^2(\Omega)$ denotes the unique solution to \eqref{eq:adj_eq_var_form} with $\mathbf{y}= \bar{\mathbf{y}}$ and $\mathbf{h} = \bar{\mathbf{y}} - \mathbf{y}_{\Omega}$.
Moreover, inequality \eqref{eq:1st_opt_cond} implies, for $i\in\{1,\ldots,n\}$ and a.e. $x\in\Omega$, that
\begin{align*}
\bar{\mathbf{u}}_{i}(x) = \mathbf{a}_i ~\text{ if }\bar{\mathbf{z}}_i(x) > 0, \quad ~  \bar{\mathbf{u}}_{i}(x) \in [\mathbf{a}_{i}, \mathbf{b}_i] ~\text{ if }\bar{\mathbf{z}}_i(x) = 0, \quad ~ \bar{\mathbf{u}}_{i}(x) = \mathbf{b}_i ~\text{ if }\bar{\mathbf{z}}_i(x) < 0.
\end{align*}

\begin{lemma}[second variation estimate]\label{lemma:sec_order}
Let $\gamma \in (n/(n + 2), 1]$ be given and let $\mathbf{u},\hat{\mathbf{u}} \in \mathcal{O}(\bar{\mathbf{u}})$. 
Define $\mathbf{u}_{\theta} := \mathbf{u} + \theta (\hat{\mathbf{u}} - \mathbf{u})$ for some measurable function $\theta$ satisfying $0 \leq \theta \leq 1$ a.e. in $\Omega$. 
Assume that $\mathbf{y}_{\Omega}\in \mathbf{L}^{\mathsf{p}}(\Omega)$ for $\mathsf{p} > n$ arbitrarily close to $n$.
Then, for all $\varepsilon > 0$ there exists $\delta>0$ such that
\[
    |J''(\mathbf{u})(\hat{\mathbf{u}} - \mathbf{u})^{2} - J''(\mathbf{u}_{\theta})(\hat{\mathbf{u}} - \mathbf{u})^{2}| \leq \varepsilon \|\hat{\mathbf{u}} - \mathbf{u}\|_{\mathbf{L}^{1}(\Omega)}^{1+ \frac{1}{\gamma}}
\]
when $\|\hat{\mathbf{u}} - \mathbf{u}\|_{\mathbf{L}^{1}(\Omega)} < \delta$.
\end{lemma}
\begin{proof}
    Since $\mathbf{u},\hat{\mathbf{u}} \in \mathcal{O}(\bar{\mathbf{u}})$, it follows that $\mathbf{u}_{\theta}\in \mathcal{O}(\bar{\mathbf{u}})$. 
    Moreover, in view of Theorem \ref{thm:properties_C_to_S}, we see that $(\boldsymbol\varphi,\zeta) := \mathcal{S}'(\mathbf{u})(\hat{\mathbf{u}} - \mathbf{u})$ and $(\boldsymbol \varphi_{\theta},\zeta_{\theta}) := \mathcal{S}'(\mathbf{u}_{\theta})(\hat{\mathbf{u}} - \mathbf{u})$ are well defined.
    Similarly, the adjoint states $(\mathbf{z}_{\mathbf{u}},\pi_{\mathbf{u}})$ and $(\mathbf{z}_{\mathbf{u}_{\theta}},\pi_{\mathbf{u}_{\theta}})$ are also well defined.
    With the previous variables at hand, and using the characterization in \eqref{eq:charact_J''}, we write
    \begin{align*}
    &\, J''(\mathbf{u})(\hat{\mathbf{u}} - \mathbf{u})^{2} - J''(\mathbf{u}_{\theta})(\hat{\mathbf{u}} - \mathbf{u})^{2}
        = \|\boldsymbol \varphi\|_{\mathbf{L}^{2}(\Omega)}^{2} - \|\boldsymbol \varphi_{\theta}\|_{\mathbf{L}^{2}(\Omega)}^{2} + 2b(\boldsymbol \varphi_{\theta}; \boldsymbol \varphi_{\theta}, \mathbf{z}_{\mathbf{u}_{\theta}})\\
    &\, - 2b(\boldsymbol \varphi; \boldsymbol \varphi, \mathbf{z}_{\mathbf{u}}) = (\boldsymbol \varphi - \boldsymbol \varphi_{\theta}, \boldsymbol \varphi + \boldsymbol \varphi_{\theta})_{\mathbf{L}^{2}(\Omega)} + 2b(\boldsymbol \varphi_{\theta}; \boldsymbol \varphi_{\theta}, \mathbf{z}_{\mathbf{u}_{\theta}} - \mathbf{z}_{\mathbf{u}}) \\
    & \, + 2b(\boldsymbol \varphi_{\theta}; \boldsymbol \varphi_{\theta} - \boldsymbol \varphi, \mathbf{z}_{\mathbf{u}}) + 2b(\boldsymbol \varphi_{\theta} - \boldsymbol \varphi; \boldsymbol \varphi, \mathbf{z}_{\mathbf{u}}) =: \mathsf{I} + \mathsf{II} + \mathsf{III} + \mathsf{IV}.
    \end{align*}
We now provide an estimate for each term in the previous equality. 

Using $\|\boldsymbol \varphi_{\theta}\|_{\mathbf{L}^{2}(\Omega)} \leq \|\boldsymbol \varphi_{\theta} - \boldsymbol \varphi\|_{\mathbf{L}^{2}(\Omega)} + \|\boldsymbol \varphi\|_{\mathbf{L}^{2}(\Omega)}$ in combination with Proposition \ref{prop:aux_linear_states}, we get
\begin{align*}
|\mathsf{I}| 
\leq &\,
\|\boldsymbol \varphi - \boldsymbol \varphi_{\theta}\|_{\mathbf{L}^{2}(\Omega)}(\|\boldsymbol \varphi\|_{\mathbf{L}^{2}(\Omega)} + \|\boldsymbol \varphi_{\theta}\|_{\mathbf{L}^{2}(\Omega)}) \\
\leq &\,
\|\boldsymbol \varphi - \boldsymbol \varphi_{\theta}\|_{\mathbf{L}^{2}(\Omega)}(2\|\boldsymbol \varphi\|_{\mathbf{L}^{2}(\Omega)} + \|\boldsymbol \varphi_{\theta} - \boldsymbol \varphi\|_{\mathbf{L}^{2}(\Omega)})  \\
\lesssim &\,
\|\nabla(\mathbf{y} - \mathbf{y}_{\theta})\|_{\mathbf{W}^{1,\mathsf{p}}(\Omega)}\|\boldsymbol{\varphi}\|_{\mathbf{L}^{2}(\Omega)}^{2}(1 + \|\nabla(\mathbf{y} - \mathbf{y}_{\theta})\|_{\mathbf{W}^{1,\mathsf{p}}(\Omega)}).
\end{align*}
We recall that $\mathsf{p}>n$ is arbitrarily close to $n$.
Since the term $(1 + \|\nabla(\mathbf{y} - \mathbf{y}_{\theta})\|_{\mathbf{W}^{1,\mathsf{p}}(\Omega)})$ is uniformly bounded, we conclude that
\[
|\mathsf{I}| 
\lesssim
\|\nabla(\mathbf{y} - \mathbf{y}_{\theta})\|_{\mathbf{W}^{1,\mathsf{p}}(\Omega)}\|\boldsymbol{\varphi}\|_{\mathbf{L}^{2}(\Omega)}^{2}.
\]

We now concentrate on $\mathsf{II}$. 
The use of \eqref{eq:properties_trilinear}, H\"older's inequality, and Lemma \ref{lemma:aux_prop_z1-z2} implies that 
\begin{align*}
|\mathsf{II}| 
=
|2b(\boldsymbol \varphi_{\theta};\mathbf{z}_{\mathbf{u}_{\theta}} - \mathbf{z}_{\mathbf{u}},\boldsymbol \varphi_{\theta})|
\lesssim &\, 
\|\nabla (\mathbf{z}_{\mathbf{u}_{\theta}} - \mathbf{z}_{\mathbf{u}})\|_{\mathbf{L}^{\infty}(\Omega)}\|\boldsymbol \varphi_{\theta}\|_{\mathbf{L}^{2}(\Omega)}^{2}\\
\lesssim &\, (\|\mathbf{y}_{\theta} - \mathbf{y}\|_{\mathbf{L}^{\mathsf{p}}(\Omega)} + \|\nabla (\mathbf{y}_{\theta} - \mathbf{y})\|_{\mathbf{L}^{\mathsf{p}}(\Omega)})\|\boldsymbol \varphi_{\theta}\|_{\mathbf{L}^{2}(\Omega)}^{2}.
\end{align*}
This, combined with Poincar\'e's inequality, gives $|\mathsf{II}|\lesssim \|\nabla(\mathbf{y} - \mathbf{y}_{\theta})\|_{\mathbf{W}^{1,\mathsf{p}}(\Omega)}\|\boldsymbol{\varphi}_{\theta}\|_{\mathbf{L}^{2}(\Omega)}^{2}$, which, in light of Proposition \ref{prop:aux_linear_states}, reveals that $|\mathsf{II}|\lesssim \|\nabla(\mathbf{y} - \mathbf{y}_{\theta})\|_{\mathbf{W}^{1,\mathsf{p}}(\Omega)}\|\boldsymbol{\varphi}\|_{\mathbf{L}^{2}(\Omega)}^{2}$.

To estimate $\mathsf{III}$ we utilize the property \eqref{eq:properties_trilinear}, H\"older's inequality, Proposition \ref{prop:aux_linear_states}, and estimate \eqref{eq:estimate_z_max_norm}. 
These arguments yield
\begin{multline*}
|\mathsf{III}|
=
|2b(\boldsymbol \varphi_{\theta};\mathbf{z}_{\mathbf{u}},\boldsymbol \varphi_{\theta} - \boldsymbol \varphi)|
\lesssim 
\|\boldsymbol \varphi_{\theta}\|_{\mathbf{L}^{2}(\Omega)}\|\nabla\mathbf{z}_{\mathbf{u}}\|_{\mathbf{L}^{\infty}(\Omega)}\|\boldsymbol \varphi_{\theta} - \boldsymbol \varphi\|_{\mathbf{L}^{2}(\Omega)}\\
\lesssim  \|\boldsymbol \varphi_{\theta}\|_{\mathbf{L}^{2}(\Omega)}(\|\mathbf{y} - \mathbf{y}_{\Omega}\|_{\mathbf{L}^{\mathsf{p}}} + \|\nabla \mathbf{y}\|_{\mathbf{L}^{\mathsf{p}}(\Omega)})\|\nabla(\mathbf{y} - \mathbf{y}_{\theta})\|_{\mathbf{W}^{1,\mathsf{p}}(\Omega)}\|\boldsymbol{\varphi}\|_{\mathbf{L}^{2}(\Omega)}.
\end{multline*}
Hence, using $\|\boldsymbol \varphi_{\theta}\|_{\mathbf{L}^{2}(\Omega)}\lesssim \|\boldsymbol{\varphi}\|_{\mathbf{L}^{2}(\Omega)}(1 + \|\nabla(\mathbf{y} - \mathbf{y}_{\theta})\|_{\mathbf{W}^{1,\mathsf{p}}(\Omega)})$ and the boundedness of $\mathbf{y},\mathbf{y}_{\theta}\in \mathbf{W}^{1,\mathsf{p}}(\Omega)$ and $\mathbf{y}_{\Omega}\in \mathbf{L}^{\mathsf{p}}(\Omega)$, we obtain  $|\mathsf{III}| \lesssim \|\nabla(\mathbf{y} - \mathbf{y}_{\theta})\|_{\mathbf{W}^{1,\mathsf{p}}(\Omega)}\|\boldsymbol{\varphi}\|_{\mathbf{L}^{2}(\Omega)}^{2}$.

The estimation of $\mathsf{IV}$ follows arguments similar to those used to estimate $\mathsf{III}$.
In particular, we get $|\mathsf{IV}|\lesssim \|\nabla(\mathbf{y} - \mathbf{y}_{\theta})\|_{\mathbf{W}^{1,\mathsf{p}}(\Omega)}\|\boldsymbol{\varphi}\|_{\mathbf{L}^{2}(\Omega)}^{2}$.

A combination of the estimates derived for $\mathsf{I},\mathsf{II}, \mathsf{III}$ and $\mathsf{IV}$ reveals that 
\begin{align*}
|J''(\mathbf{u})(\hat{\mathbf{u}} - \mathbf{u})^{2} - J''(\mathbf{u}_{\theta})(\hat{\mathbf{u}} - \mathbf{u})^{2}|
     \lesssim
 \|\nabla(\mathbf{y} - \mathbf{y}_{\theta})\|_{\mathbf{W}^{1,\mathsf{p}}(\Omega)}\|\boldsymbol{\varphi}\|_{\mathbf{L}^{2}(\Omega)}^{2},
\end{align*}
from which, in light of Proposition \ref{prop:aux_linear_states_rhsL1}, we get 
\begin{align}\label{eq:j''_last_estimate}
    |J''(\mathbf{u})(\hat{\mathbf{u}} - \mathbf{u})^{2} - J''(\mathbf{u}_{\theta})(\hat{\mathbf{u}} - \mathbf{u})^{2}|
     \lesssim
 \|\nabla(\mathbf{y} - \mathbf{y}_{\theta})\|_{\mathbf{W}^{1,\mathsf{p}}(\Omega)}\|\hat{\mathbf{u}} - \mathbf{u}\|_{\mathbf{L}^{1}(\Omega)}^{2}.
\end{align}
To estimate $\|\nabla(\mathbf{y} - \mathbf{y}_{\theta})\|_{\mathbf{W}^{1,\mathsf{p}}(\Omega)}$, we use the embedding $\mathbf{W}^{2,\mathsf{r}}(\Omega)\hookrightarrow \mathbf{W}^{1,\mathsf{p}}(\Omega)$ with $\mathsf{r} > (n/2 + \delta_{\mathsf{p}})$ with $\delta_{\mathsf{p}} > 0$ arbitrarily close to $0$.
This embedding and the use of Proposition \ref{prop:states_y-yt} imply that
\begin{align*}
    \|\nabla(\mathbf{y} - \mathbf{y}_{\theta})\|_{\mathbf{W}^{1,\mathsf{p}}(\Omega)} 
    \lesssim
    \|\mathbf{y} - \mathbf{y}_{\theta}\|_{\mathbf{W}^{2,\mathsf{r}}(\Omega)} 
    \lesssim
    \|\mathbf{u} - \mathbf{u}_{\theta}\|_{\mathbf{L}^{\mathsf{r}}(\Omega)}.
\end{align*}
Using, in the latter, the estimate $\|\mathbf{u} - \mathbf{u}_{\theta}\|_{\mathbf{L}^{\mathsf{r}}(\Omega)} \lesssim \|\mathbf{u} - \mathbf{u}_{\theta}\|_{\mathbf{L}^{1}(\Omega)}^{\frac{1}{\mathsf{r}}}$, we derive a refined bound, which, when substituted into \eqref{eq:j''_last_estimate}, yields
\begin{align*}
    |J''(\mathbf{u})(\hat{\mathbf{u}} - \mathbf{u})^{2} - J''(\mathbf{u}_{\theta})(\hat{\mathbf{u}} - \mathbf{u})^{2}|
    \lesssim
\|\hat{\mathbf{u}} - \mathbf{u}\|_{\mathbf{L}^{1}(\Omega)}^{2 + \frac{1}{\mathsf{r}}} = \|\hat{\mathbf{u}} - \mathbf{u}\|_{\mathbf{L}^{1}(\Omega)}^{\tilde{\delta}}\|\hat{\mathbf{u}} - \mathbf{u}\|_{\mathbf{L}^{1}(\Omega)}^{2 + \frac{1}{\mathsf{r}} - \tilde{\delta}},
\end{align*}
with $\tilde{\delta} > 0$ arbitrarily small. 
A straightforward analysis reveals that $2 + \frac{1}{\mathsf{r}} - \tilde{\delta} = 1 + \frac{1}{\gamma}$ with $\gamma \in (n/(n + 2), 1]$.
\end{proof}

\begin{assu}\label{sufficientcon}
Let $\bar{\mathbf{u}}$ be a local nonsingular solution of \eqref{eq:minimize_cost_func}, i.e., the velocity field $\bar{\mathbf{y}}$, with $(\bar{\mathbf{y}},\bar{p})\in S(\bar{{\mathbf{u}}})$, is regular. 
    We assume that there exist constants $c,\delta>0$ and $\gamma \in (n/(n + 2), 1]$ such that
    \begin{equation}\label{controlgrowth}
        J'(\bar{\mathbf{u}})(\mathbf{u}-\bar{\mathbf{u}})+J''(\bar{\mathbf{u}})(\mathbf{u}-\bar{\mathbf{u}})^2 \geq c \|{\mathbf{u}-\bar{\mathbf{u}}}\|_{\mathbf{L}^1(\Omega)}^{1+\frac{1}{\gamma}},
    \end{equation}
  for all $\mathbf{u}\in \mathbf{U}_{ad}$ with $ \|{\mathbf{u}-\bar{\mathbf{u}}}\|_{\mathbf{L}^1(\Omega)}\leq \delta$.
\end{assu}
It is well known that a  growth of the functional, respectively, its joint first and second variation with regard to the controls as assumed in \eqref{controlgrowth} implies the bang-bang structure of the control $\bar{\mathbf{u}}$, see, for instance, \cite[Proposition 4.1]{zbMATH07643482}. 
    A reverse result also holds, that is, if $\bar{\mathbf{u}}$ is a strict local minimizer that is bang-bang, the functional satisfies a growth in dependence of the $L^1$-norm of the controls \cite[Proposition 2.12]{zbMATH07865499}.
    Besides, the condition \eqref{controlgrowth} is directly related to the strong metric subregularity of the optimality mapping, which is used to study the behavior of the optimal control problem under perturbations. 
    That is, condition \eqref{controlgrowth} is sufficient for strong metric Hölder subregularity of the optimality mapping \cite[Theorem 4.11]{zbMATH07643482}, while a slightly weaker version, where the $J''$-term is weighted with the factor $1/2$, \cite[Theorem A.5]{zbMATH07846602}, is necessary.
    There are several conditions that, together with an additional mild assumption on the second variation $J''$, are known to imply \eqref{controlgrowth}.
    For instance, it is well known that the component-wise condition associated with the adjoint velocity field $\bar{\mathbf{z}}$, that there exist $c_i, \varepsilon_i, \gamma_i>0$, $i=1,...,n$, such that
    \begin{equation}\label{eq:adjiointmeasure}
        \vert [ \vert \bar{\mathbf{z}}_i \vert \leq \varepsilon ] \vert \leq c_i \varepsilon^{\gamma_i},
    \end{equation}
    implies, defining $\gamma:=\min \{ \gamma_i\}$ and $\tilde{c}:= \max \{c_i\}$, that
    \begin{equation}\label{eq:firstvargrowth}
        J'(\bar{\mathbf{u}})(\mathbf{u}-\bar{\mathbf{u}}) \geq \tilde{c} \| \mathbf{u}-\bar{\mathbf{u}} \|_{\mathbf{L}^1(\Omega)}^{1+\frac{1}{\gamma}}.
    \end{equation}
In \cite[Lemma 3.2]{MR2891922} it was shown that, for $\gamma_i=1$, \eqref{eq:adjiointmeasure} is implied by the condition that the adjoint $\bar{\mathbf{z}}$ belongs to $\mathbf{C}^1(\bar{\Omega})$ and satisfies, on $K = \{ x \in \bar{\Omega} \mid \bar{\mathbf{z}}_{i}(x) = 0\}$, that
\begin{equation}\label{eq:cond_DH}
    \min_{x \in K} |\nabla \bar{\mathbf{z}}_{i}(x)| > 0.
\end{equation}
Then we can ask whether a similar condition holds for $\gamma<1$, as in our case.
An answer to this question is given in the appendix.

The proof of the next theorem follows standard methods employing Taylor's theorem and Lemma \ref{lemma:sec_order}.

\begin{theorem}[local optimality]
\label{strictilocalopt}
Let $\bar{\mathbf{u}}$ be a local nonsingular solution of \eqref{eq:minimize_cost_func} satisfying assumption \eqref{controlgrowth}. 
Then, it holds 
\[
J(\mathbf{u})-J(\bar{\mathbf{u}})  \gtrsim  \|{\mathbf{u}-\bar{\mathbf{u}}}\|_{\mathbf{L}^1(\Omega)}^{1+\frac{1}{\gamma}}
\]
 for all $\mathbf{u}\in \mathbf{U}_{ad}$ with $ \|{\mathbf{u}-\bar{\mathbf{u}}}\|_{\mathbf{L}^1(\Omega)}\leq \delta$.
\end{theorem}


\section{Finite element approximation}\label{sec:fem}

We denote by $\mathcal{T}_h = \{ T\}$ a conforming partition of $\bar{\Omega}$ into simplices $T$ of size $h_T = \text{diam}(T)$, with $h:=\max \{ h_T: T \in \mathcal{T}_h \}$. 
By $\mathbb{T} = \{\mathcal{T}_h \}_{h>0}$, we denote a collection of conforming and quasi-uniform meshes $\mathcal{T}_h$.

To approximate velocity fields and pressures we consider the lowest order Taylor--Hood pair (see, e.g., \cite[Section 4.2.5]{MR2050138}):
\begin{align*}
Q_h &=\{ q_h \in C(\bar{\Omega})\, : \, q_h|_T \in \mathbb{P}_{1}(T) \ \forall \: T \in\mathcal{T}_h \}  \cap L_0^2(\Omega), \\
\mathbf{V}_h &=\{ \mathbf{v}_h \in \mathbf{C}(\bar{\Omega})\, : \, \mathbf{v}_h|_T \in [\mathbb{P}_2(T)]^n \ \forall \: T\in\mathcal{T}_h \}\cap \mathbf{H}_0^1(\Omega).
\end{align*}


\subsection{Discrete state equations}\label{sec:fem_state}

Let $\mathbf{u}\in \mathbf{L}^2(\Omega)$. We introduce the following finite element approximation of problem \eqref{eq:state_equations}: Find $(\mathbf{y}_{h},p_{h})\in \mathbf{V}_{h}\times Q_{h}$ such that 
\begin{align}
\begin{split}
\label{eq:discrete_state_eq}
\nu(\nabla \mathbf{y}_{h},\nabla \mathbf{v}_{h})_{\mathbf{L}^2(\Omega)} + b(\mathbf{y}_{h};\mathbf{y}_{h},\mathbf{v}_{h}) - (p_{h},\text{div }\mathbf{v}_{h})_{\Omega}
= &\, (\mathbf{u},\mathbf{v}_{h})_{\mathbf{L}^2(\Omega)}
\\
(q_{h},\text{div }\mathbf{y}_{h})_{\Omega} = &\, 0
\end{split}
\end{align}
for all $(\mathbf{v}_{h},q_{h})\in \mathbf{V}_{h}\times Q_{h}$.
If $\mathbf{u}\in\mathbf{L}^2(\Omega)$ is arbitrary, then problem \eqref{eq:discrete_state_eq} does not necessarily have a unique solution $(\mathbf{y}_{h},p_{h})$.
However, if $\mathbf{u}$ is close enough to a given local nonsingular solution $\bar{\mathbf{u}}$ to \eqref{eq:minimize_cost_func}, then there is a unique discrete solution to \eqref{eq:discrete_state_eq} that is sufficiently close to $(\bar{\mathbf{y}}, \bar{p})$. 
To be precise, there exist $\mathfrak{s},\mathfrak{l}> 0$, independent of $h$, and $h_{\star} > 0$ such that for all $h \in (0, h_{\star})$ and $\mathbf{u}\in B_{\mathfrak{s}}(\bar{\mathbf{u}})\subset\mathbf{L}^2(\Omega)$, problem \eqref{eq:discrete_state_eq} admits a unique solution $(\mathbf{y}_{h},p_{h})\in \left(B_{\mathfrak{l}}(\bar{\mathbf{y}})\times B_{\mathfrak{l}}(\bar{p})\right)\cap \left(\mathbf{V}_{h}\times Q_{h}\right)\subset \mathbf{H}_0^1(\Omega)\times L_0^2(\Omega)$; see \cite[Theorem 4.8]{MR2338434}.
We can thus define the operator $\mathcal{S}_{h} : B_{\mathfrak{s}}(\bar{\mathbf{u}}) \to B_{\mathfrak{l}}(\bar{\mathbf{y}})\times B_{\mathfrak{l}}(\bar{p})$ by $\mathcal{S}_{h}(\mathbf{u}) = (\mathbf{y}_{h},p_{h})$, with $(\mathbf{y}_{h},p_{h})$ being the unique solution to \eqref{eq:discrete_state_eq}.


\subsection{Semidiscrete optimal control problem}\label{sec:fem_ocp}

We consider a \emph{semidiscrete approach}, based on the variational discretization \cite{MR2122182}, in which the control variable is not discretized.
This approximation reads as follows: Find $\min \mathcal{J}_{h}(\boldsymbol{u})$ subject to 
\begin{align}
\label{eq:discrete_state_eq_semi}
\begin{split}
\nu(\nabla \mathbf{y}_{h},\nabla \mathbf{v}_{h})_{\mathbf{L}^2(\Omega)} + b(\mathbf{y}_{h};\mathbf{y}_{h},\mathbf{v}_{h}) - (p_{h},\text{div }\mathbf{v}_{h})_{\Omega}
=&\, (\boldsymbol{u},\mathbf{v}_{h})_{\mathbf{L}^2(\Omega)} \\
\qquad (q_{h},\text{div }\mathbf{y}_{h})_{\Omega} = &\, 0
\end{split}
\end{align}
for all $(\mathbf{v}_{h},q_{h})\in \mathbf{V}_{h}\times Q_{h}$, and $\boldsymbol u\in\mathbf{U}_{ad}$. 

\begin{theorem}[existence and convergence]\label{thm:exist_and_conv_sol_var}
Assume that \eqref{eq:minimize_cost_func} admits a local nonsingular solution $\bar{\mathbf{u}}$.
Then, there exists $h_{*} > 0$ such that for all $h \in (0,h_{*})$ the semidiscrete problem has a solution $\bar{\boldsymbol{u}}$. 
\end{theorem}
\begin{proof}
The existence of a feasible pair is direct. The existence of a solution to the semidiscrete problem follows from the arguments given in the proof of \cite[Theorem 4.11]{MR2338434}. For brevity, we skip the details.
\end{proof}

In what follows, we will write $\bar{\boldsymbol{u}}_{h}$ to denote a solution to the semidiscrete problem.

Let $\bar{\mathbf{u}}$ be a local nonsingular solution to \eqref{eq:minimize_cost_func}. Let $\mathfrak{s},\mathfrak{l}> 0$, $B_{\mathfrak{s}}(\bar{\mathbf{u}})$, $B_{\mathfrak{l}}(\bar{\mathbf{y}})$, $B_{\mathfrak{l}}(\bar{p})$, and $\mathcal{S}_{h}$ be given as in Section \ref{sec:fem_state}.
We consider the control problem
\begin{equation}\label{eq:discrete_local_problem}
\text{inf}\{ J_{h}(\mathbf{u}) : ~ \mathbf{u}\in \mathbf{U}_{ad}\cap B_{\mathfrak{s}}(\bar{\mathbf{u}})\},    
\end{equation}
where $J_{h}: \mathbf{U}_{ad} \cap B_{\mathfrak{s}}(\bar{\mathbf{u}}) \to \mathbb{R}$ is defined by $J_h(\mathbf{u}) := \frac{1}{2}\|\mathbf{y}_{h}-\mathbf{y}_\Omega\|_{\mathbf{L}^{2}(\Omega)}^2$ with $(\mathbf{y}_{h}, p_{h}) = \mathcal{S}_{h}(\mathbf{u})$. 

In view of Theorem \ref{thm:exist_and_conv_sol_var}, we infer that the problem \eqref{eq:discrete_local_problem} has at least one solution $\bar{\boldsymbol{u}}_{h}$. 
Moreover, the discrete control $\bar{\boldsymbol{u}}_{h}$ satisfies
\begin{equation}\label{eq:discrete_var_ineq}
J_{h}(\bar{\boldsymbol{u}}_{h})(\mathbf{u} - \bar{\boldsymbol{u}}_{h}) =  (\bar{\mathbf{z}}_{h},\mathbf{u} - \bar{\boldsymbol{u}}_{h})^{}_{\mathbf{L}^2(\Omega)}  \geq  0 \quad \forall
  \mathbf{u} \in \mathbf{U}_{ad},
\end{equation}
where the pair $(\bar{\mathbf{z}}_{h},\bar{r}_{h}) \in \mathbf{V}_{h}\times Q_{h}$ solves
\begin{align}\label{eq:discrete_adjoint_equation}
&\nu(\nabla \mathbf{w}_{h}, \nabla \bar{\mathbf{z}}_{h})_{\mathbf{L}^2(\Omega)}+b(\bar{\mathbf{y}}_{h};\mathbf{w}_{h},\bar{\mathbf{z}}_{h}) +b(\mathbf{w}_{h};\bar{\mathbf{y}}_{h},\bar{\mathbf{z}}_{h}) \\
&-(\bar{r}_{h},\text{div } \mathbf{w}_{h})_{\Omega} =(\bar{\mathbf{y}}_{h}-\mathbf{y}_\Omega,\mathbf{w}_{h})_{\mathbf{L}^2(\Omega)}, \qquad (s_{h},\text{div } \bar{\mathbf{z}}_{h})_{\Omega} = 0,\nonumber
\end{align}
for all $(\mathbf{w}_{h},s_{h}) \in \mathbf{V}_h\times Q_{h}$.
Here, $(\bar{\mathbf{y}}_{h},\bar{p}_{h}) \in \mathbf{V}_{h}\times Q_{h}$ denotes the unique solution to \eqref{eq:discrete_state_eq_semi} with $\boldsymbol{u} = \bar{\boldsymbol{u}}_h$. 
We note that, for every $h>0$ small enough, problem \eqref{eq:discrete_var_ineq} is well posed; see \cite[Lemma 4.12]{MR2338434}.

\begin{theorem}[convergence of solutions]\label{thm:convergence_of_sol}
Let $\bar{\mathbf{u}}\in \mathbf{U}_{ad}$ be a strict local nonsingular solution to \eqref{eq:minimize_cost_func} satisfying the growth condition \eqref{controlgrowth}. 
Then, there exists a sequence $\{\bar{\boldsymbol{u}}_{h}\}_{h>0}$ of local minima of the semidiscrete problem such that $\bar{\boldsymbol{u}}_{h}\to \bar{\mathbf{u}}$ in $\mathbf{L}^{1}(\Omega)$.
\begin{proof}
On the one hand, the fact that $\bar{\mathbf{u}}$ is a strict local nonsingular solution to \eqref{eq:minimize_cost_func}, implies that there exists $\varepsilon > 0$ such that the problem 
\[
\min\{ J(\mathbf{u}): \mathbf{u} \in\mathbf{U}_{ad}\cap B_{\varepsilon}(\bar{\mathbf{u}})\}, \qquad B_{\varepsilon}(\bar{\mathbf{u}}):=\{ \mathbf{u} \in \mathbf{L}^1(\Omega): \|\bar{\mathbf{u}}-\mathbf{u}\|_{\mathbf{L}^1(\Omega)}\leq \varepsilon\},
\]
admits $\bar{\mathbf{u}}$ as a unique solution. 
On the other hand, we consider the problem: 
Find min$\{J_{h}(\mathbf{u}): \mathbf{u}\in \mathbf{U}_{ad}\cap B_{\varepsilon}(\bar{\mathbf{u}})\}$.
In view of Theorem \ref{thm:exist_and_conv_sol_var}, this problem admits a global solution $\bar{\boldsymbol{u}}_{h} \in \mathbf{U}_{ad}$.
Moreover, there exists a subsequence $\{\bar{\boldsymbol{u}}_h\}_{h > 0}$ such that $\bar{\boldsymbol{u}}_{h} \rightharpoonup \bar{\mathbf{u}}$ weakly$^{*}$ in $\mathbf{L}^{\infty}(\Omega)$ to a solution of $\min\{ J(\mathbf{u}): \mathbf{u} \in\mathbf{U}_{ad}\cap B_{\varepsilon}(\bar{\mathbf{u}})\}$. Since the latter problem admits a unique solution $\bar{\mathbf{u}}$, we must have $\bar{\boldsymbol{u}}_{h} \rightharpoonup \bar{\mathbf{u}}$ weakly$^{*}$ in $\mathbf{L}^{\infty}(\Omega)$ as $h \downarrow 0$. 
Since $\bar{\mathbf{u}}$ satisfies the growth condition \eqref{controlgrowth}, we infer that $\bar{\mathbf{u}}$ corresponds to a bang-bang solution.
Therefore, the convergence $\bar{\boldsymbol{u}}_{h} \rightharpoonup \bar{\mathbf{u}}$ weakly$^{*}$ in $\mathbf{L}^{\infty}(\Omega)$  implies that $\bar{\boldsymbol{u}}_{h}\to \bar{\mathbf{u}}$ in $\mathbf{L}^{1}(\Omega)$ (see \cite[Lemma 4.2]{zbMATH07643482}). 
From this we observe that $\bar{\boldsymbol{u}}_{h}\in B_{\mathfrak{s}}(\bar{\mathbf{u}})$ when $h$ is small enough, which yields that $\bar{\boldsymbol{u}}_{h}$ is a local minimum of the semidiscrete optimal control problem.
This concludes the proof.
\end{proof}
    
\end{theorem}


\subsection{Auxiliary adjoint equations}

Let $\bar{\boldsymbol{u}}_{h}$ be an element of the sequence $\{\bar{\boldsymbol{u}}_{h}\}_{h>0}$ given as in Theorem \ref{thm:convergence_of_sol} and such that $\bar{\boldsymbol{u}}_{h}\in \mathcal{O}(\bar{\mathbf{u}})$ (cf. Theorem \ref{thm:properties_C_to_S}).
We introduce the auxiliary variable $(\mathbf{z}_{\bar{\boldsymbol{u}}_{h}},r_{\bar{\boldsymbol{u}}_{h}}) \in \mathbf{H}_0^1(\Omega)\times L_0^{2}(\Omega)$ as the unique solution to
\begin{align}\label{eq:aux_z_uh}
&\nu(\nabla \mathbf{w}, \nabla \mathbf{z}_{\bar{\boldsymbol{u}}_{h}})_{\mathbf{L}^2(\Omega)} + b(\mathbf{y}_{\bar{\boldsymbol{u}}_{h}};\mathbf{w},\mathbf{z}_{\bar{\boldsymbol{u}}_{h}}) + b(\mathbf{w};\mathbf{y}_{\bar{\boldsymbol{u}}_{h}},\mathbf{z}_{\bar{\boldsymbol{u}}_{h}}) \\
&- (r_{\bar{\boldsymbol{u}}_{h}},\text{div } \mathbf{w})_{\Omega}   = (\mathbf{y}_{\bar{\boldsymbol{u}}_{h}} - \mathbf{y}_\Omega,\mathbf{w})_{\mathbf{L}^2(\Omega)}, 
\qquad
(s,\text{div } \mathbf{z}_{\bar{\boldsymbol{u}}_{h}})_{\Omega} =   0, \nonumber
\end{align}
for all $(\mathbf{w},s) \in \mathbf{H}_0^1(\Omega)\times L_0^2(\Omega)$, where $\mathbf{y}_{\bar{\boldsymbol{u}}_{h}}\in \mathbf{H}_0^1(\Omega)$ corresponds to the unique solution to \eqref{eq:weak_navier_stokes_eq} with $\mathbf{f}=\bar{\boldsymbol{u}}_{h}$. 
We recall that, when $\bar{\boldsymbol{u}}_{h}$ is sufficiently close to $\bar{\mathbf{u}}$, then $\mathbf{y}_{\bar{\boldsymbol{u}}_{h}}$ is regular. Consequently, the problem \eqref{eq:aux_z_uh} is well posed.
We let $(\hat{\mathbf{z}}_{h},\hat{r}_h)\in \mathbf{V}_{h}\times Q_h$ be its corresponding finite element approximation, i.e., this pair solves
\begin{equation}
\label{eq:discrete_adj_aux}
\begin{array}{rl}
&\nu(\nabla \mathbf{w}_{h}, \nabla \hat{\mathbf{z}}_{h})_{\mathbf{L}^2(\Omega)} + b(\mathbf{y}_{\bar{\boldsymbol{u}}_{h}};\mathbf{w}_{h},\hat{\mathbf{z}}_{h}) + b(\mathbf{w}_{h};\mathbf{y}_{\bar{\boldsymbol{u}}_{h}},\hat{\mathbf{z}}_{h}) \\
 & \quad \quad -  (\hat{r}_{h},\text{div } \mathbf{w}_{h})_{\Omega} = (\mathbf{y}_{\bar{\boldsymbol{u}}_{h}} - \mathbf{y}_\Omega,\mathbf{w}_{h})_{\mathbf{L}^2(\Omega)},   \quad
(s_{h},\text{div } \hat{\mathbf{z}}_{h})_{\Omega}  =   0, 
\end{array}
\end{equation}
for all $(\mathbf{w}_h,s_h) \in \mathbf{V}_{h}\times Q_{h}$. 
The fact that $\mathbf{y}_{\bar{\boldsymbol{u}}_{h}}$ is regular implies that the problem \eqref{eq:discrete_adj_aux} is well posed; see \cite[Lemma 4.12]{MR2338434}. 
Moreover, the fact that $\mathbf{z}_{\bar{\boldsymbol{u}}_{h}}\in \mathbf{H}^{2}(\Omega)$, standard interpolation and inverse estimates, and \cite[eq. (4.12)]{MR2338434}, allow us to arrive at
\begin{equation}\label{eq:hat_z_W14}
    \|\nabla \hat{\mathbf{z}}_{h}\|_{\mathbf{L}^{4}(\Omega)}
    \leq
    \|\nabla (\hat{\mathbf{z}}_{h} - \mathbf{z}_{\bar{\boldsymbol{u}}_{h}})\|_{\mathbf{L}^{4}(\Omega)} + \|\nabla \mathbf{z}_{\bar{\boldsymbol{u}}_{h}}\|_{\mathbf{L}^{4}(\Omega)}
    \lesssim
    h^{1-\frac{n}{4}} + C
    \lesssim 1.
\end{equation}

In the next result, we establish an approximation result in the maximum norm.

\begin{proposition}[maximum-norm error estimate]\label{prop:max_norm_estimate_zz}
    Let $(\mathbf{z}_{\bar{\boldsymbol{u}}_{h}},r_{\bar{\boldsymbol{u}}_{h}}) \in \mathbf{H}_0^1(\Omega)\times L_0^{2}(\Omega)$ be the unique solution to problem \eqref{eq:aux_z_uh} and let $(\hat{\mathbf{z}}_{h},\hat{r}_h)\in \mathbf{V}_{h}\times Q_h$ be its finite element approximation obtained as the solution to \eqref{eq:discrete_adj_aux}.
    Assume that $\mathbf{y}_{\Omega} \in \mathbf{L}^{\mathsf{p}}(\Omega)$ with $\mathsf{p}>n$ sufficiently close to $n$.
    Then, we have  
    \begin{align*}
        \|\mathbf{z}_{\bar{\boldsymbol{u}}_{h}} - \hat{\mathbf{z}}_{h}\|_{\mathbf{L}^{\infty}(\Omega)}
        \lesssim
        h|\log h|.
    \end{align*}
    \end{proposition}
    \begin{proof}
    Let $\mathbf{F}\in \mathbf{H}^{-1}(\Omega)$ be given by $\langle\mathbf{F},\mathbf{w}\rangle_{\mathbf{H}^{-1}(\Omega),\mathbf{H}_0^{1}(\Omega)} = b(\mathbf{y}_{\bar{\boldsymbol{u}}_{h}}; \mathbf{w},\mathbf{z}_{\bar{\boldsymbol{u}}_{h}}) + b(\mathbf{w}; \mathbf{y}_{\bar{\boldsymbol{u}}_{h}},\mathbf{z}_{\bar{\boldsymbol{u}}_{h}})$, and let $(\boldsymbol{\Psi},\xi) \in \mathbf{H}_0^{1}(\Omega)\times L_0^2(\Omega)$ be the unique solution to 
    \[
     \nu(\nabla \mathbf{w},\nabla \boldsymbol{\Psi})_{\mathbf{L}^{2}(\Omega)} - (\xi, \text{div }\mathbf{w})_{\Omega}  = \langle\mathbf{F},\mathbf{w}\rangle_{\mathbf{H}^{-1}(\Omega),\mathbf{H}_0^{1}(\Omega)}, \quad    (s, \text{div }\boldsymbol{\Psi})_{\Omega}  = 0,
    \]
    for all $(\mathbf{w},s) \in \mathbf{H}_0^1(\Omega)\times L_0^2(\Omega)$.
    We prove that $(\boldsymbol{\Psi},\xi) \in \mathbf{C}^{1,\alpha}(\bar{\Omega})\times C^{0,\alpha}(\bar{\Omega})$ for some $\alpha\in (0,1)$.
    To accomplish this task, it suffices to prove that $\mathbf{F}\in \mathbf{L}^{r}(\Omega)$ for some $r > n$ sufficiently close to $n$; see \cite[Section 1.3]{MR3422453} and \cite[Lemma 14]{MR3008832}.
    Using \eqref{eq:properties_trilinear} and the fact that div $\mathbf{y}_{\bar{\boldsymbol{u}}_{h}}=0$ and $\mathbf{y}_{\bar{\boldsymbol{u}}_{h}},\mathbf{z}_{\bar{\boldsymbol{u}}_{h}}\in \mathbf{H}^{2}(\Omega)$, we have
\begin{align*}
    |b(\mathbf{y}_{\bar{\boldsymbol{u}}_{h}}; \mathbf{w},\mathbf{z}_{\bar{\boldsymbol{u}}_{h}})|
     =
    |b(\mathbf{y}_{\bar{\boldsymbol{u}}_{h}};\mathbf{z}_{\bar{\boldsymbol{u}}_{h}},\mathbf{w})|
    \leq & \,
    \|\mathbf{y}_{\bar{\boldsymbol{u}}_{h}}\|_{\mathbf{L}^{\infty}(\Omega)}\|\nabla \mathbf{z}_{\bar{\boldsymbol{u}}_{h}}\|_{\mathbf{L}^{r}(\Omega)}\|\mathbf{w}\|_{\mathbf{L}^{r^{\prime}}(\Omega)}\\
    \lesssim & \, \|\mathbf{y}_{\bar{\boldsymbol{u}}_{h}}\|_{\mathbf{L}^{\infty}(\Omega)}\|\nabla \mathbf{z}_{\bar{\boldsymbol{u}}_{h}}\|_{\mathbf{L}^{4}(\Omega)}\|\mathbf{w}\|_{\mathbf{L}^{r^{\prime}}(\Omega)}.
\end{align*}    
    Similarly, we obtain  
\begin{align*}
    |b(\mathbf{w}; \mathbf{y}_{\bar{\boldsymbol{u}}_{h}},\mathbf{z}_{\bar{\boldsymbol{u}}_{h}})|
    \leq & \,
    \|\mathbf{w}\|_{\mathbf{L}^{r^{\prime}}(\Omega)}\|\nabla \mathbf{y}_{\bar{\boldsymbol{u}}_{h}}\|_{\mathbf{L}^{r}(\Omega)}\|\mathbf{z}_{\bar{\boldsymbol{u}}_{h}}\|_{\mathbf{L}^{\infty}(\Omega)}\\
    \lesssim & \,
    \|\mathbf{w}\|_{\mathbf{L}^{r^{\prime}}(\Omega)}\|\nabla \mathbf{y}_{\bar{\boldsymbol{u}}_{h}}\|_{\mathbf{L}^{4}(\Omega)}\|\mathbf{z}_{\bar{\boldsymbol{u}}_{h}}\|_{\mathbf{L}^{\infty}(\Omega)}.
\end{align*}   
Hence, $\mathbf{F}\in \mathbf{L}^{r}(\Omega)$ with $r>n$ and thus $(\boldsymbol{\Psi},\xi) \in \mathbf{C}^{1,\alpha}(\bar{\Omega})\times C^{0,\alpha}(\bar{\Omega})$ satisfying $\|\nabla \boldsymbol{\Psi}\|_{\mathbf{L}^{\infty}(\Omega)} + \|\xi\|_{\mathbf{L}^{\infty}(\Omega)} \lesssim \|\nabla \mathbf{y}_{\bar{\boldsymbol{u}}_{h}}\|_{\mathbf{L}^{4}(\Omega)}\|\nabla \mathbf{z}_{\bar{\boldsymbol{u}}_{h}}\|_{\mathbf{L}^{4}(\Omega)}$.
We note immediately that an analogous argument can be used to prove, in view of the regularity $\mathbf{y}_\Omega \in \mathbf{L}^{\mathsf{p}}(\Omega)$ with $\mathsf{p}>n$ sufficiently close to $n$, that $(\mathbf{z}_{\bar{\boldsymbol{u}}_{h}},r_{\bar{\boldsymbol{u}}_{h}}) \in \mathbf{C}^{1,\alpha}(\bar{\Omega})\times C^{0,\alpha}(\bar{\Omega})$; cf. equation \eqref{eq:estimate_z_max_norm}.

We also introduce the pair $(\widehat{\boldsymbol{\Psi}}_{h},\hat{\xi}_{h})\in \mathbf{V}_{h}\times Q_{h}$ as the unique solution to 
    \begin{align*}
     \nu(\nabla \mathbf{w}_{h},\nabla \widehat{\boldsymbol{\Psi}}_{h})_{\mathbf{L}^{2}(\Omega)} - (\hat{\xi}_{h}, \text{div }\mathbf{w}_{h})_{\Omega} & = b(\mathbf{y}_{\bar{\boldsymbol{u}}_{h}}; \mathbf{w}_{h},\hat{\mathbf{z}}_{h}) + b(\mathbf{w}_{h}; \mathbf{y}_{\bar{\boldsymbol{u}}_{h}},\hat{\mathbf{z}}_{h}) \\
     (s_{h}, \text{div }\widehat{\boldsymbol{\Psi}}_{h})_{\Omega} & = 0 \nonumber
    \end{align*}
for all $(\mathbf{w}_{h},s_{h})\in \mathbf{V}_{h}\times Q_{h}$.
We recall that $(\hat{\mathbf{z}}_{h},\hat{r}_h)\in \mathbf{V}_{h}\times Q_h$ solves \eqref{eq:discrete_adj_aux}.

Define $(\mathbf{K},k) := (\boldsymbol{\Psi} + \mathbf{z}_{\bar{\boldsymbol{u}}_{h}},\xi + r_{\bar{\boldsymbol{u}}_{h}})$ and $(\widehat{\mathbf{K}}_{h},\hat{k}_{h}) := (\widehat{\boldsymbol{\Psi}}_{h} + \hat{\mathbf{z}}_{h},\hat{\xi}_{h} + \hat{r}_{h})$.
We immediately note that $(\mathbf{K},k) \in \mathbf{C}^{1,\alpha}(\bar{\Omega})\times C^{0,\alpha}(\bar{\Omega})$ and $(\widehat{\mathbf{K}}_{h},\hat{k}_h)\in \mathbf{V}_{h}\times Q_{h}$ solve
\[
    \nu (\nabla \mathbf{w},\nabla \mathbf{K})_{\mathbf{L}^{2}(\Omega)} - (k, \text{div }\mathbf{w})_{\Omega} = (\mathbf{y}_{\bar{\boldsymbol{u}}_{h}} - \mathbf{y}_{\Omega},\mathbf{w})_{\mathbf{L}^{2}(\Omega)}, \quad
    (s,\text{div }\mathbf{K})_{\Omega}  = 0,
\]
for all $(\mathbf{w},s) \in \mathbf{H}_0^1(\Omega)\times L_0^2(\Omega)$,
and 
\begin{align*}
    \nu (\nabla \mathbf{w}_{h},\nabla \widehat{\mathbf{K}}_{h})_{\mathbf{L}^{2}(\Omega)} - (\hat{k}_{h}, \text{div }\mathbf{w}_{h})_{\Omega} = (\mathbf{y}_{\bar{\boldsymbol{u}}_{h}} - \mathbf{y}_{\Omega},\mathbf{w}_{h})_{\mathbf{L}^{2}(\Omega)}, \quad
    (s_{h},\text{div }\widehat{\mathbf{K}}_{h})_{\Omega}  = 0, 
\end{align*}
for all $(\mathbf{w}_{h},s_{h})\in \mathbf{V}_{h}\times Q_{h}$, respectively.
The triangle inequality yields
\[
\|\mathbf{z}_{\bar{\boldsymbol{u}}_{h}} - \hat{\mathbf{z}}_{h}\|_{\mathbf{L}^{\infty}(\Omega)}
    \leq  \|\mathbf{K} - \widehat{\mathbf{K}}_{h}\|_{\mathbf{L}^{\infty}(\Omega)} +  \|\widehat{\boldsymbol{\Psi}}_{h} - \boldsymbol{\Psi} \|_{\mathbf{L}^{\infty}(\Omega)}.
\]
Using \cite[Remark 2.18]{MR4101369} we have $\|\mathbf{K} - \widehat{\mathbf{K}}_{h}\|_{\mathbf{L}^{\infty}(\Omega)} \lesssim h|\log h|$. 
Hence
\begin{align*}
\|\mathbf{z}_{\bar{\boldsymbol{u}}_{h}} - \hat{\mathbf{z}}_{h}\|_{\mathbf{L}^{\infty}(\Omega)}
    \lesssim  h|\log h| +  \|\widehat{\boldsymbol{\Psi}}_{h} - \boldsymbol{\Psi} \|_{\mathbf{L}^{\infty}(\Omega)}.
\end{align*}
We introduce the term $(\widehat{\boldsymbol{\Psi}},\hat{\xi})\in \mathbf{H}_0^{1}(\Omega)\times L_0^2(\Omega)$ defined as the unique solution to
\[
     \nu(\nabla \mathbf{w},\nabla \widehat{\boldsymbol{\Psi}})_{\mathbf{L}^{2}(\Omega)} - (\hat{\xi}, \text{div }\mathbf{w})_{\Omega}  = b(\mathbf{y}_{\bar{\boldsymbol{u}}_{h}}; \mathbf{w},\hat{\mathbf{z}}_{h}) + b(\mathbf{w}; \mathbf{y}_{\bar{\boldsymbol{u}}_{h}},\hat{\mathbf{z}}_{h}), \quad
     (s, \text{div }\widehat{\boldsymbol{\Psi}})_{\Omega}  = 0,
\]
for all $(\mathbf{w},s) \in \mathbf{H}_0^1(\Omega)\times L_0^2(\Omega)$.
We now estimate 
\begin{align*}
\|\widehat{\boldsymbol{\Psi}}_{h} - \boldsymbol{\Psi} \|_{\mathbf{L}^{\infty}(\Omega)} \leq  \|\widehat{\boldsymbol{\Psi}}_{h} - \widehat{\boldsymbol{\Psi}} \|_{\mathbf{L}^{\infty}(\Omega)} +  \|\widehat{\boldsymbol{\Psi}} - \boldsymbol{\Psi} \|_{\mathbf{L}^{\infty}(\Omega)} =: \mathbf{I} + \mathbf{II}.
\end{align*}
To control $\mathbf{I}$, we note that $\|\nabla \widehat{\boldsymbol{\Psi}}\|_{\mathbf{L}^{\infty}(\Omega)} \lesssim \|\nabla \mathbf{y}_{\bar{\boldsymbol{u}}_{h}}\|_{\mathbf{L}^{4}(\Omega)}\|\nabla\hat{\mathbf{z}}_{h}\|_{\mathbf{L}^{4}(\Omega)}$.
The latter, in view of \eqref{eq:hat_z_W14} implies that $\|\nabla \boldsymbol{\Psi}\|_{\mathbf{L}^{\infty}(\Omega)}$ is uniformly bounded. Consequently, we are allowed to apply \cite[Remark 2.18]{MR4101369} estimating the term $\mathbf{I}$ as follows: $\mathbf{I} \lesssim h|\log h|$.
The term $\mathbf{II}$ is bounded using the stability estimate $\mathbf{II}\lesssim \|\widehat{\boldsymbol{\Psi}} - \boldsymbol{\Psi}\|_{\mathbf{H}^{2}(\Omega)} \lesssim \|\nabla(\mathbf{z}_{\bar{\boldsymbol{u}}_{h}} - \hat{\mathbf{z}}_{h})\|_{\mathbf{L}^{2}(\Omega)}$ in combination with the error estimate $\|\nabla(\mathbf{z}_{\bar{\boldsymbol{u}}_{h}} - \hat{\mathbf{z}}_{h})\|_{\mathbf{L}^{2}(\Omega)} \lesssim h$ (cf. \cite[Lemma 4.13]{MR2338434}).

A combination of the previous estimates yields the desired bound.
    \end{proof}


\subsection{A priori error estimates}

\begin{theorem}[convergence rate: control variable]\label{thm:estimate_control}
Let $\bar{\mathbf{u}}\in \mathbf{U}_{ad}$ be a local nonsingular solution to \eqref{eq:minimize_cost_func} with $(\bar{\mathbf{y}},\bar{p})$ and $(\bar{\mathbf{z}},\bar{r})$ being the corresponding state and adjoint state variables, respectively. 
Let $\{\bar{\boldsymbol{u}}_h\}_{h>0}$ be the sequence of local semidiscrete minimizers constructed in Theorem \ref{thm:convergence_of_sol}.
For each $h>0$, let $(\bar{\mathbf{y}}_{h},\bar{p}_h)$ and $(\bar{\mathbf{z}}_{h},\bar{r}_h)$ denote the corresponding discrete state and discrete adjoint state variables, respectively. 
If assumption \ref{sufficientcon} holds and $\mathbf{y}_{\Omega}\in \mathbf{L}^{\mathsf{p}}(\Omega)$ for some $\mathsf{p}>n$ sufficiently close to $n$, then
\[
    \|\bar{\mathbf{u}} - \bar{\boldsymbol{u}}_h\|_{\mathbf{L}^{1}(\Omega)}
    \lesssim
    (h|\log h|)^{\gamma} \qquad 
    \forall h < h_{*},
\]
where $\gamma \in (n/(n+2), 1]$.
\end{theorem}
\begin{proof}
Choosing $\mathbf{u} = \bar{\mathbf{u}}$ in the semidiscrete first-order optimality condition \eqref{eq:discrete_var_ineq} and using the mean value theorem, we obtain
\begin{align*}
  0  \geq & \,  J_{h}^{\prime}(\bar{\boldsymbol{u}}_{h})(\bar{\boldsymbol{u}}_{h} - \bar{\mathbf{u}}) \\
   \geq & \, [J'(\bar{\mathbf{u}})(\bar{\boldsymbol{u}}_{h} - \bar{\mathbf{u}}) + J''(\bar{\mathbf{u}})(\bar{\boldsymbol{u}}_{h} - \bar{\mathbf{u}})^2] \\
   & - |[J'(\bar{\boldsymbol{u}}_{h}) - J'(\bar{\mathbf{u}})](\bar{\boldsymbol{u}}_{h} - \bar{\mathbf{u}}) - J''(\bar{\mathbf{u}})(\bar{\boldsymbol{u}}_{h} - \bar{\mathbf{u}})^2| + [J'_{h}(\bar{\boldsymbol{u}}_{h}) - J'(\bar{\boldsymbol{u}}_{h})](\bar{\boldsymbol{u}}_{h} - \bar{\mathbf{u}}) \\
 = & \, [J'(\bar{\mathbf{u}})(\bar{\boldsymbol{u}}_{h} - \bar{\mathbf{u}}) + J''(\bar{\mathbf{u}})(\bar{\boldsymbol{u}}_{h} - \bar{\mathbf{u}})^2]  \\
 & - |[J''(\mathbf{u}_{\theta})- J''(\bar{\mathbf{u}})](\bar{\boldsymbol{u}}_{h} - \bar{\mathbf{u}})^2| + [J'_{h}(\bar{\boldsymbol{u}}_{h}) - J'(\bar{\boldsymbol{u}}_{h})](\bar{\boldsymbol{u}}_{h} - \bar{\mathbf{u}}),
\end{align*}
with $\mathbf{u}_{\theta} = \bar{\mathbf{u}} + \theta (\bar{\boldsymbol{u}}_{h} - \bar{\mathbf{u}})$ with $\theta\in (0,1)$.
Since $\bar{\boldsymbol{u}}_h \to \bar{\mathbf{u}}$ in $\mathbf{L}^{1}(\Omega)$ when $h\to0$, we take $h$ sufficiently small such that we can use the growth condition \eqref{controlgrowth} in combination with the estimate $\vert [J''(\bar{\boldsymbol{u}}_{h}) -J''(\bar{\mathbf{u}})](\bar{\boldsymbol{u}}_{h}-\mathbf{\bar u})^{2}\vert \leq  \frac{c}{2}\| \bar{\boldsymbol{u}}_{h}-\bar{\mathbf{u}}\|_{\mathbf{L}^1(\Omega)}^{1+\frac{1}{\gamma}}$ (see Lemma \ref{lemma:sec_order}) in the previous inequality. 
These arguments yield
\begin{align*}
    [J'(\bar{\boldsymbol{u}}_{h}) - J'_{h}(\bar{\boldsymbol{u}}_{h})](\bar{\boldsymbol{u}}_{h} - \bar{\mathbf{u}}) \gtrsim \| \bar{\boldsymbol{u}}_{h}-\bar{\mathbf{u}}\|_{\mathbf{L}^1(\Omega)}^{1+\frac{1}{\gamma}} \quad \forall h < h_{*}.
\end{align*}
Invoke the auxiliary variable $(\mathbf{z}_{\bar{\boldsymbol{u}}_{h}},r_{\bar{\boldsymbol{u}}_{h}}) \in \mathbf{H}_0^1(\Omega)\times L_0^{2}(\Omega)$, given as the unique solution to \eqref{eq:aux_z_uh}. 
In view of \eqref{eq:discrete_var_ineq} and $J'(\bar{\boldsymbol{u}}_{h})(\bar{\boldsymbol{u}}_{h} - \bar{\mathbf{u}}) = (\mathbf{z}_{\bar{\boldsymbol{u}}_{h}},\bar{\boldsymbol{u}}_{h} - \bar{\mathbf{u}})_{\mathbf{L}^{2}(\Omega)}$, it follows that
\begin{equation}\label{eq:u-uh--z-zh}
\|\bar{\boldsymbol{u}}_{h} - \bar{\mathbf{u}}\|_{\mathbf{L}^1(\Omega)}^{1+\frac{1}{\gamma}} 
\lesssim
(\mathbf{z}_{\bar{\boldsymbol{u}}_{h}} - \bar{\mathbf{z}}_{h},\bar{\boldsymbol{u}}_{h} - \bar{\mathbf{u}})_{\mathbf{L}^{2}(\Omega)}.   
\end{equation}
Then, the use of the triangle inequality with $\hat{\mathbf{z}}_{h}$ (cf. \eqref{eq:discrete_adj_aux}) implies that
\begin{align*}
\|\bar{\boldsymbol{u}}_{h} - \bar{\mathbf{u}}\|_{\mathbf{L}^1(\Omega)}^{\frac{1}{\gamma}} 
\lesssim
\|\mathbf{z}_{\bar{\boldsymbol{u}}_{h}} - \hat{\mathbf{z}}_{h}\|_{\mathbf{L}^{\infty}(\Omega)} + \|\hat{\mathbf{z}}_{h} - \bar{\mathbf{z}}_{h}\|_{\mathbf{L}^{\infty}(\Omega)}.
\end{align*}
Using the error estimate from Proposition \ref{prop:max_norm_estimate_zz}, we obtain 
\begin{equation}\label{eq:u-u-hleqz}
    \|\bar{\boldsymbol{u}}_{h} - \bar{\mathbf{u}}\|_{\mathbf{L}^1(\Omega)}^{\frac{1}{\gamma}}  \lesssim h|\log h| + \|\hat{\mathbf{z}}_{h} - \bar{\mathbf{z}}_{h}\|_{\mathbf{L}^{\infty}(\Omega)}.
\end{equation}
Moreover, in view of \cite[Lemma 1.142]{MR2050138}, we see that 
\begin{equation}\label{eq:max_h1}
\|\hat{\mathbf{z}}_{h} - \bar{\mathbf{z}}_{h}\|_{\mathbf{L}^{\infty}(\Omega)}
     \lesssim
\begin{cases}
 (1+|\log h|)\|\nabla(\hat{\mathbf{z}}_{h} - \bar{\mathbf{z}}_{h})\|_{\mathbf{L}^{2}(\Omega)} \quad &n = 2,\\
h^{-\frac{1}{2}}\|\nabla(\hat{\mathbf{z}}_{h} - \bar{\mathbf{z}}_{h})\|_{\mathbf{L}^{2}(\Omega)} \quad &n = 3.
\end{cases}
\end{equation}

For the sake of readability, we proceed only with case $n=3$; case $n=2$ being analogous. 
From \eqref{eq:u-u-hleqz} and \eqref{eq:max_h1}, we have 
\begin{equation}\label{eq:u-u-hleqz_H1}
    \|\bar{\boldsymbol{u}}_{h} - \bar{\mathbf{u}}\|_{\mathbf{L}^1(\Omega)}^{\frac{1}{\gamma}}  \lesssim h|\log h|  + h^{-\frac{1}{2}}\|\nabla(\hat{\mathbf{z}}_{h} - \bar{\mathbf{z}}_{h})\|_{\mathbf{L}^{2}(\Omega)}.
\end{equation}
To estimate the last term in \eqref{eq:u-u-hleqz_H1}, we note that $(\bar{\mathbf{z}}_{h} - \hat{\mathbf{z}}_{h},\bar{r}_{h} - \hat{r}_{h})\in \mathbf{V}_{h}\times Q_{h}$ solves
\begin{align*}
    &\nu(\nabla \mathbf{w}_{h},\nabla (\bar{\mathbf{z}}_{h} - \hat{\mathbf{z}}_{h}))_{\mathbf{L}^{2}(\Omega)} + b(\bar{\mathbf{y}}_{h}; \mathbf{w}_h,\bar{\mathbf{z}}_{h} - \hat{\mathbf{z}}_{h}) + b(\mathbf{w}_{h};\bar{\mathbf{y}}_{h}, \bar{\mathbf{z}}_{h} - \hat{\mathbf{z}}_{h})\\
    &\,  -(\bar{r}_{h} - \hat{r}_{h}, \text{div }\mathbf{w}_{h})_{\Omega} = (\bar{\mathbf{y}}_{h} - \mathbf{y}_{\bar{\boldsymbol{u}}_{h}}, \mathbf{w}_{h})_{\mathbf{L}^{2}(\Omega)} + b(\mathbf{y}_{\bar{\boldsymbol{u}}_{h}} - \bar{\mathbf{y}}_{h}; \mathbf{w}_{h}, \hat{\mathbf{z}}_{h})\\
    &\,+ b(\mathbf{w}_h; \mathbf{y}_{\bar{\boldsymbol{u}}_{h}} - \bar{\mathbf{y}}_{h}, \hat{\mathbf{z}}_{h}), \qquad (s_{h}, \text{div}(\bar{\mathbf{z}}_{h} - \hat{\mathbf{z}}_{h}))_{\Omega} = 0,
\end{align*}
for all $(\mathbf{w}_{h},s_{h})\in \mathbf{V}_{h}\times Q_h$. 
Then, the well posedness of the previous problem (see \cite[Lemma 4.12]{MR2338434}), reveals that 
\begin{align}\label{eq:estimate_z-zh_infty}
    \|\nabla(\hat{\mathbf{z}}_{h} - \bar{\mathbf{z}}_{h})\|_{\mathbf{L}^{2}(\Omega)}
    \lesssim &\,
   \big(\|\bar{\mathbf{y}}_{h} - \mathbf{y}_{\bar{\boldsymbol{u}}_{h}}\|_{\mathbf{L}^{2}(\Omega)}\\
    &\, + \|b(\mathbf{y}_{\bar{\boldsymbol{u}}_{h}} - \bar{\mathbf{y}}_{h}; \cdot, \hat{\mathbf{z}}_{h}) + b(\cdot; \mathbf{y}_{\bar{\boldsymbol{u}}_{h}} - \bar{\mathbf{y}}_{h}, \hat{\mathbf{z}}_{h})\|_{\mathbf{H}^{-1}(\Omega)}\big). \nonumber
\end{align}
H\"older's inequality yields that $\|b(\mathbf{y}_{\bar{\boldsymbol{u}}_{h}} - \bar{\mathbf{y}}_{h}; \cdot, \hat{\mathbf{z}}_{h}) \lesssim \|\mathbf{y}_{\bar{\boldsymbol{u}}_{h}} - \bar{\mathbf{y}}_{h}\|_{\mathbf{L}^{2}(\Omega)}\| \hat{\mathbf{z}}_{h}\|_{\mathbf{L}^{\infty}(\Omega)}$.
Integration by parts and H\"older's inequality imply, for every $\mathbf{w}\in \mathbf{H}_0^1(\Omega)$, that
\begin{align*}
&\,|b(\mathbf{w}; \mathbf{y}_{\bar{\boldsymbol{u}}_{h}} - \bar{\mathbf{y}}_{h}, \hat{\mathbf{z}}_{h})\|_{\mathbf{H}^{-1}(\Omega)}
\leq 
|b(\mathbf{w};\hat{\mathbf{z}}_{h},\mathbf{y}_{\bar{\boldsymbol{u}}_{h}} - \bar{\mathbf{y}}_{h})|
+
|((\text{div }\mathbf{w})\hat{\mathbf{z}}_{h},\mathbf{y}_{\bar{\boldsymbol{u}}_{h}} - \bar{\mathbf{y}}_{h})_{\mathbf{L}^{2}(\Omega)}|\\
\lesssim &\, 
\left(\|\mathbf{w}\|_{\mathbf{L}^{4}(\Omega)}\|\nabla \hat{\mathbf{z}}_{h}\|_{\mathbf{L}^{4}(\Omega)} + \|\nabla \mathbf{w}\|_{\mathbf{L}^{2}(\Omega)}\|\hat{\mathbf{z}}_{h}\|_{\mathbf{L}^{\infty}(\Omega)}\right)\|\mathbf{y}_{\bar{\boldsymbol{u}}_{h}} - \bar{\mathbf{y}}_{h}\|_{\mathbf{L}^{2}(\Omega)}\\
\lesssim &\, \left(\|\nabla \hat{\mathbf{z}}_{h}\|_{\mathbf{L}^{4}(\Omega)} + \|\hat{\mathbf{z}}_{h}\|_{\mathbf{L}^{\infty}(\Omega)}\right)\|\nabla \mathbf{w}\|_{\mathbf{L}^{2}(\Omega)}\|\mathbf{y}_{\bar{\boldsymbol{u}}_{h}} - \bar{\mathbf{y}}_{h}\|_{\mathbf{L}^{2}(\Omega)}.
\end{align*}
Using the embedding $\mathbf{W}^{1,4}(\Omega)\hookrightarrow \mathbf{L}^{\infty}(\Omega)$ and the bound \eqref{eq:hat_z_W14}, we have $\|\hat{\mathbf{z}}_{h}\|_{\mathbf{L}^{\infty}(\Omega)} \lesssim\| \nabla \hat{\mathbf{z}}_{h}\|_{\mathbf{L}^{4}(\Omega)} \leq C$ with $C > 0$. 
Moreover, in light \cite[Eq. (4.4)]{MR2338434} we see that $\|\mathbf{y}_{\bar{\boldsymbol{u}}_{h}} - \bar{\mathbf{y}}_{h}\|_{\mathbf{L}^{2}(\Omega)}\lesssim h^{2}$. Hence, using the previous estimates in \eqref{eq:estimate_z-zh_infty} we obtain
\[
    \|\nabla(\hat{\mathbf{z}}_{h} - \bar{\mathbf{z}}_{h})\|_{\mathbf{L}^{2}(\Omega)}
    \lesssim 
    h^{2}.
\]
The use of the latter in \eqref{eq:u-u-hleqz_H1} yields the desired estimate.
\end{proof}

\begin{corollary}[error estimate: optimal states]\label{coro:error_state}
Under the framework of Theorem \ref{thm:estimate_control}, we have
\[
    \|\bar{\mathbf{y}} - \bar{\mathbf{y}}_h\|_{\mathbf{L}^{2}(\Omega)} \lesssim (h|\log h|)^{\gamma}
    \qquad 
    \forall h < h_{*},
\]
where $\gamma \in (n/(n+2), 1]$.
\end{corollary}
\begin{proof}
    Let $(\mathbf{y}_{\bar{\boldsymbol{u}}_{h}},p_{\bar{\boldsymbol{u}}_{h}})\in \mathbf{H}_{0}^{1}(\Omega)\times L_0^{2}(\Omega)$ be the unique solution to \eqref{eq:weak_navier_stokes_eq} with $\mathbf{f}=\bar{\boldsymbol{u}}_{h}$. 
    The use of the estimate $\|\mathbf{y}_{\bar{\boldsymbol{u}}_{h}} - \bar{\mathbf{y}}_{h}\|_{\mathbf{L}^{2}(\Omega)}\lesssim h^{2}$ (\cite[Eq. (4.4)]{MR2338434}) implies that
    \begin{equation}\label{eq:triangle_state_y}
        \|\bar{\mathbf{y}} - \bar{\mathbf{y}}_{h}\|_{\mathbf{L}^{2}(\Omega)}
        \leq
        \|\bar{\mathbf{y}} - \mathbf{y}_{\bar{\boldsymbol{u}}_{h}}\|_{\mathbf{L}^{2}(\Omega)} + \|\mathbf{y}_{\bar{\boldsymbol{u}}_{h}} - \bar{\mathbf{y}}_{h}\|_{\mathbf{L}^{2}(\Omega)}
        \lesssim
        \|\bar{\mathbf{y}} - \mathbf{y}_{\bar{\boldsymbol{u}}_{h}}\|_{\mathbf{L}^{2}(\Omega)} + h^{2}.
    \end{equation}
To estimate the remaining term, we note that $(\bar{\mathbf{y}} - \mathbf{y}_{\bar{\boldsymbol{u}}_{h}}, \bar{p} - p_{\bar{\boldsymbol{u}}_{h}})$ solves
\begin{align*}
\begin{split}
&\,\nu(\nabla (\bar{\mathbf{y}} - \mathbf{y}_{\bar{\boldsymbol{u}}_{h}}),\nabla \mathbf{v})_{\mathbf{L}^2(\Omega)}\! + b(\bar{\mathbf{y}};\bar{\mathbf{y}} - \mathbf{y}_{\bar{\boldsymbol{u}}_{h}},\mathbf{v}) + b(\bar{\mathbf{y}} - \mathbf{y}_{\bar{\boldsymbol{u}}_{h}};\bar{\mathbf{y}},\mathbf{v})  - (\bar{p} - p_{\bar{\boldsymbol{u}}_{h}},\text{div }\mathbf{v})_{\Omega}\\
&\, = (\bar{\mathbf{u}} - \bar{\boldsymbol{u}}_{h},\mathbf{v})_{\mathbf{L}^2(\Omega)} + b(\bar{\mathbf{y}} - \mathbf{y}_{\bar{\boldsymbol{u}}_{h}};\bar{\mathbf{y}} - \mathbf{y}_{\bar{\boldsymbol{u}}_{h}},\mathbf{v}), 
\qquad (q,\text{div}(\bar{\mathbf{y}} - \mathbf{y}_{\bar{\boldsymbol{u}}_{h}}))_{\Omega} =  0,
\end{split}
\end{align*}
for all $(\mathbf{v},q)\in \mathbf{H}_0^{1}(\Omega)\times L_0^{2}(\Omega)$. 
In view of the estimate derived in Proposition \ref{prop:aux_linear_states_rhsL1}, and using Poincare's inequality, we see that
\[
    \|\bar{\mathbf{y}} - \mathbf{y}_{\bar{\boldsymbol{u}}_{h}}\|_{\mathbf{L}^{2}(\Omega)}
    \lesssim
    \|\bar{\mathbf{u}} - \bar{\boldsymbol{u}}_{h}\|_{\mathbf{L}^{1}(\Omega)} + \|\nabla(\bar{\mathbf{y}} - \mathbf{y}_{\bar{\boldsymbol{u}}_{h}})\|_{\mathbf{L}^{2}(\Omega)}^{2}.
\]
Hence, using $\|\nabla(\bar{\mathbf{y}} - \mathbf{y}_{\bar{\boldsymbol{u}}_{h}})\|_{\mathbf{L}^{2}(\Omega)} \lesssim \|\bar{\mathbf{u}} - \bar{\boldsymbol{u}}_{h}\|_{\mathbf{H}^{-1}(\Omega)}$, which follows from an application of a mean value theorem for operators \cite[Proposition 5.3.11]{MR2511061} (see also \cite[Proposition 7]{FO2023}), and the estimate $\|\bar{\mathbf{u}} - \bar{\boldsymbol{u}}_{h}\|_{\mathbf{H}^{-1}(\Omega)} \lesssim \|\bar{\mathbf{u}} - \bar{\boldsymbol{u}}_{h}\|_{\mathbf{L}^{\frac{6}{5}}(\Omega)}$ we conclude that 
\[    
\|\bar{\mathbf{y}} - \mathbf{y}_{\bar{\boldsymbol{u}}_{h}}\|_{\mathbf{L}^{2}(\Omega)}
    \lesssim 
    \|\bar{\mathbf{u}} - \bar{\boldsymbol{u}}_{h}\|_{\mathbf{L}^{1}(\Omega)} + \|\bar{\mathbf{u}} - \bar{\boldsymbol{u}}_{h}\|_{\mathbf{L}^{\frac{6}{5}}(\Omega)}^{2}
    \lesssim 
    \|\bar{\mathbf{u}} - \bar{\boldsymbol{u}}_{h}\|_{\mathbf{L}^{1}(\Omega)} + \|\bar{\mathbf{u}} - \bar{\boldsymbol{u}}_{h}\|_{\mathbf{L}^{1}(\Omega)}^{\frac{5}{3}}.
\]
From the latter, we infer that $\|\bar{\mathbf{y}} - \mathbf{y}_{\bar{\boldsymbol{u}}_{h}}\|_{\mathbf{L}^{2}(\Omega)} \lesssim \|\bar{\mathbf{u}} - \bar{\boldsymbol{u}}_{h}\|_{\mathbf{L}^{1}(\Omega)}$, which, in light of \eqref{eq:triangle_state_y} and Theorem \ref{thm:estimate_control}, concludes the proof.
\end{proof}

\begin{remark}[Improved estimate for case $n=2$]
    In the two-dimensional case we can remove the logarithmic term in the estimate of Theorem \ref{thm:estimate_control}.
    In fact, the argument relies on using, instead of the estimation $\|\mathbf{z}_{\bar{\boldsymbol{u}}_{h}} - \hat{\mathbf{z}}_{h}\|_{\mathbf{L}^{\infty}(\Omega)} \lesssim h|\log h|$ from Proposition \ref{prop:max_norm_estimate_zz}, the estimate $\|\mathbf{z}_{\bar{\boldsymbol{u}}_{h}} - \hat{\mathbf{z}}_{h}\|_{\mathbf{L}^{\infty}(\Omega)} \lesssim h\|\mathbf{z}_{\bar{\boldsymbol{u}}_{h}}\|_{\mathbf{H}^{2}(\Omega)}$.
    Using the latter in \eqref{eq:u-u-hleqz_H1}, and proceeding as the proof of Theorem \ref{thm:estimate_control}, we can conclude that $\|\bar{\boldsymbol{u}}_{h} - \bar{\mathbf{u}}\|_{\mathbf{L}^1(\Omega)} \lesssim h^{\gamma}$. 
    From this, it thus also follows that $\|\bar{\mathbf{y}} - \bar{\mathbf{y}}_{h}\|_{\mathbf{L}^{2}(\Omega)}  \lesssim h^{\gamma}$.
\end{remark}


\section{A posteriori error estimates}\label{sec:apost_OCP}
Contrary to the a priori setting, the adaptive procedure generates a sequence of locally refined meshes rather than a family characterized by a global mesh-size parameter $h$. 
Therefore, we shall denote by $\mathcal{T}_\ell$ with $\ell\in\mathbb N_0$, the triangulation obtained after $\ell$ refinement steps, with \(\mathcal T_0\) denoting an initial mesh.
We assume that the collection $\{\mathcal{T}_\ell\}_{\ell\in \mathbb{N}_0}$ is conforming and is formed by shape-regular meshes that are refinements of $\mathcal T_0$.

Since the global mesh-size parameter $h = h_{\ell}$ does not necessarily decrease as $\ell \to \infty$, we cannot ensure the existence of discrete solutions for \eqref{eq:discrete_state_eq} nor for the semidiscrete optimal control problem presented in Section \ref{sec:fem_ocp}. 
For this reason, we need to consider an additional assumption (cf. assumption (A.1) in \cite{MR4301394}):
\begin{assu}\label{ass:small_mesh}
Let $\ell \in \mathbb{N}_0$ and $h_{\ell}$ be the global mesh-size parameter associated to the mesh $\mathcal{T}_{\ell}$. 
We assume that the parameter $h_{\ell}$ is sufficiently small such that the results in \cite[Theorem 4.11]{MR2338434} hold.
\end{assu}

Let $\bar{\mathbf{u}}\in \mathbf{U}_{ad}$ be a local nonsingular solution to \eqref{eq:minimize_cost_func}.
Taking $\boldsymbol{u}$ in \eqref{eq:discrete_state_eq_semi} such that $\|\boldsymbol{u} - \bar{\mathbf{u}}\|_{\mathbf{L}^{1}(\Omega)}$ is small enough and supposing that assumption \ref{ass:small_mesh} holds, we can guarantee the existence of at least one solution to the semidiscrete optimal control problem (see Theorem \ref{thm:exist_and_conv_sol_var}).
We denote by $\bar{\boldsymbol{u}}_{\ell}$ the solutions to this problem, and by $(\bar{\mathbf{y}}_{\ell},\bar{p}_{\ell})$ the solution to \eqref{eq:discrete_state_eq_semi} associated with $\boldsymbol{u} = \bar{\boldsymbol{u}}_{\ell}$.
We note that $\bar{\boldsymbol{u}}_{\ell}\in \mathcal{O}(\bar{\mathbf{u}})$; see Theorem \ref{thm:properties_C_to_S}.

To prove a reliability estimate, we use upper bounds on the error between discrete optimal variables and some auxiliary variables that we define in the following sections.


\subsection{A posteriori error analysis for the state equations}
\label{sec:a_post_NS}

Define the pair $(\mathbf{y}_{\bar{\boldsymbol{u}}_{\ell}},p_{\bar{\boldsymbol{u}}_{\ell}})\in \mathbf{H}_0^1(\Omega) \times L_0^{2}(\Omega)$ as the solution to \eqref{eq:weak_navier_stokes_eq} with $\mathbf{f}=\bar{\boldsymbol{u}}_{\ell}$. 
We immediately note that $(\bar{\mathbf{y}}_{\ell},\bar{p}_{\ell})$, solution to \eqref{eq:discrete_state_eq_semi} with $\boldsymbol{u} = \bar{\boldsymbol{u}}_{\ell}$, corresponds to its finite element approximation.
We consider $\boldsymbol{\psi}\in \mathbf{H}^{-1}(\Omega)$ and introduce the auxiliary dual linearized Navier–Stokes problem \cite[eq. (7)]{MR1831796}: Find $(\breve{\mathbf{z}},\breve{r}) \in \mathbf{H}_0^{1}(\Omega)\times L_0^{2}(\Omega)$ solution to
\begin{align}
\label{eq:aux_linear_NS}
\nu(\nabla \mathbf{w}, \nabla \breve{\mathbf{z}})_{\mathbf{L}^2(\Omega)} + b(\mathbf{y}_{\bar{\boldsymbol{u}}_{\ell}};\mathbf{w},\breve{\mathbf{z}}) + b(\mathbf{w};\bar{\mathbf{y}}_{\ell},\breve{\mathbf{z}}) + (\breve{r},\textnormal{div } \mathbf{w})_{\Omega} & \, = \langle \boldsymbol{\psi}, \mathbf{w}\rangle \\
-(s,\textnormal{div } \breve{\mathbf{z}})_{\Omega}  & \, =   0  \nonumber
\end{align}
for all $(\mathbf{w},s) \in \mathbf{H}_0^1(\Omega)\times L_0^2(\Omega)$. 

\begin{proposition}[well-posedness and regularity]\label{prop:wp_reg_breve_z}
Suppose that Assumption \ref{ass:small_mesh} holds and that
\begin{equation}\label{eq:assump_grady_small}
\|\nabla \bar{\mathbf{y}}_{\ell}\|_{\mathbf{L}^{2}(\Omega)} < \mathcal{C}_{b}^{-1}\nu.
\end{equation}
Then, the problem \eqref{eq:aux_linear_NS} is well posed.
If, in addition, $\boldsymbol{\psi}\in \mathbf{L}^{t}$ with $t\in [\tfrac{4}{3},2]$ and 
\begin{equation}\label{eq:assumption_y_12/5}
    \|\nabla \bar{\mathbf{y}}_{\ell}\|_{\mathbf{L}^{\frac{12}{5}}(\Omega)}
    \leq C \qquad (C > 0)
\end{equation}
when $n=3$, then $(\breve{\mathbf{z}},\breve{r}) \in (\mathbf{H}_0^{1}(\Omega) \times L_0^{2}(\Omega))\cap (\mathbf{W}^{2,t}(\Omega)\times W^{1,t}(\Omega))$.
\end{proposition}
\begin{proof}
Define $\mathcal{B}: \mathbf{H}_0^1(\Omega) \times \mathbf{H}_0^1(\Omega)$ by $\mathcal{B}(\mathbf{w},\mathbf{v}):=\nu (\nabla \mathbf{w},\nabla\mathbf{v})_{\mathbf{L}^2(\Omega)}+b(\mathbf{y}_{\bar{\boldsymbol{u}}_{\ell}};\mathbf{w},\mathbf{v})+b(\mathbf{w};\bar{\mathbf{y}}_{\ell}, \mathbf{v})$.
Property \eqref{eq:properties_trilinear} and assumption \eqref{eq:assump_grady_small} imply that
\begin{equation}
\label{eq:coercivity_bilinear_form_B}
\mathcal{B}(\mathbf{w},\mathbf{w}) \geq \|  \nabla \mathbf{w} \|^2_{\mathbf{L}^2(\Omega)}(\nu - \mathcal{C}_b \| \nabla \bar{\mathbf{y}}_{\ell} \|_{\mathbf{L}^2(\Omega)}) \gtrsim\|  \nabla \mathbf{w} \|^2_{\mathbf{L}^2(\Omega)}.
\end{equation}
Hence, $\mathcal{B}$ is coercive on $\mathbf{H}_0^1(\Omega)\times \mathbf{H}_0^1(\Omega)$. 
Therefore, the standard inf--sup theory for saddle-point problems applies; in particular, by \cite[Theorem 2.34]{MR2050138}, there exists a unique solution $(\breve{\mathbf{z}},\breve{r})\in \mathbf{H}_0^1(\Omega)\times L_0^2(\Omega)$ to \eqref{eq:aux_linear_NS}.
Moreover, choosing $\mathbf{w} = \breve{\mathbf{z}}$ and $s=0$ in \eqref{eq:aux_linear_NS} and using \eqref{eq:coercivity_bilinear_form_B}, we obtain $\|\nabla\breve{\mathbf{z}}\|_{\mathbf{L}^2(\Omega)} \lesssim  \|\boldsymbol{\psi}\|_{\mathbf{H}^{-1}(\Omega)}$.

To prove the additional regularity of the solution to \eqref{eq:aux_linear_NS}, we write 
\[
\nu(\nabla \mathbf{w}, \nabla \breve{\mathbf{z}})_{\mathbf{L}^2(\Omega)} + (\breve{r},\textnormal{div } \mathbf{w})_{\Omega} = (\boldsymbol{\psi}, \mathbf{w})_{\mathbf{L}^{2}(\Omega)} - b(\mathbf{y}_{\bar{\boldsymbol{u}}_{\ell}};\mathbf{w},\breve{\mathbf{z}}) - b(\mathbf{w};\bar{\mathbf{y}}_{\ell},\breve{\mathbf{z}}) =: \langle \mathbf{F}, \mathbf{w}\rangle.
\]
First, we prove that $\mathbf{F}\in \mathbf{L}^{\frac{4}{3}}(\Omega)$. 
For readability, we will develop the case $n=3$ only.
The use of H\"older's inequality and the fact that $\mathbf{y}_{\bar{\boldsymbol{u}}_{\ell}}\in \mathbf{L}^{\infty}(\Omega)$ imply that 
\[
|b(\mathbf{y}_{\bar{\boldsymbol{u}}_{\ell}};\mathbf{w},\breve{\mathbf{z}})| 
\! = \!
|b(\mathbf{y}_{\bar{\boldsymbol{u}}_{\ell}};\breve{\mathbf{z}},\mathbf{w})| 
\!\leq \!
\|\mathbf{y}_{\bar{\boldsymbol{u}}_{\ell}}\|_{\mathbf{L}^{\infty}(\Omega)}\|\nabla \breve{\mathbf{z}}\|_{\mathbf{L}^{2}(\Omega)} \|\mathbf{w}\|_{\mathbf{L}^{2}(\Omega)}
\! \lesssim \!
\|\boldsymbol{\psi}\|_{\mathbf{H}^{-1}(\Omega)}\|\mathbf{w}\|_{\mathbf{L}^{2}(\Omega)}.
\]
Using the embedding $\mathbf{H}_0^{1}(\Omega)\hookrightarrow \mathbf{L}^{4}(\Omega)$ and $\|\nabla\breve{\mathbf{z}}\|_{\mathbf{L}^2(\Omega)} \lesssim  \|\boldsymbol{\psi}\|_{\mathbf{H}^{-1}(\Omega)}$, we obtain
\[
|b(\mathbf{w};\bar{\mathbf{y}}_{\ell},\breve{\mathbf{z}})|
\leq 
\|\mathbf{w}\|_{\mathbf{L}^{4}(\Omega)}\|\nabla \bar{\mathbf{y}}_{\ell}\|_{\mathbf{L}^{2}(\Omega)}\|\breve{\mathbf{z}}\|_{\mathbf{L}^{4}(\Omega)} 
\lesssim
\|\boldsymbol{\psi}\|_{\mathbf{H}^{-1}(\Omega)}\|\mathbf{w}\|_{\mathbf{L}^{4}(\Omega)}.
\]
Therefore, since $\boldsymbol{\psi}\in \mathbf{L}^{t}$ with $t\in [\tfrac{4}{3},2]$ it follows that 
\[
\|\mathbf{F}\|_{\mathbf{L}^{\frac{4}{3}}(\Omega)} 
\lesssim 
\|\boldsymbol{\psi}\|_{\mathbf{L}^{\frac{4}{3}}(\Omega)} + \|\boldsymbol{\psi}\|_{\mathbf{H}^{-1}(\Omega)}
\lesssim 
\|\boldsymbol{\psi}\|_{\mathbf{L}^{\frac{4}{3}}(\Omega)},
\]
upon using the embedding $\mathbf{H}_0^{1}(\Omega)\hookrightarrow \mathbf{L}^{4}(\Omega)$.
This in combination with the convexity of $\Omega$ yields that $(\breve{\mathbf{z}},\breve{r}) \in \mathbf{W}^{2,\frac{4}{3}}(\Omega)\times W^{1,\frac{4}{3}}(\Omega)$ and $\|\breve{\mathbf{z}}\|_{\mathbf{W}^{2,\frac{4}{3}}(\Omega)} + \|\breve{r}\|_{W^{1,\frac{4}{3}}(\Omega)} \lesssim \|\mathbf{F}\|_{\mathbf{L}^{\frac{4}{3}}(\Omega)}$; see \cite[Corollary 1.8]{MR2987056}.
We now invoke the embedding $\mathbf{W}^{2,\frac{4}{3}}(\Omega) \hookrightarrow \mathbf{L}^{12}(\Omega)$ and assumption \eqref{eq:assumption_y_12/5} to estimate, again, the term $|b(\cdot;\bar{\mathbf{y}}_{\ell},\breve{\mathbf{z}})|$ as follows:
\[
|b(\mathbf{w};\bar{\mathbf{y}}_{\ell},\breve{\mathbf{z}})|
\leq
\|\mathbf{w}\|_{\mathbf{L}^{2}(\Omega)}\|\nabla \bar{\mathbf{y}}_{\ell}\|_{\mathbf{L}^{\frac{12}{5}}(\Omega)}\|\breve{\mathbf{z}}\|_{\mathbf{L}^{12}(\Omega)} 
\lesssim
\|\mathbf{F}\|_{\mathbf{L}^{\frac{4}{3}}(\Omega)} \|\mathbf{w}\|_{\mathbf{L}^{2}(\Omega)}.
\]
From this we infer that, for every $t\in [\tfrac{4}{3},2]$,
\[
\|\mathbf{F}\|_{\mathbf{L}^{t}(\Omega)} 
\lesssim \|\boldsymbol{\psi}\|_{\mathbf{L}^{t}(\Omega)} + \|\boldsymbol{\psi}\|_{\mathbf{H}^{-1}(\Omega)} + \|\mathbf{F}\|_{\mathbf{L}^{\frac{4}{3}}(\Omega)}\lesssim \|\boldsymbol{\psi}\|_{\mathbf{L}^{t}(\Omega)}.
\]
Hence, using again \cite[Corollary 1.8]{MR2987056} we conclude that $(\breve{\mathbf{z}},\breve{r}) \in \mathbf{W}^{2,t}(\Omega)\times W^{1,t}(\Omega)$ and $\|\breve{\mathbf{z}}\|_{\mathbf{W}^{2,t}(\Omega)} + \|\breve{r}\|_{W^{1,t}(\Omega)} \lesssim \|\boldsymbol{\psi}\|_{\mathbf{L}^{t}(\Omega)}$ for every $t\in [\tfrac{4}{3},2]$.
\end{proof}

Let $t'\in [2,4]$ be fixed.
We consider the a posteriori error estimator
\[
\eta_{\textrm{st},t^{\prime}}:= \left(\sum_{T\in\mathcal{T}_{\ell}}\eta_{\textrm{st},t^{\prime},T}^{t^{\prime}}\right)^{1/t'},
\]
where, for every $T\in\mathcal{T}_{\ell}$, the local indicators are defined by
\begin{align*}
\eta_{\textrm{st},t^{\prime},T}^{t^{\prime}} := &\, h_T^{2t'}\|\bar{\boldsymbol{u}}_{\ell} + \nu\Delta\bar{\mathbf{y}}_{\ell} - (\bar{\mathbf{y}}_{\ell}\cdot \nabla)\bar{\mathbf{y}}_{\ell} - \nabla \bar{p}_{\ell}\|_{\mathbf{L}^{t^{\prime}}(T)}^{t^{\prime}} \\ \nonumber
& + h_T^{t'}\|\text{div }\bar{\mathbf{y}}_{\ell}\|_{L^{t^{\prime}}(T)}^{t^{\prime}}
+ h_T^{t' + 1}\|\llbracket (\nu\nabla \bar{\mathbf{y}}_{\ell} - \bar{p}_{\ell}\mathbb{I}_{d} )\cdot \mathbf{n}\rrbracket\|_{\mathbf{L}^{t^{\prime}}(\partial T \setminus \partial \Omega)}^{t^{\prime}}.
\end{align*}
The following result establishes reliability properties of the error estimator $\eta_{\textrm{st},t^{\prime}}$.

\begin{proposition}[reliability of $\eta_{\textnormal{st},t^{\prime}}$]
\label{prop:rel_st_t'}
Suppose that assumptions \ref{ass:small_mesh}, \eqref{eq:assump_grady_small}, and \eqref{eq:assumption_y_12/5} hold.
Then, for $t'\in [2,4]$, we have
\begin{align}\label{eq:reliab_NS_Volker}
\|\mathbf{y}_{\bar{\boldsymbol{u}}_{\ell}} - \bar{\mathbf{y}}_{\ell}\|_{\mathbf{L}^{t^{\prime}}(\Omega)} \lesssim \eta_{\textnormal{st},t^{\prime}}.
\end{align}
\end{proposition}
\begin{proof}
    The case $t'=2$ was proved in \cite[Proposition 2]{MR1831796}. 
    For the case $t' \in (2,4]$, we let $\boldsymbol{\psi}\in \mathbf{L}^{t}(\Omega)$ and choose $\mathbf{w}=\mathbf{y}_{\bar{\boldsymbol{u}}_{\ell}} - \bar{\mathbf{y}}_{\ell}$ and $s=p_{\bar{\boldsymbol{u}}_{\ell}} - \bar{p}_{\ell}$ in \eqref{eq:aux_linear_NS} to obtain
\begin{align*}
\langle \boldsymbol{\psi}, \mathbf{y}_{\bar{\boldsymbol{u}}_{\ell}} - \bar{\mathbf{y}}_{\ell}\rangle = &\,
    \nu(\nabla (\mathbf{y}_{\bar{\boldsymbol{u}}_{\ell}} - \bar{\mathbf{y}}_{\ell}), \nabla \breve{\mathbf{z}})_{\mathbf{L}^2(\Omega)} + b(\mathbf{y}_{\bar{\boldsymbol{u}}_{\ell}};\mathbf{y}_{\bar{\boldsymbol{u}}_{\ell}} - \bar{\mathbf{y}}_{\ell},\breve{\mathbf{z}}) + b(\mathbf{y}_{\bar{\boldsymbol{u}}_{\ell}} - \bar{\mathbf{y}}_{\ell};\bar{\mathbf{y}}_{\ell},\breve{\mathbf{z}}) \\
 &\,   + (\breve{r},\textnormal{div} (\mathbf{y}_{\bar{\boldsymbol{u}}_{\ell}} - \bar{\mathbf{y}}_{\ell}))_{\Omega} -(p_{\bar{\boldsymbol{u}}_{\ell}} - \bar{p}_{\ell},\textnormal{div } \breve{\mathbf{z}})_{\Omega}.
\end{align*}
Let $\mathbf{I}_{\ell}: \mathbf{W}_0^{1,1}(\Omega) \to \mathbf{V}_{\ell}$ and $I_{\ell} : W^{1,1}(\Omega) \to Q_{\ell}$ be some suitable quasi-interpolation operators, e.g., the Scott-Zhang interpolant (see \cite{MR1011446} and \cite[Section 4.8]{MR2373954}).
Using the problem that $(\mathbf{y}_{\bar{\boldsymbol{u}}_{\ell}}, p_{\bar{\boldsymbol{u}}_{\ell}})$ solves in combination with Galerkin orthogonality gives
\begin{align*}
\langle \boldsymbol{\psi}, \mathbf{y}_{\bar{\boldsymbol{u}}_{\ell}} - \bar{\mathbf{y}}_{\ell}\rangle = &\,
    (\bar{\boldsymbol{u}}_{\ell}, \breve{\mathbf{z}} - \mathbf{I}_{\ell}\breve{\mathbf{z}})_{\mathbf{L}^{2}(\Omega)} -  \nu(\nabla\bar{\mathbf{y}}_{\ell}, \nabla (\breve{\mathbf{z}} - \mathbf{I}_{\ell}\breve{\mathbf{z}}))_{\mathbf{L}^2(\Omega)} - b(\bar{\mathbf{y}}_{\ell};\bar{\mathbf{y}}_{\ell},\breve{\mathbf{z}} - \mathbf{I}_{\ell}\breve{\mathbf{z}}) \\
 &\,   - (\breve{r} - I_{\ell}\breve{r},\textnormal{div }\bar{\mathbf{y}}_{\ell})_{\Omega} + (\bar{p}_{\ell},\textnormal{div}( \breve{\mathbf{z}} - \mathbf{I}_{\ell}\breve{\mathbf{z}}))_{\Omega}.
\end{align*}
From this identity and using element-wise integration by parts, standard interpolation results, the finite overlapping property of stars, and Proposition \ref{prop:wp_reg_breve_z} we arrive at
\[
\langle \boldsymbol{\psi}, \mathbf{y}_{\bar{\boldsymbol{u}}_{\ell}} - \bar{\mathbf{y}}_{\ell}\rangle 
\lesssim 
\eta_{st,t^{\prime}}(\|\breve{\mathbf{z}}\|_{\mathbf{W}^{2,t}(\Omega)} + \|\breve{r}\|_{W^{1,t}(\Omega)}) 
\lesssim 
\eta_{st,t^{\prime}}\|\boldsymbol{\psi}\|_{\mathbf{L}^{t}(\Omega)},
\]
with $1/t + 1/t' = 1$, i.e., $t\in[\tfrac{4}{3},2]$.
The latter implies that $\|\mathbf{y}_{\bar{\boldsymbol{u}}_{\ell}} - \bar{\mathbf{y}}_{\ell}\|_{\mathbf{L}^{t'}(\Omega)} = \sup_{\boldsymbol{\psi}\in \mathbf{L}^{t}(\Omega)}\|\boldsymbol{\psi}\|_{\mathbf{L}^{t}(\Omega)}^{-1}\langle \boldsymbol{\psi}, \mathbf{y}_{\bar{\boldsymbol{u}}_{\ell}} - \bar{\mathbf{y}}_{\ell}\rangle \lesssim \eta_{st,t^{\prime}}$, which concludes the proof.
\end{proof}


\subsection{A posteriori error analysis for the adjoint equations}
\label{sec:a_post_linear_NS}
Let $(\bar{\mathbf{y}}_{\ell},\bar{p}_{\ell})$ be the unique solution to \eqref{eq:discrete_state_eq_semi} with $\boldsymbol{u} = \bar{\boldsymbol{u}}_{\ell}$, which exists under assumption \ref{ass:small_mesh}.
We consider the problem: Given $\mathbf{G}\in \mathbf{H}^{-1}(\Omega)$ find $(\hat{\mathbf{z}},\hat{r}) \in \mathbf{H}_0^1(\Omega)\times L_0^2(\Omega)$ such that
\begin{align}
\label{eq:cont_adj_aux2}
\begin{split}
\nu(\nabla \mathbf{w}, \nabla \hat{\mathbf{z}})_{\mathbf{L}^2(\Omega)} + b(\bar{\mathbf{y}}_{\ell};\mathbf{w},\hat{\mathbf{z}}) + b(\mathbf{w};\bar{\mathbf{y}}_{\ell},\hat{\mathbf{z}}) -  (\hat{r},\textnormal{div } \mathbf{w})_{\Omega} &\, = \langle \mathbf{G}, \mathbf{w}\rangle \\
(s,\textnormal{div } \hat{\mathbf{z}})_{\Omega} &\,  =   0
\end{split}
\end{align}
for all $(\mathbf{w},s) \in \mathbf{H}_0^1(\Omega)\times L_0^2(\Omega)$. 
In order to guarantee that problem \eqref{eq:cont_adj_aux2} is well posed, we assume that \cite[assumption (32)]{MR4301394}
\begin{equation}\label{eq:ass_on_yl}
    2\|\nabla \bar{\mathbf{y}}_{\ell}\|_{\mathbf{L}^{2}(\Omega)} < \mathcal{C}_{b}^{-1}\nu.
\end{equation}

\begin{proposition}[well-posedness]\label{prop:reg_hatz}
Suppose that assumptions \ref{ass:small_mesh} and  \eqref{eq:ass_on_yl} hold. 
Then, the problem \eqref{eq:cont_adj_aux2} is well posed. 
Moreover, if $\mathbf{G}\in \mathbf{W}^{-1,\mathsf{p}}(\Omega)$ for some $\mathsf{p}>n$, then $\hat{\mathbf{z}}\in \mathbf{C}(\bar{\Omega})$.
\end{proposition}
\begin{proof}
The well-posedness follows by arguments analogous to those in the proof of Proposition \ref{prop:wp_reg_breve_z}.
Let us prove the regularity for the case $n=3$.
Define
\[
\mathcal{G} = \mathbf{G} - b(\bar{\mathbf{y}}_{\ell};\cdot,\hat{\mathbf{z}}) - b(\cdot;\bar{\mathbf{y}}_{\ell},\hat{\mathbf{z}}) \in \mathbf{H}^{-1}(\Omega)
\]
and write $\nu(\nabla \mathbf{w}, \nabla \hat{\mathbf{z}})_{\mathbf{L}^2(\Omega)} -  (\hat{r},\textnormal{div } \mathbf{w})_{\Omega} = \langle \mathcal{G}, \mathbf{w}\rangle_{\mathbf{H}^{-1}(\Omega), \mathbf{H}_0^{1}(\Omega)}$.
First, we prove that $\mathcal{G}\in \mathbf{W}^{-1,3}(\Omega)$. 
H\"older's inequality and the embedding $\mathbf{H}_0^{1}(\Omega)\hookrightarrow \mathbf{L}^{6}(\Omega)$ imply that
\begin{align*}
    &|b(\bar{\mathbf{y}}_{\ell};\mathbf{w},\hat{\mathbf{z}})| 
    \leq 
    \|\bar{\mathbf{y}}_{\ell}\|_{\mathbf{L}^{6}(\Omega)}\|\nabla \mathbf{w}\|_{\mathbf{L}^{3/2}(\Omega)}\|\hat{\mathbf{z}}\|_{\mathbf{L}^{6}(\Omega)}
    \lesssim
    \|\nabla \bar{\mathbf{y}}_{\ell}\|_{\mathbf{L}^{2}(\Omega)}\|\nabla \mathbf{w}\|_{\mathbf{L}^{3/2}(\Omega)}\|\nabla \hat{\mathbf{z}}\|_{\mathbf{L}^{2}(\Omega)},\\
  &  |b(\mathbf{w};\bar{\mathbf{y}}_{\ell},\hat{\mathbf{z}})|
    \leq 
    \|\mathbf{w}\|_{\mathbf{L}^{3}(\Omega)} \|\nabla \bar{\mathbf{y}}_{\ell}\|_{\mathbf{L}^{2}(\Omega)}\|\hat{\mathbf{z}}\|_{\mathbf{L}^{6}(\Omega)}
    \lesssim
    \|\nabla \mathbf{w}\|_{\mathbf{L}^{3/2}(\Omega)}\|\nabla \bar{\mathbf{y}}_{\ell}\|_{\mathbf{L}^{2}(\Omega)}\|\nabla \hat{\mathbf{z}}\|_{\mathbf{L}^{2}(\Omega)}.
\end{align*}
It follows that $\mathcal{G}\in \mathbf{W}^{-1,3}(\Omega)$, which, in view of \cite[Theorem 2.9]{Brown_Shen} (see also \cite[eq. (1.52)]{MR2987056}), shows that $\hat{\mathbf{z}}\in \mathbf{W}_0^{1,3}(\Omega)$.
We now use the previous regularity to prove that $\mathcal{G}\in \mathbf{W}^{-1,\mathsf{p}}(\Omega)$ for some $\mathsf{p}> 3$.
In fact, H\"older's inequality and the embedding $\mathbf{W}_0^{1,3}(\Omega)\hookrightarrow \mathbf{L}^{\mathsf{s}}(\Omega)$ with $\mathsf{s} < \infty$ yield that
\[
    |b(\bar{\mathbf{y}}_{\ell};\mathbf{w},\hat{\mathbf{z}})| 
    \leq 
    \|\bar{\mathbf{y}}_{\ell}\|_{\mathbf{L}^{6}(\Omega)}\|\nabla \mathbf{w}\|_{\mathbf{L}^{\mathsf{p}'}(\Omega)}\|\hat{\mathbf{z}}\|_{\mathbf{L}^{\mathfrak{m}}(\Omega)}
    \lesssim
    \|\nabla \bar{\mathbf{y}}_{\ell}\|_{\mathbf{L}^{2}(\Omega)}\|\nabla \mathbf{w}\|_{\mathbf{L}^{\mathsf{p}'}(\Omega)}\|\nabla \hat{\mathbf{z}}\|_{\mathbf{L}^{3}(\Omega)},
\]
with $1/\mathsf{p} + 1/\mathsf{p}' = 1$ and $\mathfrak{m} = 6\mathsf{p}'/(6\mathsf{p}'-6-\mathsf{p}')$. 
Similarly, we can prove that $|b(\mathbf{w};\bar{\mathbf{y}}_{\ell},\hat{\mathbf{z}})|\lesssim \|\nabla \mathbf{w}\|_{\mathbf{L}^{\mathsf{p}'}(\Omega)}\|\nabla \bar{\mathbf{y}}_{\ell}\|_{\mathbf{L}^{2}(\Omega)}\|\nabla \hat{\mathbf{z}}\|_{\mathbf{L}^{3}(\Omega)}$.
Therefore, $\mathcal{G}\in \mathbf{W}^{-1,\mathsf{p}}(\Omega)$ with $\mathsf{p}>3$.
Invoking again \cite[Theorem 2.9]{Brown_Shen} (see also \cite[eq. (1.52)]{MR2987056}), we see that $\hat{\mathbf{z}}\in \mathbf{W}_0^{1,\mathsf{p}}(\Omega)$.
We conclude in view of the embedding $\mathbf{W}_0^{1,\mathsf{p}}(\Omega)\hookrightarrow \mathbf{C}(\bar{\Omega})$.
\end{proof}

\begin{lemma}[auxiliary estimate I]\label{lemma:a_post_adj_hat}
Suppose that assumptions \ref{ass:small_mesh}, \eqref{eq:assump_grady_small}, \eqref{eq:assumption_y_12/5}, and \eqref{eq:ass_on_yl} hold. 
Let $(\mathbf{z}_{\bar{\boldsymbol{u}}_{\ell}},r_{\bar{\boldsymbol{u}}_{\ell}})\in \mathbf{H}_0^{1}(\Omega)\times L_0^{2}(\Omega)$ be the solution to \eqref{eq:aux_z_uh}, and let $(\hat{\mathbf{z}},\hat{r}) \in \mathbf{H}_0^1(\Omega)\times L_0^2(\Omega)$ denote the solution to \eqref{eq:cont_adj_aux2} with $\mathbf{G} = \bar{\mathbf{y}}_{\ell} - \mathbf{y}_{\Omega}$.
Then
\[
    \|\mathbf{z}_{\bar{\boldsymbol{u}}_{\ell}} - \hat{\mathbf{z}}\|_{\mathbf{L}^{\infty}(\Omega)} 
    \lesssim
    \eta_{\textrm{st},2} + \eta_{\textnormal{st},\mathsf{p}} + \|\textnormal{div }\bar{\mathbf{y}}_{\ell}\|_{\mathbf{L}^{\mu}(\Omega)} ,
\]
where $\mathsf{p}> n$ is arbitrarily close to $n$ and $\mu = \tfrac{n\mathsf{p}'}{n(\mathsf{p}'-1) + \mathsf{p}'} = \tfrac{n\mathsf{p}}{n+\mathsf{p}}>\tfrac{n}{2}$.
\end{lemma}
\begin{proof}
Note that $(\mathbf{z}_{\bar{\boldsymbol{u}}_{\ell}} - \hat{\mathbf{z}},r_{\bar{\boldsymbol{u}}_{\ell}} - \hat{r})\in \mathbf{H}_{0}^{1}(\Omega)\times L_0^{2}(\Omega)$ solves
\begin{align}\label{eq:problem_diff_z}
&\nu(\nabla \mathbf{w}, \nabla (\mathbf{z}_{\bar{\boldsymbol{u}}_{\ell}} - \hat{\mathbf{z}}))_{\mathbf{L}^2(\Omega)} + b(\mathbf{y}_{\bar{\boldsymbol{u}}_{\ell}};\mathbf{w},\mathbf{z}_{\bar{\boldsymbol{u}}_{\ell}} - \hat{\mathbf{z}}) + b(\mathbf{w};\mathbf{y}_{\bar{\boldsymbol{u}}_{\ell}},\mathbf{z}_{\bar{\boldsymbol{u}}_{\ell}} - \hat{\mathbf{z}}) \\
\nonumber
& -  (r_{\bar{\boldsymbol{u}}_{\ell}} - \hat{r},\text{div } \mathbf{w})_{\Omega} = (\mathbf{y}_{\bar{\boldsymbol{u}}_{\ell}} - \bar{\mathbf{y}}_{\ell},\mathbf{w})_{\mathbf{L}^2(\Omega)} - b(\mathbf{y}_{\bar{\boldsymbol{u}}_{\ell}} - \bar{\mathbf{y}}_{\ell};\mathbf{w},\hat{\mathbf{z}}) \\
& \hspace{3.8cm}
- b(\mathbf{w};\mathbf{y}_{\bar{\boldsymbol{u}}_{\ell}} - \bar{\mathbf{y}}_{\ell},\hat{\mathbf{z}}),   \qquad
(s,\text{div}(\mathbf{z}_{\bar{\boldsymbol{u}}_{\ell}} - \hat{\mathbf{z}}))_{\Omega}  =   0, \nonumber
\end{align}
for all $(\mathbf{w},s) \in \mathbf{H}_0^1(\Omega)\times L_0^2(\Omega)$.  
We observe that 
\[
|b(\mathbf{y}_{\bar{\boldsymbol{u}}_{\ell}} - \bar{\mathbf{y}}_{\ell};\mathbf{w},\hat{\mathbf{z}})|
\leq 
\|\mathbf{y}_{\bar{\boldsymbol{u}}_{\ell}} - \bar{\mathbf{y}}_{\ell}\|_{\mathbf{L}^{2}(\Omega)}\|\nabla \mathbf{w}\|_{\mathbf{L}^{2}(\Omega)}\|\hat{\mathbf{z}}\|_{\mathbf{L}^{\infty}(\Omega)}
\lesssim
\|\mathbf{y}_{\bar{\boldsymbol{u}}_{\ell}} - \bar{\mathbf{y}}_{\ell}\|_{\mathbf{L}^{2}(\Omega)}\|\nabla \mathbf{w}\|_{\mathbf{L}^{2}(\Omega)}.
\]
Similarly, since $\hat{\mathbf{z}}\in \mathbf{W}_0^{1,3}(\Omega)\cap \mathbf{L}^{\infty}(\Omega)$ (cf. Proposition \ref{prop:reg_hatz}), we have that 
\begin{align*}
&|b(\mathbf{w};\mathbf{y}_{\bar{\boldsymbol{u}}_{\ell}} - \bar{\mathbf{y}}_{\ell},\hat{\mathbf{z}})|
\leq 
|b(\mathbf{w};\hat{\mathbf{z}},\mathbf{y}_{\bar{\boldsymbol{u}}_{\ell}} - \bar{\mathbf{y}}_{\ell})|
+
|((\text{div }\mathbf{w})\hat{\mathbf{z}},\mathbf{y}_{\bar{\boldsymbol{u}}_{\ell}} - \bar{\mathbf{y}}_{\ell})_{\mathbf{L}^{2}(\Omega)}|\\
& \qquad \lesssim 
\|\mathbf{w}\|_{\mathbf{L}^{6}(\Omega)}\|\nabla\hat{\mathbf{z}}\|_{\mathbf{L}^{3}(\Omega)}\|\mathbf{y}_{\bar{\boldsymbol{u}}_{\ell}} - \bar{\mathbf{y}}_{\ell}\|_{\mathbf{L}^{2}(\Omega)} + \|\text{div }\mathbf{w}\|_{\mathbf{L}^{2}(\Omega)}\|\hat{\mathbf{z}}\|_{\mathbf{L}^{\infty}(\Omega)}\|\mathbf{y}_{\bar{\boldsymbol{u}}_{\ell}} - \bar{\mathbf{y}}_{\ell}\|_{\mathbf{L}^{2}(\Omega)}\\
& \qquad \lesssim \|\mathbf{y}_{\bar{\boldsymbol{u}}_{\ell}} - \bar{\mathbf{y}}_{\ell}\|_{\mathbf{L}^{2}(\Omega)} \|\nabla \mathbf{w}\|_{\mathbf{L}^{2}(\Omega)}.
\end{align*}
Moreover, $|(\mathbf{y}_{\bar{\boldsymbol{u}}_{\ell}} - \bar{\mathbf{y}}_{\ell}, \mathbf{w})_{\mathbf{L}^{2}(\Omega)}| \lesssim \|\mathbf{y}_{\bar{\boldsymbol{u}}_{\ell}} - \bar{\mathbf{y}}_{\ell}\|_{\mathbf{L}^{2}(\Omega}\|\nabla \mathbf{w}\|_{\mathbf{L}^{2}(\Omega)}$.
Therefore, since $\mathbf{y}_{\bar{\boldsymbol{u}}_{\ell}}$ is regular (cf. assumption \ref{ass:small_mesh}), Step 1 of the proof of Theorem 2.9 in \cite{MR3936891} implies the well-posedness of problem \eqref{eq:problem_diff_z}.
In particular, we infer that 
\begin{align}\label{eq:stab_estim_H1}
    \|\nabla (\mathbf{z}_{\bar{\boldsymbol{u}}_{\ell}} - \hat{\mathbf{z}})\|_{\mathbf{L}^{2}(\Omega)} 
    \lesssim 
    \|\mathbf{y}_{\bar{\boldsymbol{u}}_{\ell}} - \bar{\mathbf{y}}_{\ell}\|_{\mathbf{L}^{2}(\Omega)}.
\end{align}

We now prove that $\mathbf{z}_{\bar{\boldsymbol{u}}_{\ell}} - \hat{\mathbf{z}}\in \mathbf{W}_0^{1,\mathsf{p}}(\Omega)$ for some $\mathsf{p}>n$ sufficiently close to $n$, with a stability estimate similar to \eqref{eq:stab_estim_H1}.
To this end, we proceed as in the proof of Proposition \ref{prop:reg_hatz} and write the first equation in \eqref{eq:problem_diff_z} as 
\begin{align*}
&\nu(\nabla \mathbf{w}, \nabla (\mathbf{z}_{\bar{\boldsymbol{u}}_{\ell}} - \hat{\mathbf{z}}))_{\mathbf{L}^2(\Omega)}  - (r_{\bar{\boldsymbol{u}}_{\ell}} - \hat{r},\text{div } \mathbf{w})_{\Omega} = - b(\mathbf{w};\mathbf{y}_{\bar{\boldsymbol{u}}_{\ell}},\mathbf{z}_{\bar{\boldsymbol{u}}_{\ell}} - \hat{\mathbf{z}}) \\
&  - b(\mathbf{y}_{\bar{\boldsymbol{u}}_{\ell}};\mathbf{w},\mathbf{z}_{\bar{\boldsymbol{u}}_{\ell}} - \hat{\mathbf{z}}) + (\mathbf{y}_{\bar{\boldsymbol{u}}_{\ell}} - \bar{\mathbf{y}}_{\ell},\mathbf{w})_{\mathbf{L}^2(\Omega)} + b(\mathbf{y}_{\bar{\boldsymbol{u}}_{\ell}} - \bar{\mathbf{y}}_{\ell};\mathbf{w},\hat{\mathbf{z}}) + b(\mathbf{w};\mathbf{y}_{\bar{\boldsymbol{u}}_{\ell}} - \bar{\mathbf{y}}_{\ell},\hat{\mathbf{z}}).
\end{align*}
We shall estimate each term on the right-hand side of this equation in $\mathbf{W}^{-1,\mathsf{p}}(\Omega)$. 
H\"older's inequality and the fact that $\mathbf{y}_{\bar{\boldsymbol{u}}_{\ell}} \in \mathbf{H}^{2}(\Omega)$ imply that
\begin{align*}
|b(\mathbf{w};\mathbf{y}_{\bar{\boldsymbol{u}}_{\ell}},\mathbf{z}_{\bar{\boldsymbol{u}}_{\ell}} - \hat{\mathbf{z}})|
 & \leq 
\|\mathbf{w}\|_{\mathbf{L}^{2}(\Omega)}\|\nabla \mathbf{y}_{\bar{\boldsymbol{u}}_{\ell}}\|_{\mathbf{L}^{3}(\Omega)}\|\mathbf{z}_{\bar{\boldsymbol{u}}_{\ell}} - \hat{\mathbf{z}}\|_{\mathbf{L}^6(\Omega)} \\
& \lesssim \|\nabla\mathbf{w}\|_{\mathbf{L}^{\mathsf{p}'}(\Omega)}\|\nabla \mathbf{y}_{\bar{\boldsymbol{u}}_{\ell}}\|_{\mathbf{L}^{3}(\Omega)}\|\nabla(\mathbf{z}_{\bar{\boldsymbol{u}}_{\ell}} - \hat{\mathbf{z}})\|_{\mathbf{L}^2(\Omega)}\\
& \lesssim \|\nabla\mathbf{w}\|_{\mathbf{L}^{\mathsf{p}'}(\Omega)}\|\mathbf{y}_{\bar{\boldsymbol{u}}_{\ell}} - \bar{\mathbf{y}}_{\ell}\|_{\mathbf{L}^{2}(\Omega)}.
\end{align*}
Similarly, since div $\mathbf{y}_{\bar{\boldsymbol{u}}_{\ell}}=0$, we have 
\begin{align*}
|b(\mathbf{y}_{\bar{\boldsymbol{u}}_{\ell}};\mathbf{w},\mathbf{z}_{\bar{\boldsymbol{u}}_{\ell}} - \hat{\mathbf{z}})|
= 
|b(\mathbf{y}_{\bar{\boldsymbol{u}}_{\ell}};\mathbf{z}_{\bar{\boldsymbol{u}}_{\ell}} - \hat{\mathbf{z}},\mathbf{w})|
\leq & ~
\|\mathbf{y}_{\bar{\boldsymbol{u}}_{\ell}}\|_{\mathbf{L}^{\infty}(\Omega)}\|\nabla(\mathbf{z}_{\bar{\boldsymbol{u}}_{\ell}} - \hat{\mathbf{z}})\|_{\mathbf{L}^{2}(\Omega)}\|\mathbf{w}\|_{\mathbf{L}^{2}(\Omega)} \\
\lesssim & ~ \|\nabla\mathbf{w}\|_{\mathbf{L}^{\mathsf{p}'}(\Omega)}\|\mathbf{y}_{\bar{\boldsymbol{u}}_{\ell}} - \bar{\mathbf{y}}_{\ell}\|_{\mathbf{L}^{2}(\Omega)}.
\end{align*}
The embedding $\mathbf{W}_0^{1,\mathsf{p}'}(\Omega) \hookrightarrow \mathbf{L}^{2}(\Omega)$ ($\mathsf{p}' < \tfrac{n}{n-1}$ sufficiently close to $\tfrac{n}{n-1}$) implies that 
\[
   |(\mathbf{y}_{\bar{\boldsymbol{u}}_{\ell}} - \bar{\mathbf{y}}_{\ell},\mathbf{w})_{\mathbf{L}^2(\Omega)}|
   \lesssim
   \|\mathbf{y}_{\bar{\boldsymbol{u}}_{\ell}} - \bar{\mathbf{y}}_{\ell}\|_{\mathbf{L}^{2}(\Omega)}\|\nabla \mathbf{w}\|_{\mathbf{L}^{\mathsf{p}'}(\Omega)}.
\]
Using again that div $\mathbf{y}_{\bar{\boldsymbol{u}}_{\ell}}=0$ we obtain
\begin{align*}
 & |b(\mathbf{y}_{\bar{\boldsymbol{u}}_{\ell}} - \bar{\mathbf{y}}_{\ell};\mathbf{w},\hat{\mathbf{z}})| 
\leq 
|b(\mathbf{y}_{\bar{\boldsymbol{u}}_{\ell}} - \bar{\mathbf{y}}_{\ell};\hat{\mathbf{z}},\mathbf{w})| +
|((\text{div }\bar{\mathbf{y}}_{\ell})\hat{\mathbf{z}},\mathbf{w})_{\mathbf{L}^{2}(\Omega)}| \\
& \quad \lesssim  
\left(\|\mathbf{y}_{\bar{\boldsymbol{u}}_{\ell}} - \bar{\mathbf{y}}_{\ell}\|_{\mathbf{L}^{n}(\Omega)}\|\nabla \hat{\mathbf{z}}\|_{\mathbf{L}^{\mathsf{p}}(\Omega)}
+
\|\text{div }\bar{\mathbf{y}}_{\ell}\|_{\mathbf{L}^{\mu}}\|\hat{\mathbf{z}}\|_{\mathbf{L}^{\infty}(\Omega)}\right)\|\mathbf{w}\|_{\mathbf{L}^{n\mathsf{p}\prime/(n - \mathsf{p}\prime)}(\Omega)}\\
& \quad \lesssim  (\|\mathbf{y}_{\bar{\boldsymbol{u}}_{\ell}} - \bar{\mathbf{y}}_{\ell}\|_{\mathbf{L}^{n}(\Omega)} + \|\text{div }\bar{\mathbf{y}}_{\ell}\|_{\mathbf{L}^{\mu}(\Omega)})\|\nabla \mathbf{w}\|_{\mathbf{L}^{\mathsf{p}'}(\Omega)},
\end{align*}
where $\mu = \tfrac{n\mathsf{p}'}{n(\mathsf{p}'-1) +\mathsf{p}'} = \tfrac{n\mathsf{p}}{n+\mathsf{p}} > \tfrac{n}{2}$.
The remaining term is estimated as follows:
\begin{align*}
&|b(\mathbf{w};\mathbf{y}_{\bar{\boldsymbol{u}}_{\ell}} - \bar{\mathbf{y}}_{\ell},\hat{\mathbf{z}})|
\leq 
|b(\mathbf{w};\hat{\mathbf{z}},\mathbf{y}_{\bar{\boldsymbol{u}}_{\ell}} - \bar{\mathbf{y}}_{\ell})|
+
|((\text{div }\mathbf{w})\hat{\mathbf{z}},\mathbf{y}_{\bar{\boldsymbol{u}}_{\ell}} - \bar{\mathbf{y}}_{\ell})_{\mathbf{L}^{2}(\Omega)}|\\
& \lesssim 
\|\mathbf{w}\|_{\mathbf{L}^{n\mathsf{p}\prime/(n - \mathsf{p}\prime)}(\Omega)}\|\nabla\hat{\mathbf{z}}\|_{\mathbf{L}^{\mathsf{p}}(\Omega)}\|\mathbf{y}_{\bar{\boldsymbol{u}}_{\ell}} - \bar{\mathbf{y}}_{\ell}\|_{\mathbf{L}^{n}(\Omega)}\\
& \quad + \|\text{div }\mathbf{w}\|_{\mathbf{L}^{\mathsf{p}'}(\Omega)}\|\hat{\mathbf{z}}\|_{\mathbf{L}^{\infty}(\Omega)}\|\mathbf{y}_{\bar{\boldsymbol{u}}_{\ell}} - \bar{\mathbf{y}}_{\ell}\|_{\mathbf{L}^{\mathsf{p}}(\Omega)}
\lesssim \|\mathbf{y}_{\bar{\boldsymbol{u}}_{\ell}} - \bar{\mathbf{y}}_{\ell}\|_{\mathbf{L}^{\mathsf{p}}(\Omega)} \|\nabla \mathbf{w}\|_{\mathbf{L}^{\mathsf{p}'}(\Omega)}.
\end{align*}
Therefore, the use of \cite[Theorem 2.9]{Brown_Shen}  (see also \cite[eq. (1.52)]{MR2987056}) results in 
\[
\|\nabla(\mathbf{z}_{\bar{\boldsymbol{u}}_{\ell}} - \hat{\mathbf{z}})\|_{\mathbf{L}^{\mathsf{p}}(\Omega)}
\lesssim
\|\mathbf{y}_{\bar{\boldsymbol{u}}_{\ell}} - \bar{\mathbf{y}}_{\ell}\|_{\mathbf{L}^{2}(\Omega)} + \|\text{div }\bar{\mathbf{y}}_{\ell}\|_{\mathbf{L}^{\mu}(\Omega)} + \|\mathbf{y}_{\bar{\boldsymbol{u}}_{\ell}} - \bar{\mathbf{y}}_{\ell}\|_{\mathbf{L}^{\mathsf{p}}(\Omega)}.
\]
Consequently, in view of the embedding $\mathbf{W}_0^{1,\mathsf{p}}(\Omega)\hookrightarrow \mathbf{C}(\bar{\Omega})$, we conclude that 
\[
\|\mathbf{z}_{\bar{\boldsymbol{u}}_{\ell}} - \hat{\mathbf{z}}\|_{\mathbf{L}^{\infty}(\Omega)}
\lesssim
\|\mathbf{y}_{\bar{\boldsymbol{u}}_{\ell}} - \bar{\mathbf{y}}_{\ell}\|_{\mathbf{L}^{2}(\Omega)} + \|\text{div }\bar{\mathbf{y}}_{\ell}\|_{\mathbf{L}^{\mu}(\Omega)} + \|\mathbf{y}_{\bar{\boldsymbol{u}}_{\ell}} - \bar{\mathbf{y}}_{\ell}\|_{\mathbf{L}^{\mathsf{p}}(\Omega)}.
\]
We obtain the desired estimate using, in the latter bound, the error estimate \eqref{eq:reliab_NS_Volker} with $t' = 2$ and $t' = \mathsf{p}$.
\end{proof}

Let $(\bar{\mathbf{z}}_{\ell},\bar{r}_{\ell}) \in \mathbf{V}_{\ell}\times Q_{\ell}$ be the unique solution to \eqref{eq:discrete_adjoint_equation}.
For every $T\in\mathcal{T}_{\ell}$, we consider the local error indicators 
\begin{align*}
\eta_{\textrm{adj},T}
:= & \, 
h_T^2\|\bar{\mathbf{y}}_{\ell} - \mathbf{y}_\Omega + \nu\Delta\bar{\mathbf{z}}_{\ell}-(\nabla \bar{\mathbf{y}}_{\ell})^{\intercal}\bar{\mathbf{z}}_{\ell} + (\bar{\mathbf{y}}_{\ell}\cdot \nabla)\bar{\mathbf{z}}_{\ell} + \text{div}(\bar{\mathbf{y}}_{\ell})\bar{\mathbf{z}}_{\ell} - \nabla \bar{r}_{\ell}\|_{\mathbf{L}^\infty(T)}\\
& \, + h_{T}\|\text{div }\bar{\mathbf{z}}_{\ell}\|_{L^\infty(T)} + \frac{1}{2} h_T\|\llbracket (\nu\nabla \bar{\mathbf{z}}_{\ell} + (\bar{\mathbf{y}}_{\ell} \otimes \bar{\mathbf{z}}_{\ell}) - \bar{r}_{\ell}\mathbb{I}_{n} )\cdot \mathbf{n}\rrbracket\|_{\mathbf{L}^\infty(\partial T \setminus \partial \Omega)},
\end{align*}
and the a posteriori error estimator $\eta_{\textrm{adj},\infty}:= \max_{T\in\mathcal{T}_{\ell}}\eta_{\textrm{adj},T}$.

\begin{lemma}[auxiliary estimate II]\label{lemma:a_post_adj_hat_discrete}
Suppose that assumptions \ref{ass:small_mesh}, \eqref{eq:assump_grady_small}, and \eqref{eq:ass_on_yl} hold. 
Let $(\hat{\mathbf{z}},\hat{r}) \in \mathbf{H}_0^1(\Omega)\!\times\! L_0^2(\Omega)$ be the solution to \eqref{eq:cont_adj_aux2} with $\mathbf{G} = \bar{\mathbf{y}}_{\ell} \!-\! \mathbf{y}_{\Omega}$. Then
\[
\|\hat{\mathbf{z}} - \bar{\mathbf{z}}_{\ell}\|_{\mathbf{L}^{\infty}(\Omega)} \lesssim |\log h_{\min}|^{4/n}\eta_{\textnormal{adj},\infty}.
\]
\end{lemma}
\begin{proof}
Define $\mathbf{F} :=  b(\bar{\mathbf{y}}_{\ell};\cdot,\hat{\mathbf{z}} - \bar{\mathbf{z}}_{\ell}) + b(\cdot;\bar{\mathbf{y}}_{\ell},\hat{\mathbf{z}} - \bar{\mathbf{z}}_{\ell}) \in \mathbf{H}^{-1}(\Omega)$ and the pair $(\boldsymbol{\Psi}, \xi)\in \mathbf{H}_0^{1}(\Omega)\times L_0^{2}(\Omega)$ as the unique solution to
\[
\nu(\nabla \mathbf{w}, \nabla \boldsymbol{\Psi})_{\mathbf{L}^2(\Omega)} - (\xi,\text{div } \mathbf{w})_{\Omega}  = \langle \mathbf{F}, \mathbf{w}\rangle, 
\qquad
(s,\text{div } \boldsymbol{\Psi})_{\Omega}  =   0, 
\]
for all $(\mathbf{w},s) \in \mathbf{H}_0^1(\Omega)\times L_0^2(\Omega)$.
With $\boldsymbol{\Psi}$ at hand, we introduce the pair $(\mathbf{K},k):= (\hat{\mathbf{z}} - \bar{\mathbf{z}}_{\ell} + \boldsymbol{\Psi}, \hat{r} - \bar{r}_\ell + \xi)$ and observe that
\begin{align*}
\nu(\nabla \mathbf{w}, \nabla \boldsymbol{\Psi})_{\mathbf{L}^2(\Omega)} \!+\! b(\bar{\mathbf{y}}_{\ell};\mathbf{w},\boldsymbol{\Psi}) \!+\! b(\mathbf{w};\bar{\mathbf{y}}_{\ell},\boldsymbol{\Psi}) \!-\! (\xi,\text{div } \mathbf{w})_{\Omega} & = b(\bar{\mathbf{y}}_{\ell};\mathbf{w},\mathbf{K}) \!+\! b(\mathbf{w};\bar{\mathbf{y}}_{\ell},\mathbf{K}) \\
(s,\text{div } \boldsymbol{\Psi})_{\Omega} & =   0
\end{align*}
for all $(\mathbf{w},s) \in \mathbf{H}_0^1(\Omega)\times L_0^2(\Omega)$.
In particular, the pair $(\boldsymbol{\Psi}, \xi)$ solves problem \eqref{eq:cont_adj_aux2} with $\mathbf{G} = b(\bar{\mathbf{y}}_{\ell};\cdot,\mathbf{K}) + b(\cdot;\bar{\mathbf{y}}_{\ell},\mathbf{K})\in \mathbf{W}^{-1,\mathsf{p}}(\Omega)$ for some $\mathsf{p}>n$ close enough to $n$.
Consequently, in light of Proposition \ref{prop:reg_hatz}, we infer that $\|\boldsymbol{\Psi}\|_{\mathbf{L}^{\infty}(\Omega)}  \lesssim  \|\nabla \boldsymbol{\Psi}\|_{\mathbf{L}^{\mathsf{p}}(\Omega)} \lesssim   \|\mathbf{K}\|_{\mathbf{L}^{\infty}(\Omega)}$, from which it follows that
\begin{equation}\label{eq:hat_z_bar_zl_k}
    \|\hat{\mathbf{z}} - \bar{\mathbf{z}}_{\ell}\|_{\mathbf{L}^{\infty}(\Omega)}
    \lesssim
    \|\mathbf{K} - \boldsymbol{\Psi}\|_{\mathbf{L}^{\infty}(\Omega)}
    \lesssim
    \|\mathbf{K}\|_{\mathbf{L}^{\infty}(\Omega)}.
\end{equation}

We devote the remainder of this proof to estimating the term $\|\mathbf{K}\|_{\mathbf{L}^{\infty}(\Omega)}$.
To accomplish this task, we let $x_0\in \Omega$ and $i\in\{1,\ldots,n\}$ be such that $\|\mathbf{K}\|_{\mathbf{L}^{\infty}(\Omega)} = |\mathbf{K}_{i}(x_0)|$.
We now follow \cite[Section 2]{Larsson_Svensson} and let $\delta = \delta_{x_0,\rho/2}$ be a regularization of the Dirac distribution at $x_0$.
More precisely, $\delta$ satisfies \cite[eq. (2.9)]{Larsson_Svensson}
\[
    \text{supp}(\delta)\subset B_{x_0}(\tfrac{\rho}{2}), \qquad \int_{\mathbb{R}^{n}}\delta = 1, \qquad 0 \leq \delta \lesssim \rho^{-n},
\]
where $B_{x_0}(\tfrac{\rho}{2})$ denotes the ball with center in $x_0$ and radius $\tfrac{\rho}{2}$ chosen such that $\rho \leq h_{\text{min}}^{\tau}$ where $\tau > 0$ will be specified later.
It also satisfies $\|\delta\|_{W^{l,q}(\Omega)} \lesssim \rho^{-n(1-1/q)-l}$; see \cite[eq. (2.11)]{Larsson_Svensson}.
Since $\mathbf{K}\in \mathbf{C}(\bar{\Omega})$, the mean value theorem reveals that there exists $x_1 \in B_{x_0}(\tfrac{\rho}{2})$ such that $(\mathbf{K}_{i},\delta)_{\Omega}= \mathbf{K}_{i}(x_1)$.
Hence,
\begin{equation}\label{eq:error_e_I_II}
    \|\mathbf{K}\|_{\mathbf{L}^{\infty}(\Omega)}
    \leq
    |\mathbf{K}_{i}(x_0) - \mathbf{K}_{i}(x_1)| + |(\mathbf{K}_{i},\delta)_{\Omega}|=:\mathrm{I} + \mathrm{II}.
\end{equation}
Using the fact that $\mathbf{K}\in \mathbf{W}_0^{1,\mathsf{p}}(\Omega)\hookrightarrow \mathbf{C}^{0,\gamma}(\bar{\Omega})$ with $\gamma \in [0,1 - n/\mathsf{p}]$, we have that
\[
    \mathrm{I}
    \lesssim \rho^{\gamma}\|\mathbf{K}_{i}\|_{C^{0,\gamma}(B_{x_0}(\frac{\rho}{2})\cap \Omega)} \lesssim \rho^{\gamma}\|\nabla \mathbf{K}\|_{\mathbf{L}^{\mathsf{p}}(\Omega)}.
\]
A direct computation reveals that $(\mathbf{K},k)\in \mathbf{W}_0^{1,\mathsf{p}}(\Omega)\times L^{\mathsf{p}}(\Omega)/\mathbb{R}$ ($\mathsf{p}>n$) solves
\begin{equation}\label{eq:problemKk}
- \nu \Delta \mathbf{K} + \nabla k = \mathbf{R}_{\textrm{adj}}, \qquad \text{div } \mathbf{K} =  -\text{div } \bar{\mathbf{z}}_{\ell},
\end{equation}
where $
    \mathbf{R}_{\textrm{adj}} := \bar{\mathbf{y}}_{\ell} - \mathbf{y}_\Omega + \nu\Delta\bar{\mathbf{z}}_{\ell}-(\nabla \bar{\mathbf{y}}_{\ell})^{\intercal}\bar{\mathbf{z}}_{\ell} + (\bar{\mathbf{y}}_{\ell}\cdot \nabla)\bar{\mathbf{z}}_{\ell} + \text{div}(\bar{\mathbf{y}}_{\ell})\bar{\mathbf{z}}_{\ell} - \nabla \bar{r}_{\ell}.
$
Using that $\mathbf{R}_{\textrm{adj}} \in \mathbf{W}^{-1,\mathsf{p}}(\Omega)$ and $\int_{\Omega} \text{div }\bar{\mathbf{z}}_{\ell}=0$, we conclude, in view of \cite[Theorem 2.9]{Brown_Shen} that 
\begin{align*}
    \|\nabla \mathbf{K}\|_{\mathbf{L}^{\mathsf{p}}(\Omega)} + \|k\|_{L^{\mathsf{p}}(\Omega)} 
    \lesssim
    \|\mathbf{R}_{\textrm{adj}}\|_{\mathbf{W}^{-1,\mathsf{p}}(\Omega)} + \|\text{div }\bar{\mathbf{z}}_{\ell}\|_{L^{\mathsf{p}}(\Omega)}.
\end{align*} 
We observe that $\|\mathbf{R}_{\textrm{adj}}\|_{\mathbf{W}^{-1,\mathsf{p}}(\Omega)} \lesssim \|\mathbf{R}_{\textrm{adj}}\|_{\mathbf{W}^{-1,\infty}(\Omega)}$. 
To estimate $\|\mathbf{R}_{\textrm{adj}}\|_{\mathbf{W}^{-1,\infty}(\Omega)}$, we use Galerkin orthogonality, element-wise integration by parts, standard interpolation results, and the definition of $\eta_{\textrm{adj},\infty}$. 
These arguments yield
\begin{align*}
     & \|\mathbf{R}_{\textrm{adj}}\|_{\mathbf{W}^{-1,\infty}(\Omega)} = \sup_{\mathbf{w}\in \mathbf{W}_0^{1,1}(\Omega)}\frac{\langle \mathbf{R}_{\textrm{adj}}, \mathbf{w}\rangle }{\|\nabla \mathbf{w}\|_{\mathbf{L}^{1}(\Omega)}} \\
     & \lesssim  
      \sup_{\mathbf{w}\in \mathbf{W}_0^{1,1}(\Omega)}\frac{\sum_{T\in \mathcal{T}_{\ell}} \left(\|\mathbf{R}_{\textrm{adj}}|_{T}\|_{\mathbf{L}^{\infty}(T)} + \frac{h_{T}^{-1}}{2} \|\mathbf{J}_{\textrm{adj}}|_{T}\|_{\mathbf{L}^\infty(\partial T \setminus \partial \Omega)}\right)\|\mathbf{w} - \mathbf{I}_{\ell}\mathbf{w}\|_{\mathbf{L}^{1}(T)}}{\|\nabla \mathbf{w}\|_{\mathbf{L}^{1}(\Omega)}}\\
    & \lesssim \max_{T\in \mathcal{T}_{\ell}}\frac{\eta_{\textrm{adj},T}}{h_{T}}
    \lesssim
    h_{\text{min}}^{-1}\eta_{\textrm{adj},\infty}.
\end{align*}
Here, $\mathbf{J}_{\textrm{adj}}:=\llbracket (\nu\nabla \bar{\mathbf{z}}_{\ell} + (\bar{\mathbf{y}}_{\ell} \otimes \bar{\mathbf{z}}_{\ell}) - \bar{r}_{\ell}\mathbb{I}_{d} )\cdot \mathbf{n}\rrbracket$ and $\mathbf{I}_{\ell}: \mathbf{W}_0^{1,1}(\Omega) \to \mathbf{V}_{\ell}$ denotes some suitable quasi-interpolation operator, for example, the Scott-Zhang interpolant (\cite{MR1011446} and \cite[Section 4.8]{MR2373954}).
The latter, in combination with the fact that $\|\text{div }\bar{\mathbf{z}}_{\ell}\|_{L^{\mathsf{p}}(\Omega)} \lesssim h_\text{min}^{-1}\eta_{\textrm{adj},\infty}$, results in $\mathrm{I} \lesssim \rho^{\gamma} h_{\text{min}}^{-1}\eta_{\textrm{adj},\infty}$.
Hence, taking $\tau>0$ such that $\gamma\tau=1$, we conclude that $\mathrm{I} \leq \eta_{\textrm{adj},\infty}$.

To bound the term $\mathrm{II}$ in \eqref{eq:error_e_I_II}, we proceed as in \cite[Lemma 2.1]{Larsson_Svensson}.
First, we define $(\boldsymbol{\psi},\mathfrak{r})\in \mathbf{H}_0^1(\Omega)\times L_0^2(\Omega)$ as the unique solution to the dual Stokes problem
\begin{equation}\label{eq:dual_problemKk}
    \nu(\nabla \boldsymbol{\psi}, \nabla\mathbf{v})_{\mathbf{L}^2(\Omega)}  + (\mathfrak{r},\text{div } \mathbf{v})_{\Omega} = (\delta e_{i}, \mathbf{v})_{\mathbf{L}^{2}(\Omega)}, \quad -(\text{div } \boldsymbol{\psi}, q)_{\Omega} = 0,
\end{equation}
for all $(\mathbf{v},q) \in \mathbf{H}_0^1(\Omega)\times L_0^2(\Omega)$, where $e_i$ denotes the $i$-th canonical unit vector.
Then, in view of \eqref{eq:dual_problemKk}, \eqref{eq:problemKk}, and Galerkin orthogonality, we observe that 
\begin{align*}
    \mathrm{II} 
    & \, =
    |\nu(\nabla \boldsymbol{\psi}, \nabla\mathbf{K})_{\mathbf{L}^2(\Omega)}  + (\mathfrak{r},\text{div } \mathbf{K})_{\Omega} - (\text{div } \boldsymbol{\psi}, k)_{\Omega}|,\\
    &\, = |\langle \mathbf{R}_{\textrm{adj}}, \boldsymbol{\psi}\rangle - (\text{div } \bar{\mathbf{z}}_{\ell}, \mathfrak{r})_{\Omega}|
    = |\langle \mathbf{R}_{\textrm{adj}}, \boldsymbol{\psi} - \mathbf{I}_{\ell}\boldsymbol{\psi}\rangle - (\text{div } \bar{\mathbf{z}}_{\ell}, \mathfrak{r} - I_{\ell}\mathfrak{r})_{\Omega}|,
\end{align*}
where $I_{\ell} : W^{1,1}(\Omega) \to Q_{\ell}$ denotes a suitable quasi-interpolation operator. 
This estimate, element-wise integration by parts, and standard interpolation estimates, give
\[
    \mathrm{II} \lesssim 
    \eta_{\textrm{adj},\infty}(\|\boldsymbol{\psi}\|_{\mathbf{W}^{2,1}(\Omega)} + \|\mathfrak{r}\|_{W^{1,1}(\Omega)}).
\]
Finally, we use that $\|\boldsymbol{\psi}\|_{\mathbf{W}^{2,1}(\Omega)} + \|\mathfrak{r}\|_{W^{1,1}(\Omega)}\lesssim |\log h_{\min}|^{4/n}$ (see \cite[Corollary 3.5]{Larsson_Svensson}) to arrive at $\mathrm{II} \lesssim |\log h_{\min}|^{4/n}\eta_{\textrm{adj},\infty}$.

Therefore, we have proved that $\mathrm{I} + \mathrm{II} \lesssim |\log h_{\min}|^{4/n}\eta_{\textrm{adj},\infty}$, which, in light of \eqref{eq:error_e_I_II} and \eqref{eq:hat_z_bar_zl_k}, concludes the proof.
\end{proof}


\subsection{A posteriori error estimate: optimal control problem}

\begin{theorem}[reliability estimate]\label{thm:rel_ocp}
Let $\bar{\mathbf{u}}\in \mathbf{U}_{ad}$ be a local nonsingular solution to \eqref{eq:minimize_cost_func} with $(\bar{\mathbf{y}},\bar{p})$ and $(\bar{\mathbf{z}},\bar{r})$ being the corresponding state and adjoint state variables, respectively. 
Suppose that assumptions \ref{ass:small_mesh}, \eqref{eq:assump_grady_small}, \eqref{eq:assumption_y_12/5}, and \eqref{eq:ass_on_yl} hold. 
Let $\bar{\boldsymbol{u}}_{\ell}\in \mathcal{O}(\bar{\mathbf{u}})$ be a solution to the semidiscrete optimal control problem provided by assumption \ref{ass:small_mesh} with $(\bar{\mathbf{y}}_{\ell},\bar{p}_\ell)$ and $(\bar{\mathbf{z}}_{\ell},\bar{r}_\ell)$ being the corresponding discrete state and discrete adjoint state variables, respectively.
If assumption \ref{controlgrowth} holds, then
\begin{align*}
    \|\bar{\mathbf{u}} - \bar{\boldsymbol{u}}_{\ell}\|_{\mathbf{L}^{1}(\Omega)}
    \lesssim
    (\eta_{\textrm{st},2} + \eta_{\textrm{st},\mathsf{p}} + |\log h_{\min}|^{4/n}\eta_{\textrm{adj},\infty} +  \|\textnormal{div }\bar{\mathbf{y}}_{\ell}\|_{\mathbf{L}^{\mu}(\Omega)})^{\gamma},
\end{align*}
where $\mathsf{p}> n$ is arbitrarily close to $n$, $\mu = \tfrac{n\mathsf{p}}{n+\mathsf{p}}>\tfrac{n}{2}$, and $\gamma \in (n/(n+2), 1]$.
\end{theorem}
\begin{proof}
Arguments similar to those leading to \eqref{eq:u-uh--z-zh} in Theorem \ref{thm:estimate_control} yield
\[
\|\bar{\boldsymbol{u}}_{\ell} - \bar{\mathbf{u}}\|_{\mathbf{L}^1(\Omega)}^{1+\frac{1}{\gamma}} 
\lesssim
(\mathbf{z}_{\bar{\boldsymbol{u}}_{\ell}} - \bar{\mathbf{z}}_{\ell},\bar{\boldsymbol{u}}_{\ell} - \bar{\mathbf{u}})_{\mathbf{L}^{2}(\Omega)},
\]
where $(\mathbf{z}_{\bar{\boldsymbol{u}}_{\ell}},r_{\bar{\boldsymbol{u}}_{\ell}}) \in \mathbf{H}_0^1(\Omega)\times L_0^{2}(\Omega)$ denotes the unique solution to problem \eqref{eq:aux_z_uh}. 
Invoke $(\hat{\mathbf{z}},\hat{r}) \in \mathbf{H}_0^1(\Omega)\times L_0^2(\Omega)$, the unique solution to \eqref{eq:cont_adj_aux2} with $\mathbf{G} = \bar{\mathbf{y}}_{\ell} - \mathbf{y}_{\Omega}$, and the triangle inequality to arrive at
\[
\|\bar{\boldsymbol{u}}_{\ell} - \bar{\mathbf{u}}\|_{\mathbf{L}^1(\Omega)}^{\frac{1}{\gamma}} 
\lesssim
\|\mathbf{z}_{\bar{\boldsymbol{u}}_{\ell}} - \hat{\mathbf{z}}\|_{\mathbf{L}^{\infty}(\Omega)} + \|\hat{\mathbf{z}} - \bar{\mathbf{z}}_{\ell}\|_{\mathbf{L}^{\infty}(\Omega)}.
\]
A direct application of Lemmas \ref{lemma:a_post_adj_hat} and \ref{lemma:a_post_adj_hat_discrete} concludes the proof.
\end{proof}

\section*{Appendix}
In the present section, we give an answer to whether a condition similar to \eqref{eq:cond_DH} can be formulated when $\gamma<1$. 
To accomplish this, we need the following two lemmas.

\begin{lemma}
Let $\phi\in W^{1,\infty}(\Omega)$ and $\gamma\in(0,1]$.
Assume that there exist constants $C>0$ and $\delta>0$ such that for a.a. $t\in(-\delta,\delta)\setminus\{0\}$.
\begin{equation}\label{eq_levelsetest}
\int_{\{\phi=t\}} \frac{1}{|\nabla \phi|}
\,d\mathcal H^{n-1}
\le C |t|^{\gamma-1}
\end{equation}
Assume, in addition, that $|[0<|\phi|<\delta,\ |\nabla\phi|=0]|=0$. Then, for every $0<\varepsilon<\delta$,
\[
|[0<|\phi|<\varepsilon]|
\le \frac{2C}{\gamma}\varepsilon^\gamma.
\]
If further $|[\phi=0]|=0$, then $|[|\phi|\le \varepsilon]|
\le \frac{2C}{\gamma}\varepsilon^\gamma$.
\end{lemma}
\begin{proof}
By the Coarea formula for Lipschitz functions and the assumption
$|[0<|\phi|<\delta,\ |\nabla\phi|=0]|=0$, we obtain
\[
|[0<|\phi|<\varepsilon]|
=
\int_{-\varepsilon}^{\varepsilon}
\int_{[\phi=t]}
\frac{1}{|\nabla\phi|}
\,d\mathcal H^{n-1}\,dt .
\]
Using \eqref{eq_levelsetest}, we infer
\[
|[0<|\bar\phi|<\varepsilon]|
\le
\int_{-\varepsilon}^{\varepsilon}
C |t|^{\gamma-1}\,dt
=
2C_0\int_0^\varepsilon t^{\gamma-1}\,dt
=
\frac{2C}{\gamma}\varepsilon^\gamma .
\]
If $|[\bar\phi=0]|=0$, the same estimate holds for
$[|\bar\phi|\le\varepsilon]$.
\end{proof}

\begin{lemma}
Let $\phi\in W^{1,\infty}(\Omega)$. Set $A:=[\phi=0]$ and assume that $A$ is a finite union of compact $C^1$--hypersurfaces. Let $\gamma\in(0,1]$.
Suppose that there exist constants $c>0$ and $\rho>0$ such that
\begin{equation}\label{eq_growth}
|\phi(x)|\ge c\,\operatorname{dist}(x,A)^{1/\gamma}
\end{equation}
for all $x$ with $\operatorname{dist}(x,A)<\rho$.
Assume further that
\begin{equation}\label{eq_inf}
    \inf_{[\operatorname{dist}(\cdot,A)\geq \rho]} \vert \phi \vert>0
\end{equation}
Then there exist constants $C>0$ and $\varepsilon_0>0$ such that for all $0<\varepsilon<\varepsilon_0$
\[
|[|\phi|\le\varepsilon]|
\le C\varepsilon^\gamma.
\]
\end{lemma}

\begin{proof}
Since $A$ is a finite union of compact $C^1$--hypersurfaces, 
$\mathcal H^{n-1}(A)<\infty$.
If $|\phi|\le\varepsilon$ on $[\operatorname{dist}(\cdot,A)<\rho]$, then \eqref{eq_growth} gives $c\operatorname{dist}(x,A)^{1/\gamma}\le \varepsilon$.
Thus  $\operatorname{dist}(x,A)\le c^{-\gamma}\varepsilon^\gamma$ and with \eqref{eq_inf}, $[|\phi|\le\varepsilon]
\subset
[\operatorname{dist}(\cdot,A)\le c^{-\gamma}\varepsilon^\gamma]$.
Since $A$ is a finite union of compact $C^1$--hypersurfaces, it is a
closed $(n-1)$-rectifiable subset of $\mathbb R^n$. By \cite[Theorem 3.2.39]{zbMATH03280855}, $
\mathcal M^{n-1}(A)=\mathcal H^{n-1}(A)$, where $\mathcal M$ denotes the $(n-1)$-dimensional Minkowski content of $A$ which is
\[
\lim_{r\to0_+}
\frac{
\mathcal L^n(\{x\in\mathbb R^n:\operatorname{dist}(x,A)<r\})
}
{2r}
=
\mathcal H^{n-1}(A).
\]
From this we infer the existence of constants $C_A>0$ and $r_A>0$ such that for $0<r<r_A$
\[
\mathcal L^n(\{x\in\mathbb R^n:\operatorname{dist}(x,A)<r\})
\le C_A r.
\]
After decreasing $r_A$ if necessary, we obtain for $0<r<r_A$ that $\mathcal L^n(\{x\in \Omega:\operatorname{dist}(x,A)\le r\})
\le C_A r$.
Choosing $\varepsilon_0$ such that $c^{-\gamma} \varepsilon_0^\gamma<\min\{\rho,r_A\}$ completes the proof.
\end{proof}

From the previous two lemmas we infer that, when the adjoint velocity field $\bar{\mathbf{z}}$ satisfies one of the above conditions for each of its components, then the growth \eqref{eq:firstvargrowth} holds.
Finally, we mention pure second-order conditions that imply a growth as in Theorem \ref{strictilocalopt}. 
Results in this direction for $\gamma=1$ and dimensions $n=2, 3$ were obtained in \cite[Theorem 6.4]{zbMATH06913572} and \cite[Section 5.1]{zbMATH07875626} and were further extended to arbitrary dimensions ($n\in \mathbb N$), in \cite{2026arXiv260214632W}.

\bibliographystyle{siam}
\bibliography{references}

\end{document}